\documentclass[leqno]{article}

\usepackage[frenchb,english]{babel}
\usepackage[utf8]{inputenc}
\usepackage{amsmath}
\usepackage{amssymb}
\usepackage{amsfonts}
\usepackage{enumerate}
\usepackage{vmargin}
\usepackage[all]{xy}
\usepackage{mathrsfs}
\usepackage{mathtools}
\usepackage{lmodern}
\usepackage{slashed}
\usepackage[colorlinks=true,linkcolor=blue,pagebackref=true]{hyperref}%
\setmarginsrb{3cm}{3cm}{3.5cm}{3cm}{0cm}{0cm}{1.5cm}{3cm}

\newcommand{\Cbb}{\ensuremath{\mathbb{C}}}
\newcommand{\C}{\mathrm{C}}
\newcommand{\E}{\ensuremath{\mathbb{E}}}
\newcommand{\N}{\ensuremath{\mathbb{N}}}
\newcommand{\B}{\mathrm{B}} 
\let\H\relax 
\newcommand{\H}{\mathrm{H}}
\newcommand{\I}{\mathrm{I}} 

\let\L\relax 
\newcommand{\L}{\mathrm{L}}

\newcommand{\scr}{\mathscr}

\newcommand{\M}{\mathrm{M}}

\newcommand{\CC}{\mathrm{CC}}

\let\cal\relax
\newcommand{\cal}{\mathcal}

\newcommand{\R}{\ensuremath{\mathbb{R}}}
\newcommand{\T}{\ensuremath{\mathbb{T}}}

\newcommand{\W}{\mathrm{W}}
\let\S\relax 
\newcommand{\S}{\mathrm{S}} 

\newcommand{\dist}{\mathrm{dist}} 
\newcommand{\Id}{\mathrm{Id}}

\newcommand{\Reg}{\mathrm{Reg}}
\newcommand{\la}{\langle}
\newcommand{\ra}{\rangle}

\newcommand{\MK}{\mathrm{MK}} 

\renewcommand{\leq}{\ensuremath{\leqslant}}
\renewcommand{\geq}{\ensuremath{\geqslant}}
\newcommand{\qed}{\hfill \vrule height6pt  width6pt depth0pt}
\newcommand{\bnorm}[1]{ \big\| #1  \big\|}

\newcommand{\norm}[1]{\left\Vert#1\right\Vert}
\newcommand{\mk}{\medskip}
\newcommand{\xra}{\xrightarrow}
\newcommand{\co}{\colon}

\newcommand{\otpb}{\hat{\ot}}
\newcommand{\ot}{\otimes}
\newcommand{\ovl}{\overline}

\DeclareMathOperator{\sgn}{\mathrm{sgn}} 

\newcommand{\dsp}{\displaystyle}
\let\i\relax 
\newcommand{\i}{\mathrm{i}}

\newcommand{\ov}{\overset}
\newcommand{\sa}{\mathrm{sa}}

\renewcommand{\d}{\mathop{}\mathopen{}\mathrm{d}} 
\newcommand{\e}{\mathrm{e}} 
\renewcommand{\d}{\mathop{}\mathopen{}\mathrm{d}}

\DeclareMathOperator{\Span}{span} 
\DeclareMathOperator{\Lip}{\mathrm{Lip}} 
\DeclareMathOperator{\supp}{supp} 
\DeclareMathOperator{\tr}{Tr} 
\let\ker\relax 
\DeclareMathOperator{\ker}{Ker} 
\DeclareMathOperator{\Ran}{Ran} 

\DeclareMathOperator{\dom}{dom} 
\let\Re\relax 
\DeclareMathOperator{\Re}{Re} 
\DeclareMathOperator*{\esssup}{esssup} 
\newtheorem{thm}{Theorem}[section]
\newtheorem{defi}[thm]{Definition}
\newtheorem{prop}[thm]{Proposition}
\newtheorem{conj}[thm]{Conjecture}

\newtheorem{cor}[thm]{Corollary}
\newtheorem{lemma}[thm]{Lemma}

\newtheorem{remark}[thm]{Remark}
\newtheorem{example}[thm]{Example}

\newenvironment{proof}[1][]{\noindent {\it Proof #1} : }{\hbox{~}\qed
\smallskip
}

\usepackage{tocloft}
\numberwithin{equation}{section}
\usepackage[nottoc,notlot,notlof]{tocbibind}

\let\OLDthebibliography\thebibliography
\renewcommand\thebibliography[1]{
  \OLDthebibliography{#1}
  \setlength{\parskip}{0pt}
  \setlength{\itemsep}{0pt plus 0.3ex}
}

\newcommand\reallywidehat[1]{\arraycolsep=0pt\relax%
\begin{array}{c}
\stretchto{
  \scaleto{
    \scalerel*[\widthof{\ensuremath{#1}}]{\kern-.5pt\bigwedge\kern-.5pt}
    {\rule[-\textheight/2]{1ex}{\textheight}} 
  }{\textheight} %
}{0.5ex}\\           
#1\\                 
\rule{-1ex}{0ex}
\end{array}
}

\begin{document}
\selectlanguage{english}
\title{\bfseries{Sobolev algebras on Lie groups and noncommutative geometry}}
\date{}
\author{\bfseries{C\'edric Arhancet}}

\maketitle


\begin{abstract}
We show that there exists a quantum compact metric space which underlies the setting of each Sobolev algebra associated to a subelliptic Laplacian $\Delta=-(X_1^2+\cdots+X_m^2)$ on a compact connected Lie group $G$ if $p$ is large enough, more precisely under the (sharp) condition $p > \frac{d}{\alpha}$ where $d$ is the local dimension of $(G,X)$ and where $0 < \alpha \leq 1$. We also provide locally compact variants of this result and generalizations for real second order subelliptic operators. We also introduce a compact spectral triple (=noncommutative manifold) canonically associated to each subelliptic Laplacian on a compact group. In addition, we show that its spectral dimension is equal to the local dimension of $(G,X)$. Finally, we prove that the Connes spectral pseudo-metric allows us to recover the Carnot-Carath\'eodory distance. 
\end{abstract}


\makeatletter
 \renewcommand{\@makefntext}[1]{#1}
 \makeatother

\tableofcontents

\section{Introduction}

Suppose $1<p<\infty$. If $\Delta_p \co \dom\Delta_p \subset \L^p(\R^n) \to \L^p(\R^n)$ is the (positive) Laplacian and if $\alpha \in \R$, we can consider the fractional powers $\Delta_p^{\frac{\alpha}{2}}$ and $(\Id+\Delta_p)^{\frac{\alpha}{2}}$. If $\alpha<0$, these operators are the Riesz potential and the Bessel potential of order $-\alpha$. The last one was independently introduced by Aronszajn and Smith \cite{ArS61} and Calder\'on \cite{Cal61}, which is nowadays a classical notion in harmonic analysis, see \cite[p.~131]{Ste70} and \cite[Definition 1.2.4 p.~13]{Gra14}. 
 
Strichartz proved in \cite[Theorem 2.1, Chap.~2]{Str67} that the Bessel potential space $\L^p_\alpha(\R^n)\ov{\mathrm{def}}{=}\big\{f \in \L^p(\R^n) :  \text{ there exists } g \in \L^p(\R^n) \text{ such that }f=(\Id+\Delta_p)^{-\frac{\alpha}{2}}g \big\}$ is an algebra for the pointwise product for any $1 < p < \infty$ and any $\alpha > 0$ such that $\alpha p > n$. Note that by \cite[Theorem 12.3.4 p.~301]{MCSA01}, we have
$$
\L^p_\alpha(\R^n) 
= \dom \Delta_p^{\frac{\alpha}{2}}.
$$
Indeed, for any $1 < p < \infty$ and any $\alpha > 0$, Kato and Ponce showed in their work \cite[Lemma X.4 p.~906]{KaP88} on Navier-Stokes equations, that $\L^p_\alpha(\R^n) \cap \L^\infty(\R^n)$ is an algebra for the pointwise product (see also \cite[Theorem 2.2.12 p.~81]{GPH15}). 
This is a stronger result since by the Sobolev embedding theorem \cite[Theorem 1.2.4 (c) p.~14]{AdH96} 
we have a continuous inclusion $\L^p_\alpha(\R^n) \subset \L^\infty(\R^n)$ if $\alpha p > n$. The proof is a simple consequence of the inequality
\begin{equation}
\label{Leibniz-rules-Rn-1}
\norm{fg}_{\L^p_\alpha(\R^n)} 
\lesssim_{\alpha,p} \norm{f}_{\L^p_\alpha(\R^n)} \norm{g}_{\L^\infty(\R^n)} + \norm{f}_{\L^\infty(\R^n)} \norm{g}_{\L^p_\alpha(\R^n)}
\end{equation}
for any $f,g \in \L^p_\alpha(\R^n) \cap \L^\infty(\R^n)$, where we use the graph norm of the closed operator $\Delta_p^{\frac{\alpha}{2}}$
$$
\norm{f}_{\L^p_\alpha(\R^n)} \ov{\mathrm{def}}{=} \norm{f}_{\L^p(\R^n)} + \bnorm{\Delta_p^{\frac{\alpha}{2}}(f)}_{\L^p(\R^n)}
\approx \bnorm{(\Id+\Delta_p)^{\frac{\alpha}{2}}(f)}_{\L^p(\R^n)}.
$$ 
The motivation of this result was the estimate of $\norm{(\Id+\Delta_p)^{\frac{\alpha}{2}}(fg)-f(\Id+\Delta_p)^{\frac{\alpha}{2}}(g)}_{\L^p(\R^n)}$ for any Schwartz functions $f$ and $g$. This commutator estimate is needed in the study of some nonlinear partial differential equations. We refer to \cite{Li19} and references therein for a comprehensive view of the state of the art in this kind of inequalities and to \cite[Section 12.3]{MCSA01} for several equivalent definitions of the Banach space $\L^p_\alpha(\R^n)$.

In 1996, in their study of Schr\"odinger semigroups, Gulisashvili and Kon considered in \cite{GuK96} the homogeneous Sobolev space $\dot\L^p_\alpha(\R^n)$ which is the completion of the space $\dom \Delta_p^{\frac{\alpha}{2}}$ with respect to the norm
\begin{equation}
\label{norm-Lp-dot}
\norm{f}_{\dot{\L}^p_\alpha(\R^n)} \ov{\mathrm{def}}{=} \bnorm{ \Delta_p^{\frac{\alpha}{2}}(f)}_{\L^p(\R^n)}.
\end{equation}
Note that there exist several definitions of this abstract space. We refer to \cite{MPS20} for more information. In this paper, we only use functions of this space belonging to $\dom \Delta_p^{\frac{\alpha}{2}}$. If $\alpha >0$, Gulisashvili and Kon observed that $\dot\L^p_\alpha(\R^n)\cap \L^\infty(\R^n)$ is also an algebra for the pointwise product. This result is again a consequence of the Leibniz's rule \cite[Theorem 1.4]{GuK96}
\begin{equation}
\label{Leibniz-rules-Rn}
\norm{fg}_{\dot{\L}^p_\alpha(\R^n)} 
\lesssim_{\alpha,p} \norm{f}_{\dot{\L}^p_\alpha(\R^n)} \norm{g}_{\L^\infty(\R^n)} + \norm{f}_{\L^\infty(\R^n)} \norm{g}_{\dot{\L}^p_\alpha(\R^n)}.
\end{equation}
Coulhon, Russ and Tardivel-Nachef \cite{CRT01} extended this result to the case of a unimodular connected Lie group $G$ with polynomial volume growth by replacing the Laplacian by a subelliptic Laplacian $-X_1^2- \cdots -X_m^2$ where $X\ov{\mathrm{def}}{=} (X_1, \ldots, X_m)$ is a family of left-invariant H\"ormander vector fields. Replacing $\R^n$ by $G$, they obtain generalizations 
\begin{equation}
\label{Leibniz-Coulhon-bis}
\norm{fg}_{\L^p_\alpha(G)} 
\lesssim_{\alpha,p} \norm{f}_{\L^p_\alpha(G)} \norm{g}_{\L^\infty(G)} + \norm{f}_{\L^\infty(G)} \norm{g}_{\L^p_\alpha(G)}, \quad \alpha>0
\end{equation} 
for any $f,g \in \L^p_\alpha(G) \cap \L^\infty(G)$ and
\begin{equation}
\label{Leibniz-Coulhon}
\norm{fg}_{\dot{\L}^p_\alpha(G)} 
\lesssim_{\alpha,p} \norm{f}_{\dot{\L}^p_\alpha(G)} \norm{g}_{\L^\infty(G)} + \norm{f}_{\L^\infty(G)} \norm{g}_{\dot{\L}^p_\alpha(G)}, \quad 0<\alpha \leq 1
\end{equation}
($\alpha>0$ if $G$ is nilpotent) of the Leibniz's rules \eqref{Leibniz-rules-Rn-1} and \eqref{Leibniz-rules-Rn}. Furthermore, they obtained algebras $\L^p_\alpha(G) \cap \L^\infty(G)$ and $\dot{\L}^p_\alpha(G) \cap \L^\infty(G)$ called Sobolev algebras. Note that Bohnke has previously proved in \cite[Th\'eor\`eme 1]{Bon85} that $\L^p_\alpha(G)$ is an algebra for the pointwise product if $G$ is a stratified Lie group and if $\alpha p >d$ where $d$ is the local dimension of the group. With the help of Sobolev embedding theorem $\L^p_\alpha(G) \subset \L^\infty(G)$ under the condition $\alpha p >d$, we can see that this is a particular case of the results of \cite{CRT01}. See \cite{BPTV19} and \cite{PeV18} for the more complicated case of non-unimodular Lie groups.

The concept of quantum compact metric space has its origins in Connes' paper \cite{Con89} of 1989, in which he shows that we can recover the geodesic distance $\dist$ of a compact oriented Riemannian spin manifold $M$ using the Dirac operator $\scr{D}$, by the formula
\begin{equation}
\label{Connes-for-1}
\dist(x,y)
=\sup_{f \in \mathrm{C}(M), \norm{[\scr{D},f]} \leq 1} |f(x)-f(y)|, \quad x,y \in M
\end{equation}
where the supremum is taken on all the continuous functions such that the commutator $[\scr{D},f]\ov{\mathrm{def}}{=} \scr{D}f-f\scr{D}$ extends to a contractive operator.  Recall that $\scr{D}$ is an unbounded operator acting on the Hilbert space of $\L^2$-spinors and that the functions of $\mathrm{C}(M)$ act on the same Hilbert space by multiplication operators. Indeed, it is well-known that the commutator $[\scr{D},f]$ induces a bounded operator if and only if $f$ is a Lipschitz function and in this case the Lipschitz constant of $f$ is equal to the norm $\norm{[\scr{D},f]}$. Moreover, this space of functions is norm dense in the space $\C(M)$ of continuous functions. See \cite[Chapter 6]{Con94} for more information and we refer to \cite{Var1} for a complete proof. If we identify the points $x,y$ as pure states $\omega_x$ and $\omega_y$ on the unital $\C^*$-algebra $\C(M)$, we can see this formula as 
$$
\dist(\omega_x,\omega_y)
=\sup_{f \in \C(M), \norm{[\scr{D},f]} \leq 1} |\omega_x(f)-\omega_y(f)|, \quad x,y \in M.
$$

After many years, Rieffel \cite{Rie04} axiomatized this formula replacing the algebra $\mathrm{C}(M)$ by a unital $\C^*$-algebra $A$, $f \mapsto \norm{[\mathscr{D},f]}$ by a seminorm $\norm{\cdot}$ defined on a dense subspace of $A$ and $\omega_x,\omega_y$ by arbitrary states of $A$ obtaining essentially the formula \eqref{MongeKant-def} below and giving rise to a theory of \textit{quantum} compact metric spaces. With this notion, Rieffel was able to define a quantum analogue of Gromov-Hausdorff distance and to give a meaning to many approximations found in the physics literature, as the case of matrix algebras converging to a sphere. Moreover, as research in noncommutative metric geometry progressed, some additional conditions are often added as Leibniz's rules
$$
\norm{ab} 
\lesssim \norm{a}\norm{b}_{A} +\norm{a}_A\norm{b}, \quad a, b \in \dom \norm{\cdot}
$$ 
which reminds \eqref{Leibniz-rules-Rn-1} and \eqref{Leibniz-rules-Rn}. 

If $G$ is a compact connected Lie group $G$ and if $\alpha p >d$, we show that $\big(\C(G),\norm{\cdot}_{\dot{\L}^p_\alpha(G)}\big)$ is a quantum compact metric space which underlies the setting of each Leibniz's rule \eqref{Leibniz-Coulhon} associated to a subelliptic Laplacian $\Delta$ on a compact connected Lie group $G$. Here $\C(G)$ is the algebra of continuous function on $G$. We also provide locally compact variants for non-compact groups in Section \ref{Sec-locally} with the help of seminorms $\norm{\cdot}_{\L^p_\alpha(G)}$. Finally, our approach is flexible and should be adaptable to different contexts. The present work  gets its inspiration from the papers \cite{Rie02}, \cite{JuM10} and \cite{ArK22}.



%

Spectral triples are generalizations of the setting of Hodge-Dirac operators and  Dirac operators on compact oriented Riemannian (spin) manifolds. This notion has emerged as a mean to encode geometric information of spaces in operator and spectral theory. It is at the heart of noncommutative geometry and used to describe quantum spaces providing efficient tools for such an analysis. Remarkably, this notion also provide a framework for the study of classical spaces as fractals (e.g. \cite{CGIS14}) or orbifolds. We refer to \cite{CoM08} for an extensive list of examples.

We introduce a spectral triple associated to each subelliptic Laplacian on a compact connected Lie group $G$ and we show that the spectral dimension is equal to the local dimension of $(G,X)$ where $X\ov{\mathrm{def}}{=} (X_1, \ldots, X_m)$ is the family of left-invariant H\"ormander vector fields which defines the subelliptic Laplacian \eqref{def-de-Delta}. In retrospect, our proof of this computation is quite simple. However, a variant shows a link between what we call the \textit{local Coulhon-Varopoulos dimension} of a suitable (symmetric) markovian semigroup $(T_t)_{t \geq 0}$ acting on $\L^\infty(\Omega)$ where $\Omega$ is a finite measure space and the spectral dimension of the spectral triple defined by an associated canonical Hodge-Dirac operator. We will investigate this more general setting in a future publication \cite{Arh22}, also providing generalizations to suitable markovian semigroups acting on von Neumann algebras. 

Recall that this dimension is defined as the infimum of positive real numbers $d$ such that 
\begin{equation}
\label{Varo-Coulhon-dim}
\norm{T_t}_{\L^1(\Omega) \to \L^\infty(\Omega)} 
\lesssim \frac{1}{t^{\frac{d}{2}}}, \quad 0 < t \leq 1,
\end{equation}
see \cite{Cou90} and \cite[p.~187]{CoS91} (see also \cite{CKS87} for a related work). Note that the terminology \textit{ultracontractivity} is equally used in \cite[Section 7.3.2]{Are04} and in \cite{AtE17}. This notion is also referred in \cite{Xio17} under the more suitable term \textit{local ultracontractivity}. We warn the reader that really different definitions of ultracontractivity coexist in the literature, e.g. see \cite{CoM93}, \cite{CoM03}, \cite[p.~89]{Dav89} and \cite{GrT12}. In the case of a connected Lie group $G$ equipped with a family $X$ of left-invariant H\"ormander vector fields, the inequality \eqref{Varo-Coulhon-dim} is satisfied for the heat semigroup whose generator is the opposite $-\Delta$ of the subelliptic Laplacian and the local dimension $d$ of $(G,X)$.



\paragraph{Structure of the paper}
The paper is organized as follows. Section \ref{sec-preliminaries} gives background on operator theory. The aim of Section \ref{sec-Background-Lie} is to describe our setting related to Lie groups and to prove some preliminary useful results. In Section \ref{Sec-quantum-compact-matric-spaces}, we show the existence of our quantum compact metric spaces. Section \ref{Sec-locally} is devoted to give \textit{locally} compact variants of these quantum compact metric spaces. In Section \ref{Sec-spectral-triples}, we introduce new compact spectral triples and we describe some properties. In particular we compute the spectral dimension. In Section \ref{Section-distances}, we investigate the links between the Connes spectral pseudo-distance and the Carnot-Caratheodory distance but also with the intrinsic pseudo-distance associated to some Dirichlet form. We prove an analogue of formula \eqref{Connes-for-1} for the Carnot-Caratheodory distance. Finally, we state in Section \ref{Sec-open-problems} two natural conjectures on the functional calculus of subelliptic Laplacians and their associated Hodge-Dirac operators.


\section{Preliminaries on operator theory}
\label{sec-preliminaries}

\paragraph{Minkowski's inequality} Suppose that $(\Omega_1,\mu_1)$ and $(\Omega_2,\mu_2)$ are two $\sigma$-finite measure spaces and consider a measurable function $f \co \Omega_1 \times \Omega_2 \to \mathbb{C}$. We will use the following classical inequality \cite[A.1 p.~271]{Ste70}
\begin{equation}
\label{Minkowski-Lp}
\left[\int_{\Omega_2} \left|\int_{\Omega_1} f(x,y) \d\mu_1(x)\right|^p \d\mu_2(y)\right]^{\frac{1}{p}} 
\leq \int_{\Omega_1} \left(\int_{\Omega_2} |f(x,y)|^p \d\mu_2(y)\right)^{\frac{1}{p}}\d\mu_1(x).
\end{equation}

\paragraph{Dunford-Pettis theorem}  
Let $\Omega$ be a $\sigma$-finite measure space such that $\L^1(\Omega)$ is separable. A particular case of Dunford-Pettis theorem, e.g. \cite[p.~528]{Rob91} \cite[Section 3]{GrT12}
, says that if $T \co \L^1(\Omega) \to \L^\infty(\Omega)$ is a bounded operator then there exists a function $K \in \L^\infty(\Omega \times \Omega)$ such that for any $f \in \L^1(\Omega)$ we have $(Tf)(x) = \int_\Omega K(x,y)f(y) \d y$ for almost all $x \in \Omega$. Moreover, we have
\begin{equation}
\label{Dunford-Pettis-1}
\norm{T}_{\L^1(\Omega) \to \L^\infty(\Omega)}
=\norm{K}_{\L^\infty(\Omega \times \Omega)}.
\end{equation}
Conversely, such a function $K$ defines a bounded operator $T \co \L^1(\Omega) \to \L^\infty(\Omega)$. We also have a similar result for a bounded operator $T \co \L^2(\Omega) \to \L^\infty(\Omega)$. In this case, the equality \eqref{Dunford-Pettis-1} is replaced by
\begin{equation}
\label{Dunford-Pettis-2}
\norm{T}_{\L^2(\Omega) \to \L^\infty(\Omega)}
=\norm{K}_{\L^\infty(\Omega,\L^2(\Omega))}.
\end{equation}


\paragraph{Operator theory} 

Recall the characterization of the domain of the closure $\ovl{T}$ of a closable unbounded operator $T \co \dom T \subset Y \to Z$ between Banach spaces $Y$ and $Z$. We have 
\begin{equation}
\label{Domain-closure}
x \in \dom \ovl{T} \text{ iff there exists } (x_n) \subset \dom T \text{ such that } x_n \to x
\text{ and } T(x_n) \to y \text{ for some } y.  
\end{equation}
The following is \cite[Corollary 5.6 p.~144]{Tha92}.

\begin{thm}
\label{Th-unit-eq}
Let $T$ be a closed densely defined operator on a Hilbert space $H$. Then the operator $T^*T$ on $(\ker T)^\perp$ is unitarily equivalent to the operator $TT^*$ on $(\ker T^*)^\perp$.
\end{thm}
If $T$ is densely defined, by \cite[Problem 5.27 p.~168]{Kat76}, we have
\begin{equation}
\label{lien-ker-image}
\ker T^*
=(\Ran T)^\perp.
\end{equation}
We will also use the following classical equalities \cite[2.8.45 p.~171]{KaR97}
\begin{equation}
\label{Kadison-image}
\Ran{T^*T}
=\Ran{T^*}
\quad \text{and} \quad 
\ker T^*T
=\ker T. 
\end{equation}
If $A$ is sectorial operator on a \textit{reflexive} Banach space $Y$, we have by \cite[Proposition 2.1.1 (h)]{Haa06} a decomposition
\begin{equation}
\label{decompo-reflexive}
Y
=\ker A \oplus \ovl{\Ran A}.
\end{equation}

\paragraph{Fractional powers}
See \cite{Haa06} and \cite{KuW04} for more information on fractional powers. Let $A$ be a sectorial operator on a Banach space $Y$. If $A$ is densely defined and if $\alpha$ is a complex number with $0 < \Re \alpha <n$, where $n$ is an integer, then the space $\dom A^n$ is a core of $A^\alpha$ by \cite[p.~62]{Haa06}, i.e. $\dom A^n$ is dense in $\dom A^\alpha$ for the graph norm of $A^\alpha$, and we have
\begin{equation}
\label{Balakrishnan}
A^\alpha(x)
=\frac{\Gamma(n)}{\Gamma(\alpha)\Gamma(n-\alpha)} \int_{0}^{\infty} t^{\alpha-1} \big(A(t+A)^{-1}\big)^n x \d t
, \quad x \in \dom A^n.
\end{equation}
For any complex numbers $\alpha,\beta$ with $\Re \alpha, \Re \beta > 0$ we have $A^\alpha A^{\beta} = A^{\alpha+\beta}$. By \cite[p.~62]{Haa06} and \cite[Corollary 3.1.11]{Haa06}, for any $\alpha \in \Cbb$ with $\Re \alpha > 0$ we have
\begin{equation}
\label{inclusion-range}
\Ran A^\alpha 
\subset \ovl{\Ran A} 
\quad \text{and} \quad
\ker A^\alpha=\ker A.
\end{equation}
If $A$ is a sectorial operator on a Banach space $Y$ and if $\Re \alpha>0$, then by \cite[p.~137]{Ege15} the graph norms of the operators $A^\alpha$ and $(\Id+A)^{\alpha}$ are equivalent, i.e.
\begin{equation}
\label{equiv-Egert}
\norm{A^\alpha x}_Y+\norm{x}_Y
\approx \norm{(\Id+A)^{\alpha}x}_Y, \quad x \in \dom A^\alpha.
\end{equation}
The proof uses \cite[p.~28]{Ouh05}, the equality $\dom (\Id+A)^{\alpha}=\dom A^{\alpha}$ of \cite[Proposition 3.1.9 p.~65]{Haa06} and the boundedness of the operator $(\Id+A)^{-\alpha}$. See \cite[Lemma 15.22 p.~294]{KuW04} and \cite[Lemma 6.3.2 p.~148]{Haa06} for the particular case where $A$ is injective.


\paragraph{Compactness of fractional powers} Let $\Omega$ be a finite measure space. Consider a weak* continuous semigroup $(T_t)_{t \geq 0}$ of selfadjoint positive unital contractions on $\L^\infty(\Omega)$ with weak* (negative) generator $A_\infty$. A classical argument 
shows that each operator $T_t$ is integral preserving. Such a semigroup induces a strongly continuous semigroup $(T_{t,p})_{t \geq 0}$ on $\L^p(\Omega)$ and its generator $A_p$ is sectorial if $1<p<\infty$.

There exists a weak* continuous conditional expectation $\E \co \L^\infty(\Omega) \to \L^\infty(\Omega)$ on the fixed subalgebra $\{f \in \L^\infty(\Omega) : T_t(f)=f \text{ for all } t \geq 0\}$. This subset is equal to $\ker A_\infty$. We sketch the argument.  By \cite[Proposition 3.1.4 p.~120]{Eme07}, the induced semigroup $(T_{t,1})_{t \geq 0}$ on the space $\L^1(\Omega)$ is mean ergodic. In particular, we have a bounded projection $Q \co \L^1(\Omega) \to \L^1(\Omega)$ onto $\ker A_1$ along $\ovl{\Ran A_1}$ which is clearly contractive, satisfying $Q(1)=1$. By \cite[Corollary 5.52 p.~222]{AbA02}, $Q$ is a conditional expectation. We conclude by duality that the suitable conditional expectation $\E$ exists, see \cite[Exercise 9 p.~159]{EFHN15}.

If $\{f \in \L^\infty(\Omega) : T_t(f)=f \text{ for all } t \geq 0\}=\Cbb 1$, the condition expectation is given by $\E(f)=\big(\int_\Omega f\big)1$. We use the notation $\L_0^p(\Omega)$ for the subspace $\ker \E_p$ of $\L^p(\Omega)$. It is the space of functions with mean 0. We have $\L_0^p(\Omega)=\ovl{\Ran A_p}$. Finally, for $1 \leq p \leq q \leq \infty$, consider the property
\begin{equation}
\label{Rnpq}
\norm{T_t}_{\L^p_0(\Omega) \to \L^q_0(\Omega)}
\lesssim \frac{1}{t^{\frac{n}{2}(\frac{1}{p}-\frac{1}{q})}}, \quad 0 <t\leq 1,
\end{equation}
which is a \textit{local} version of the property \cite[$(\mathrm{R}_n^{pq})$ p.~619]{JuM10}
$$
\norm{T_t}_{\L^p_0(\Omega) \to \L^q_0(\Omega)}
\lesssim \frac{1}{t^{\frac{n}{2}(\frac{1}{p}-\frac{1}{q})}}, \quad t>0.
$$
By an interpolation argument similar to the one of \cite[Lemma 1.1.2]{JuM10}, each of these properties holds for one pair $1 \leq p < q \leq \infty$ if and only if it holds for all $1 \leq p \leq q \leq \infty$. See also \cite[Section 7.3.2 p.~65]{Are04} for a variant.

Recall the following result \cite[Theorem 9]{CwK95} (see also \cite[Theorem 5.5]{KMS03}) which allows to obtain compactness via complex interpolation.

\begin{thm}
\label{Th-interpolation-Kalton}
Suppose that $(X_0,X_1)$ and $(Y_0,Y_1)$ are Banach couples and that $X_0$ is a $\mathrm{UMD}$-space. Let $T \co X_0+X_1 \to Y_0+Y_1$ such that its restriction $T_0 \co X_0 \to Y_0$ is compact and such that $T_1 \co X_1 \to Y_1$ is bounded. Then for any $0 < \theta < 1$ the map $T \co (X_0,X_1)_\theta \to (Y_0,Y_1)_\theta$ is compact.
\end{thm}

The following is \cite[Theorem 1.1.7]{JuM10}. Note that the proof of this result uses \cite[Lemma 1.1.6]{JuM10} whose proof unfortunately seems false in light of the classical problem \cite[Problem 5.4]{KMS03}. However, \cite[Lemma 1.1.6]{JuM10} can be replaced by Theorem \ref{Th-interpolation-Kalton}. 

\begin{prop}
\label{prop-interpolation-1}
Let $\Omega$ be a finite measure space. Let $(T_t)_{t \geq 0}$ be a weak* continuous semigroup of selfadjoint positive contractions on $\L^\infty(\Omega)$ with $\{x \in \L^\infty(\Omega) : T_t(x)=x \text{ for all } t \geq 0\}=\Cbb 1$ satisfying $\norm{T_t}_{\L_0^1(\Omega) \to \L^\infty(\Omega)} \lesssim \frac{1}{t^{\frac{n}{2}}}$ for some $n$ and such that $A^{-1}$ is compact on $\L_0^2(\Omega)$. Then for all $1 \leq p < q \leq \infty$ such that $\frac{2\Re z}{n} > \frac{1}{p}-\frac{1}{q}$ the operator $A^{-z} \co \L_0^p(\Omega) \to \L_0^q(\Omega)$ is compact.
\end{prop}

\section{Background and preliminaries results on Lie groups}
\label{sec-Background-Lie}

\paragraph{Convolution}
If $G$ is a \textsl{unimodular} locally compact group equipped with a Haar measure $\mu_G$,  recall that the convolution product of two functions $f$ and $g$ is given, when it exists, by
\begin{equation}
\label{Convolution-formulas}
(f*g)(s)
\ov{\mathrm{def}}{=} \int_G f(r)g(r^{-1}s) \d\mu_G(r)
=\int_G f(sr^{-1})g(r) \d\mu_G(r).
\end{equation}

\paragraph{Carnot-Carath\'eodory distances on connected Lie groups} Let $G$ be a connected Lie group with neutral element $e$. We consider a finite sequence $X \ov{\mathrm{def}}{=}(X_1,\ldots,X_m)$ of left invariant vector fields which generate the Lie algebra $\frak{g}$ of the group $G$ such that the vectors $X_1(e),\ldots, X_m(e)$ are linearly independent. We say that it is a family of left-invariant H\"ormander vector fields.


Let $\gamma \co [c,d] \to G$ be an absolutely continuous 
path such that $\dot\gamma(t)$ belongs to the subspace $\Span \{ X_1|_{\gamma(t)}, \ldots, X_m|_{\gamma(t)} \}$ for almost all $t \in [c,d]$. If $\dot\gamma(t) = \sum_{k=1}^m \dot\gamma_k(t) \, X_k|_{\gamma(t)}$ for almost all $t \in [c,d]$, where $\dot\gamma_k(t) \in \R$ and where each $\dot\gamma_k$ is measurable, we can define the $p$-length of $\gamma$ by 
\begin{equation}
\label{Def-lpgamma}
\ell_p(\gamma) 
\ov{\mathrm{def}}{=} \int_c^d  \Big( \sum_{k=1}^m |\dot\gamma_k(t)|^p \Big)^{\frac{1}{p}} \d t
\end{equation}
which belongs to $[0,\infty]$. For any $s,s' \in G$ there exists such a path $\gamma \co [0,1] \to G$ with finite length with $\gamma(0) = s$ and $\gamma(1) = s'$. If $s,s' \in G$ and $1<p<\infty$ then we define the real number $\dist_\CC^p(s,s')$ between $s$ and $s'$ to be the infimum of the length of all such paths with $\gamma(0) = s$ and $\gamma(1) = s'$:
\begin{equation}
\label{distance-Carnot}
\dist_\CC^p(s,s')
\ov{\mathrm{def}}{=} \inf_{\gamma(0)=s,\gamma(1)=s'} \ell_p(\gamma).
\end{equation}
See \cite[p.~39]{VSCC92} and \cite[p.~22]{DtER03} if $p=2$. In this case, we recover the Carnot-Carath\'eodory distance $\dist_\CC(s,s')$. We refer also to \cite{Nag05}.

If $f \co G \to \mathbb{C}$ is a smooth function and if $\gamma \co [0,1] \to G$ is an absolutely continuous path with tangents in the subspace spanned by $X_1,\ldots,X_m$ then by \cite[p.~64]{ElR01} we have
\begin{equation}
\label{Chain-rule}
\frac{\d }{\d t}f(\gamma(t))
=\sum_{k=1}^{m} \dot\gamma_k(t)(X_k f)(\gamma(t)) \quad \text{a.e. $t \in [0,1]$}
\end{equation}
i.e.
\begin{equation*}
\frac{\d }{\d t}f(\gamma(t))
=\big\langle \dot\gamma(t), (X f)(\gamma(t)) \big\rangle \quad \text{a.e. $t \in [0,1]$}.
\end{equation*}

We will need the following elementary inequality. In the following statement, each $\dom X_{k,p}$ is the domain of $X_k$ on $\L^p(G)$, i.e. $X_{k,p} \co \dom X_{k,p} \subset \L^p(G) \to \L^p(G)$. Recall that connected Lie groups are $\sigma$-finite under Haar measure. Let $\lambda \co G \to \B(\L^p(G))$, $s \mapsto (f \mapsto f(s^{-1} \cdot))$ be the left regular representation of $G$.

The following result is a variant of \cite[Lemma VIII.1.1 p.~106]{VSCC92}. 

\begin{lemma}
\label{Lemma-useful-estimate}
Suppose $1<p<\infty$ and $\frac{1}{p}+\frac{1}{p^*}=1$. Then for any $s \in G$ and any $f$ belonging to $\dom X_{1,p} \cap \cdots \cap \dom X_{m,p}$, we have
\begin{equation}
\label{Ine-pratique}
\norm{(\Id-\lambda_s)f}_{\L^p(G)}
\leq \dist_\CC^{p^*}(s,e) \bigg(\sum_{k=1}^m \norm{X_{k,p}f}_{\L^p(G)}^p \bigg)^{\frac{1}{p}}.
\end{equation}
\end{lemma}

\begin{proof}
Let $f \in \C_c^\infty(G)$. Let $s \in G$ and let $\gamma \co [0,1] \mapsto G$ be an absolutely continuous path from $e$ to $s^{-1}$. For any $s' \in G$, we have
\begin{align*}
\label{}
\MoveEqLeft
((\Id-\lambda_s)f)(s')
=f(s')-(\lambda_sf)(s')
=f(s')-f(s^{-1}s')  \\    
&=-\int_{0}^{1} \frac{\d}{\d t}f(\gamma(t)s') \d t
\ov{\eqref{Chain-rule}}{=}-\int_{0}^{1} \sum_{k=1}^{m} \dot\gamma_k(t) (X_kf)(\gamma(t)s') \d t.
\end{align*} 
Consequently, using H\"older's inequality, we obtain
\begin{align}
\MoveEqLeft
\label{Divers-356}
\big|((\Id-\lambda_s)f)(s') \big|            
\leq \int_{0}^{1}  \bigg(\sum_{k=1}^{m} \dot\gamma_k(t)^p\bigg)^{\frac{1}{p}} \bigg(\sum_{k=1}^{m} \big[(X_k f)(\gamma(t)s')\big]^{p^*}\bigg)^{\frac{1}{p^*}} \d t.
\end{align}
Using Minkowski's inequality \eqref{Minkowski-Lp} and left invariance in the last equality, we deduce that 
\begin{align*}
\MoveEqLeft
\norm{(\Id-\lambda_s)f}_{p^*}
= \bigg[\int_G \big|((\Id-\lambda_s)f)(s') \big|^{p^*}  \d \mu_G(s')\bigg]^{\frac{1}{p^*}} \\
&\ov{\eqref{Divers-356}}{\leq} \bigg[\int_G \bigg(\int_{0}^{1} \bigg(\sum_{k=1}^{m} \dot\gamma_k(t)^p\bigg)^{\frac{1}{p}} \bigg(\sum_{k=1}^{m} \big[(X_k f)(\gamma(t)s')\big]^{p^*}\bigg)^{\frac{1}{p^*}} \d t\bigg)^{p^*}  \d \mu_G(s')\bigg]^{\frac{1}{p^*}}  \\
&\ov{\eqref{Minkowski-Lp}}{\leq} \int_{0}^{1} \bigg[\int_G  \bigg(\sum_{k=1}^{m} \dot\gamma_k(t)^p\bigg)^{\frac{p^*}{p}} \bigg(\sum_{k=1}^{m} \big[(X_k f)(\gamma(t)s')\big]^{p^*}\bigg) \d \mu_G(s')\bigg]^{\frac{1}{p^*}} \d t\\
&=\int_{0}^{1} \bigg(\sum_{k=1}^{m} \dot\gamma_k(t)^p\bigg)^{\frac{1}{p}} \bigg[\sum_{k=1}^{m}\int_G  \big[(X_k f)(\gamma(t)s')\big]^{p^*} \d \mu_G(s')\bigg]^{\frac{1}{p^*}} \d t \\
&= \int_{0}^{1} \bigg(\sum_{k=1}^{m} \dot\gamma_k(t)^p\bigg)^{\frac{1}{p}} \bigg[\sum_{k=1}^{m} \norm{X_kf}_{\L^{p^*}(G)}^{p^*} \bigg]^{\frac{1}{p^*}} \d t.
\end{align*}
Hence by taking the infimum over all possible paths, and observing that $\dist_\CC^{p^*}(e,s^{-1})=\dist_\CC^{p^*}(s,e)$ we obtain \eqref{Ine-pratique} with \eqref{distance-Carnot}. We conclude by using an approximation argument as in the proof of the next Proposition \ref{Prop-domain} for a general $f$.
\end{proof}

\paragraph{Growth of volume and dimensions}
Let $G$ be a connected Lie group equipped with a family $X \ov{\mathrm{def}}{=} (X_1, \ldots, X_m)$ of left-invariant H\"ormander vector fields and a left Haar measure $\mu_G$. For any $r > 0$ and any $x \in G$, we denote by $B(x,r)$ the open ball with respect to the Carnot-Carath\'eodory metric centered at $x$ and of radius $r$, and by $V(r) \ov{\mathrm{def}}{=} \mu_G(B(x,r))$ the Haar measure of any ball of radius $r$. It is well-known, e.g. \cite[p.~124]{VSCC92} that there exist $d \in \N^*$, $c,C > 0$ such that
\begin{equation}
\label{local-dim}
c r^d 
\leq V(r) 
\leq C r^d, \quad r \in ]0, 1[.
\end{equation}
The integer $d$ is called the local dimension of $(G,X)$.

When $r \geq 1$, only two situations may occur, independently of the choice of $X$ (see e.g. \cite[p.~26]{DtER03}): either $G$ has polynomial volume growth, which means that there exist $D \in \N$ and $c',C' > 0$ such that
\begin{equation}
\label{poly-growth}
c' r^D 
\leq V(r) 
\leq 
C' r^D, \quad r \geq 1
\end{equation}
or $G$ has exponential volume growth, which means that there exist $c_1,C_1, c_2,C_2 > 0$ such that
$$
c_1 \e^{c_2 r}
\leq  V(r)
\leq C_1 \e^{C_2 r}, \quad r \geq 1.
$$
When $G$ has polynomial volume growth, the integer $D$ in \eqref{poly-growth} is called the dimension at infinity of $G$. Note that, contrary to $d$, it only depends on $G$ and not on $X$, see \cite[Chapter 4]{VSCC92}. 

By \cite[II.4.5 p.~26]{DtER03} or \cite[p.~381]{Rob91}, each connected Lie group of polynomial growth is unimodular. By \cite[pp.~256-257]{Rob91} and \cite[p.~26]{DtER03}, a connected compact Lie group has polynomial volume growth with $D=0$. Recall finally that connected nilpotent Lie groups have polynomial volume growth by \cite[p.~28]{DtER03}.


\begin{example}
\normalfont
\label{Ex-dimension-torus}
The local dimension of the abelian compact Lie group $\T^n$ is $n$ by \cite[p.~274]{Rob91} and its dimension at infinity is of 0 since $\T^n$ is compact.
\end{example}

\begin{example}
\normalfont
\label{Ex-dimension-Stratified}
The local dimension and the dimension at infinity of a stratified Lie group are equal by \cite[II.4.15]{DtER03}. The three-dimensional Heisenberg group $\mathbb{H}_3$ (equipped with its canonical stratification) is a stratified group and its dimensions are equal to 4 by \cite[Example II.4.16]{DtER03}.
\end{example}

Let $G$ be a unimodular connected Lie group endowed with a family $(X_1, \ldots, X_m)$ of left-invariant H\"ormander vector fields and let $\mu_G$ be a Haar measure. We consider the subelliptic Laplacian $\Delta$ on $G$ defined by
\begin{equation}
\label{def-de-Delta}
\Delta 
\ov{\mathrm{def}}{=} -\sum_{k=1}^{m} X_k^2.
\end{equation}
For $1 \leq p < \infty$, let $\Delta_p \co \dom \Delta_p \subset \L^p(G)\to \L^p(G)$ be the smallest closed extension of the closable unbounded operator $\Delta|\C^\infty_c(G)$ to $\L^p(G)$. Note that the domain $\dom \Delta_p^{\frac{\alpha}{2}}$ is closed under the adjoint operation $f \mapsto \ovl{f}$. 

We denote by $(T_t)_{t \geq 0}$ the associated weak* continuous semigroup of selfadjoint unital (i.e. $T_t(1)=1$) positive contractive operators on $\L^\infty(G)$, see \cite[pp.~20-21]{VSCC92}, \cite[p.~21]{DtER03} and \cite[p.~301]{Rob91}. By \cite[Proposition 4.13 p.~323]{Rob91} and \cite[Proposition 11.3.1 p.~20]{DtER03}, for any $t>0$, the operator $T_t \co \L^p(G) \to \L^p(G)$ is a convolution operator by a positive function $K_t$ of $\L^1(G)$.

Suppose $1<p<\infty$ and that the Lie group $G$ has polynomial volume growth. By \cite[Theorem 2]{Ale92} and \cite[p.~339]{CRT01}, for any $f \in \C^\infty_c(G)$ we have
\begin{equation}
\label{Riesz1}
\bnorm{\Delta_p^{\frac12}(f)}_{\L^p(G)}
\approx_p \sum_{k=1}^{m} \bnorm{X_k(f)}_{\L^p(G)}.
\end{equation}
Since $\dom \Delta_p$ is a core of $\Delta_p^{\frac12}$, a classical argument \cite[p.~29]{Ouh05} reveals that the subspace $\C^\infty_c(G)$ is a core of the operator $\Delta_p^{\frac{1}{2}}$.

The following observation is a natural complement of the equivalences \eqref{Riesz1}. Since it is always written in the literature without proof, we give an argument. Note that the subspace $\dom X_{1,p} \cap \cdots \cap \dom X_{m,p}$ is considered in the paper \cite[p.~194]{BEM13} and in the book \cite[p.~15]{DtER03} and respectively denoted by $\W_{1,2}'(G)$ and $\L_{2,1}'(G)$.

\begin{prop}
\label{Prop-domain}
Let $G$ be a unimodular connected Lie group with polynomial volume growth. Suppose $1<p<\infty$. We have $\dom \Delta_p^{\frac12} = \dom X_{1,p} \cap \cdots \cap \dom X_{m,p}$. Moreover, for any $f \in \dom \Delta_p^{\frac12}$, we have \eqref{Riesz1}.
\end{prop}

\begin{proof}
Let $f \in \dom \Delta_p^{\frac{1}{2}}$. The subspace $\C^\infty_c(G)$ is dense in $\dom \Delta_p^{\frac{1}{2}}$ equipped with the graph norm. Hence we can find a sequence $(f_n)$ of $\C^\infty_c(G)$ such that $f_n \to f$ and $\Delta_p^{\frac{1}{2}}(f_n) \to \Delta_p^{\frac{1}{2}}(f)$ in $\L^p(G)$. For any integers $n,l$ and any $1 \leq k \leq m$, we obtain
\begin{align*}
\MoveEqLeft
\norm{f_n-f_l}_{\L^p(G)} + \norm{X_k(f_n) - X_k(f_l)}_{\L^p(G)} \\
&\ov{\eqref{Riesz1}}{\lesssim_p} \norm{f_n-f_l}_{\L^p(G)} + \bnorm{\Delta_p^{\frac{1}{2}}(f_n)-\Delta_p^{\frac{1}{2}}(f_l)}_{\L^p(G)}
\end{align*} 
which shows that $(f_n)$ is a Cauchy sequence in $\dom X_{k,p}$. By the closedness of $X_{k,p}$ we infer that this sequence converges to some $g \in \dom X_{k,p}$ equipped with the graph norm. Since $\dom X_{k,p}$ equipped with the graph norm is continuously embedded into $\L^p(G)$, we have $f_n \to g$ in $\L^p(G)$, and therefore $f=g$ since $f_n \to f$. It follows that $f \in \dom X_{k,p}$. This proves the inclusion $\dom \Delta_p^{\frac{1}{2}} \subset \dom X_{k,p}$. Moreover, for any integer $n$, we have 
$$
\norm{X_k(f_n)}_{\L^p(G)} 
\ov{\eqref{Riesz1}}{\lesssim_p}\bnorm{\Delta_p^{\frac{1}{2}}(f_n)}_{\L^p(G)}.
$$
Since $f_n \to f$ in $\dom X_{k,p}$ and in $\dom \Delta_p^{\frac{1}{2}}$ both equipped with the graph norm, we conclude that
$$
\bnorm{X_k(f)}_{\L^p(G)} 
\lesssim_p\bnorm{\Delta_p^{\frac{1}{2}}(f)}_{\L^p(G)}.
$$
The proof of the reverse inclusion and estimate are similar.
\end{proof}

Suppose $1 \leq p < \infty$ and $\alpha > 0$. When $f \in \dom \Delta^{\frac{\alpha}{2}}_p$, we let
\begin{equation}
\label{Def-Lpalpha}
\norm{f}_{\dot{\L}^p_\alpha(G)}
\ov{\mathrm{def}}{=} \bnorm{\Delta_p^{\frac{\alpha}{2}}(f)}_{\L^p(G)}.
\end{equation}
It is related to Sobolev towers, see \cite[Section II.5]{EnN00} and \cite[Section 15.E]{KuW04}. Note that $\norm{\cdot}_{\dot{\L}^p_\alpha(G)}$ is a seminorm on the subspace $\dot{\L}^p_\alpha(G)$ of $\L^p(G)$ (and even a norm if $G$ is not compact). In this paper, we have no intention to define and to use a Banach space $\dot{\L}^p_\alpha(G)$.

Assume that the unimodular connected Lie group $G$ has \textit{polynomial volume growth}. For any $\alpha \in [0,1]$ and any $p \in ]1, +\infty[$, by \cite[Theorem 3]{CRT01} the space $\dot{\L}^p_\alpha(G) \cap \L^\infty(G)$ is an algebra under pointwise product. In \cite[pp.~289-290]{CRT01}, the authors give a simple proof of the case $\alpha=1$. More precisely for all $f,g \in \dot{\L}^p_\alpha(G) \cap \L^\infty(G)$ we have $fg \in \dot{\L}^p_\alpha(G) \cap \L^\infty(G)$ and \eqref{Leibniz-Coulhon}. If $G$ is in addition nilpotent, the conclusion holds for all $\alpha \geq 0$. See also \cite[Theorem 1.4]{GuK96} for the particular case $G=\R^n$ with some generalizations.
In the following statement, the seminorm $\norm{\cdot}_{\dot{\L}^p_\alpha(G)}$ is defined on 
\begin{equation}
\label{def-domaine}
\dom \norm{\cdot}_{\dot{\L}^p_\alpha(G)} 
\ov{\mathrm{def}}{=} \C_0(G) \cap \dom \Delta_p^{\frac{\alpha}{2}}
\end{equation}
where $\C_0(G)$ is the Banach space of complex-valued continuous functions on $G$ that vanish at infinity. Recall that $\C_0(G)$ is equipped with the restriction of the norm $\norm{\cdot}_{\L^\infty(G)}$. If the group $G$ is compact, we have of course the equality $\C_0(G)=\C(G)$ where $\C(G)$ is the Banach space of complex-valued continuous functions on $G$.

\begin{lemma}
\label{Lemma-recapitulatif}
Let $G$ be a connected unimodular Lie group. Suppose $1<p<\infty$ and $\alpha>0$.
\begin{enumerate}
	\item The $\Cbb$-subspace $\dom \norm{\cdot}_{\dot\L^p_\alpha(G)}$ is dense in the Banach space $\C_0(G)$.
	\item The subspace $\dom \norm{\cdot}_{\dot\L^p_\alpha(G)}$ is closed under the adjoint operation $f \mapsto \ovl{f}$.
	\item If $G$ is compact, we have
\begin{equation}
\label{equ-2-proof-prop-Fourier-quantum-compact-metric}
\left\{ f \in \dom \norm{\cdot}_{\dot\L^p_\alpha(G)} : \: \norm{f}_{\dot\L^p_\alpha(G)} = 0 \right\} 
= \Cbb 1_{\C(G)}.
\end{equation}
If $G$ has polynomial volume growth and is non-compact, we have
\begin{equation}
\label{equ-2-proof-prop-Fourier-quantum-compact-metric-bis}
\left\{ f \in \dom \norm{\cdot}_{\dot\L^p_\alpha(G)} : \: \norm{f}_{\dot\L^p_\alpha(G)} = 0 \right\} 
= \{0\}.
\end{equation}
\item If $G$ is compact, the seminorm $\norm{\cdot}_{\dot{\L}^p_\alpha(G)}$ is lower semicontinuous.
\end{enumerate}
\end{lemma}

\begin{proof}
1. Recall that the space $\dom \Delta_p^n$ is a core of the operator $\Delta_p^{\frac{\alpha}{2}}$ if $\frac{\alpha}{2}<n$. Consequently the domain $\dom \norm{\cdot}_{\dot\L^p_\alpha(G)}\ov{\eqref{def-domaine}}{=} \C_0(G) \cap \dom \Delta_p^{\frac{\alpha}{2}}$ contains the subspace $\C^\infty_c(G)$ of $\C_0(G)$. Note that this subspace is dense in $\C_0(G)$ by regularization by \cite[Theorem 2.11]{Mag92}. We infer that the $\Cbb$-subspace $\dom \norm{\cdot}_{\L^p_\alpha(G)}$ is dense in the Banach space $\C_0(G)$. 

2. Note that the space $\C_0(G)$ is obviously closed under the adjoint operation $f \mapsto \ovl{f}$. We will show that $\dom \norm{\cdot}_{\dot{\L}^p_\alpha(G)} $ is equally closed under the same operation. 

Let $f \in \dom \Delta_p^{\frac{\alpha}{2}}$. We know that the subspace $\dom \Delta_p^n$ is core of $\Delta_p^{\frac{\alpha}{2}}$. Hence there exists a sequence $(f_j)$ of $\dom \Delta_p^n$ such that $f_j \to f$ and $\Delta_p^{\frac{\alpha}{2}}(f_j) \to \Delta_p^{\frac{\alpha}{2}}(f)$. We have $\ovl{f_j} \to \ovl{f}$ and $\Delta_p^{\frac{\alpha}{2}}(\ovl{f_j}) = \ovl{\Delta_p^{\frac{\alpha}{2}}(f_j)} \to  \ovl{\Delta_p^{\frac{\alpha}{2}}(f)}$ where the equality can be seen with \eqref{Balakrishnan}. By \eqref{Domain-closure}, we conclude that $\ovl{f} \in \dom \Delta_p^{\frac{\alpha}{2}}$ and that $\Delta_p^{\frac{\alpha}{2}}(\ovl{f})=\ovl{\Delta_p^{\frac{\alpha}{2}}(f)}$. We conclude that $\dom \norm{\cdot}_{\L^p_\alpha(G)}$ is closed under the adjoint operation $f \mapsto \ovl{f}$. 

3. We have $\Delta_p(1) \ov{\eqref{def-de-Delta}}{=} -\sum_{k=1}^{m} X_k^2(1)= 0$. Hence the constant function $1$ belongs to $\ker \Delta_p \ov{\eqref{inclusion-range}}{=} \ker \Delta_p^{\frac \alpha2}$. We conclude that $\norm{1}_{\dot{\L}^p_\alpha(G)}=\bnorm{\Delta_p^{\frac\alpha2}(1)}_p=0$.

In the other direction, if $\norm{f}_{\dot{\L}^p_\alpha(G)} = 0$, we have $\bnorm{\Delta_p^{\frac\alpha2}(f)}_p=0$. Hence $f$ belongs to $\ker \Delta_p^{\frac\alpha2}$. By \eqref{inclusion-range}, we deduce that $\bnorm{\Delta_p^{\frac12}(f)}_p=0$. Then according to Proposition \ref{Prop-domain} and \eqref{Riesz1}, we have $\norm{X_k(f)}_p = 0$ for any $k$. By Lemma \ref{Lemma-useful-estimate}, we infer that $\lambda_s(f)=f$ for any $s \in G$. If $G$ is compact, we conclude that the function $f$ is constant, that is $f \in \Cbb 1$ and that $f=0$ if $G$ is not compact.

4. Let $f \in \C(G)$ and $(f_n)$ be a sequence of elements of $\C(G) \cap \dom \Delta^{\frac\alpha2}_p$ such that $(f_n)$ converges to $f$ for the norm topology of $\C(G)$ and $\norm{f_n}_{\dot{\L}^p_\alpha(G)} \leq 1$ for any $n$, that is $\bnorm{\Delta_p^{\frac{\alpha}{2}}(f_n)}_{\L^p(G)} \leq 1$ by \eqref{Def-Lpalpha}. We have to prove that $f$ belongs to $\dom \Delta^{\frac\alpha2}_p$ and that $\norm{f}_{\dot{\L}^p_\alpha(G)} \leq 1$. 

Since $\norm{\cdot}_{\L^p(G)} \leq \norm{\cdot}_{\C(G)}$, the sequence $(f_n)$ converges to $f$ for the norm topology of $\L^p(G)$, hence for the weak topology of $\L^p(G)$. Note that the sequence $\Delta_p^{\frac{\alpha}{2}}(f_n)$ is bounded in the Banach space $\L^p(G)$. Since bounded sets are weakly relatively compact by \cite[Theorem 2.8.2]{Meg98}, there exists a weakly convergent subnet $\big(\Delta_p^{\frac{\alpha}{2}}f_{n_j}\big)$. Then $\big(f_{n_j},\Delta_p^{\frac{\alpha}{2}}f_{n_j}\big)$ is a weakly convergent net in the graph of the closed operator $\Delta_p^{\frac{\alpha}{2}}$. Note that this graph is closed and convex, hence weakly closed by \cite[Theorem 2.5.16]{Meg98}. Thus the limit of $\big(f_{n_j},\Delta_p^{\frac{\alpha}{2}}f_{n_j}\big)$ belongs again to the graph and is of the form $\big(g,\Delta_p^{\frac{\alpha}{2}}g\big)$ for some $g \in \dom \Delta_p^{\frac{\alpha}{2}}$. In particular, $(f_{n_j})$ converges weakly to $g$ and $\Delta_p^{\frac{\alpha}{2}}(f_{n_j})$ converges weakly to $\Delta_p^{\frac{\alpha}{2}}(g)$. We infer that $f=g$. We conclude that $f$ belongs to $\dom \Delta_p^{\frac{\alpha}{2}}$. Moreover, using the weakly lower semicontinuity of the norm \cite[Theorem 2.5.21]{Meg98}, we obtain
\begin{align*}
\MoveEqLeft
\norm{f}_{\dot{\L}^p_\alpha(G)}
\ov{\eqref{Def-Lpalpha}}{=}\bnorm{\Delta_p^{\frac{\alpha}{2}}(f)}_{\L^p(G)}
\leq \liminf_j \bnorm{\Delta_p^{\frac{\alpha}{2}}(f_{n_j})}_{\L^p(G)}
\leq 1.
\end{align*}
\end{proof}

\section{Quantum compact metric spaces}
\label{Sec-quantum-compact-matric-spaces}

\paragraph{Lipschitz pairs and quantum compact metric spaces} 
Following \cite[Definition 2.3]{Lat16a}, a Lipschitz pair $(A,\norm{\cdot})$ is a $\mathrm{C}^*$-algebra $A$ equipped with a seminorm $\norm{\cdot}$ defined on a dense subspace $\dom \norm{\cdot}$ of the selfadjoint part $(uA)_\sa$ such that
\begin{equation}
\label{point1}
\left\{a \in \dom \norm{\cdot} : \norm{a} = 0 \right\} 
= \R1_{uA}
\end{equation}
where $uA$ is the unitization of the algebra $A$. If $A$ is in addition unital, we say that $(A,\norm{\cdot})$ is a unital Lipschitz pair.

Recall that a state of a $\mathrm{C}^*$-algebra $A$ is a positive linear form $\varphi$ on $A$ with $\norm{\varphi}=1$. If $X$ is a compact topological space and if $A=\C(X)$, a state is the integral associated to a regular Borel measure of probability on $X$. 

A pair $(A,\norm{\cdot})$ is a quantum compact metric space when:
\begin{enumerate}
\item $(A,\norm{\cdot})$ is a unital Lipschitz pair. 

\item The Monge-Kantorovich metric on the set $\S(A)$ of the states of $A$, defined for any two states $\varphi,\psi \in \S(A)$ by:
\begin{equation}
\label{MongeKant-def}
\dist_{\MK}(\varphi, \psi) 
\ov{\mathrm{def}}{=} \sup \left\{ |\varphi(a) - \psi(a)| : a \in \dom \norm{\cdot} \text{ and }\norm{a} \leq 1 \right\}\text{,}
\end{equation}
induces the weak* topology on $\S(A)$.
\end{enumerate}
In this case, we say that $\norm{\cdot}$ is a Lip-norm. We refer to the nice surveys \cite{Lat16a} and \cite{Rie04} and references therein for more information.

\begin{example}
\label{ex-Lip-functions}
\normalfont
If $(X,\dist)$ is a compact metric space, a fundamental example \cite[Example 2.6]{Lat16a}, \cite[Example 2.9]{Lat16b} is given by $(\C(X),\Lip)$ where $\C(X)$ is the commutative $\C^*$-algebra of continuous functions on $X$ and where $\Lip$ is the Lipschitz seminorm, defined for any Lipschitz function $f \co X \to \Cbb$ by 
\begin{equation}
\label{Lip-eq}
\Lip(f) 
\ov{\mathrm{def}}{=} \sup\left\{\frac{|f(x)-f(y)|}{\dist(x,y)} : x,y\in X,x\not=y \right\}.
\end{equation}
The set of real Lipschitz functions is norm-dense in $\C(X)_\sa$ by the Stone-Weierstrass theorem. Indeed, $\Lip(X)$ contains the constant functions. Moreover, $\Lip(X)$ separates points in $X$. If $x_0,y_0 \in X$ with $x_0 \not=y_0$, we can use the lipschitz function $f \co X\to \R$, $x \mapsto \dist(x,y_0)$ since we have $f(x_0)>0=f(y_0)$. Moreover, it is immediate that a function $f$ has zero Lipschitz constant if and only if it is constant on $X$, i.e. \eqref{point1} is satisfied. 



In the case of $(\C(X),\Lip)$, the equality \eqref{MongeKant-def} gives the dual formulation of the classical Kantorovich-Rubinstein metric \cite[Remark 6.5]{Vil09} for Borel probability measures $\mu$ and $\nu$ on $X$
\begin{equation}
\label{Distance-Monge-commutatif}
\dist_{\mk}(\mu,\nu) 
\ov{\mathrm{def}}{=} \sup \left\{\left|\int_X f \d\mu- \int_X f\d \nu \right| : f \in \C(X)_\sa, \Lip(f) \leq 1 \right\}.
\end{equation}
which is a basic concept in optimal transport theory \cite{Vil09}. Considering the Dirac measures $\delta_x$ and $\delta_y$ at points $x,y \in X$ instead of $\mu$ and $\nu$, we recover the distance $\dist(x,y)$ with the formula \eqref{Distance-Monge-commutatif}.

\paragraph{Characterizations of quantum compact metric spaces} The compatibility of Monge-Kantorovich metric with the weak* topology is hard to check directly in general. Fortunately, there exists a condition which is more practical. This condition is inspired by the fact that Arz\'ela-Ascoli's theorem shows that for any $x \in X$ the set
$$
\big\{f \in \C(X)_\sa :  \Lip(f) \leq 1, f(x)= 0 \big\}
$$
is norm relatively compact and it is known that this property implies that \eqref{Distance-Monge-commutatif} metrizes the weak* topology on the space of Borel probability measures on $X$. Now, we give sufficient conditions in order to obtain quantum compact metric spaces, \cite[Theorem 2.43]{Lat16a}. See also \cite[Proposition 1.3]{OzR05}.

\begin{prop}
\label{Prop-carac-quantum}
Let $(A,\norm{\cdot})$ be a unital Lipschitz pair. The following assertions are equivalent:
\begin{enumerate}
\item[(a)] $(A,\norm{\cdot})$ is a quantum compact metric space,

\item[(b)] there exists a state $\mu \in \S(A)$ such that the set $
\left\{ a \in A_\sa : \norm{a} \leq 1, \mu(a) = 0 \right\}$ is relatively compact in $A$ for the topology of the norm of $A$,

\item[(c)] for all states $\mu \in \S(A)$, the set $\left\{ a \in A_\sa : \norm{a} \leq 1, \mu(a) = 0 \right\}$ is relatively compact in $A$ for the topology of the norm of $A$.
\end{enumerate}
\end{prop}

\paragraph{Quasi-Leibniz quantum compact metric space} The Lipschitz seminorm $\mathrm{Lip}$ of Example \ref{ex-Lip-functions} associated to a compact metric space $(X,\dist)$ enjoys a natural property with respect to the multiplication of functions in $\C(X)$, called the Leibniz property for any Lipschitz functions $f,g \co X \to \Cbb$:
\begin{equation}
\label{equa-Leibniz}
\Lip(fg)
\leq \norm{f}_{\C(X)}\Lip(g) + \Lip(f)\norm{g}_{\C(X)}.
\end{equation}

Moreover, the Lipschitz seminorm is lower-semicontinuous with respect to the norm of the algebra $\C(X)$, i.e. the uniform convergence norm on $X$. 
\end{example}

We want to have these additional properties for a quantum compact metric space $(A,\norm{\cdot})$. Unfortunately, because of difficulties with Lipschitz seminorms, Latr\'emoli\`ere has not chosen a direct generalization of \eqref{equa-Leibniz} in this work on quantum compact metric spaces. He introduced the following definition by considering the Jordan-Lie-algebra of selfadjoint elements. We say that a quantum compact metric space $(A,\norm{\cdot})$ is a $(C,0)$-quasi-Leibniz quantum compact metric space if $\norm{\cdot}$ is Jordan-Lie subalgebra of $A$ and if for any $a,b \in \dom \norm{\cdot}$ we have
\begin{equation}
\label{Jordan-Lie}
\norm{a \circ b}
\leq C\big[\norm{a} \norm{b}_A + \norm{a}_A \norm{b}\big]
\quad\text{and} \quad
\norm{\{a,b\}}
\leq C\big[\norm{a}\norm{b}_A + \norm{a}_A \norm{b}\big]
\end{equation}
for some constant $C >0$, where we use the Jordan product $a \circ b \ov{\mathrm{def}}{=} \frac{1}{2}(ab+ba)$ and the Lie product $\{a,b\} \ov{\mathrm{def}}{=} \frac{1}{2 \i} (ab-ba)$ and if $\norm{\cdot}$ is lower semicontinuous, i.e. 
\begin{equation}
\label{semicontinuous-1}
\{ x \in \dom \norm{\cdot} : \norm{x} \leq 1 \}
\end{equation}
is closed for the topology of the norm of $A$.



The following is essentially \cite[Proposition 2.17]{Lat16b} and Proposition \ref{Prop-carac-quantum}. It is our main tool for checking the definition of quasi-Leibniz quantum compact metric spaces.

\begin{prop}
\label{Prop-carac-quantum2}
Let $A$ be a unital $\C^*$-algebra and $\norm{\cdot}$ be a seminorm defined on a dense $\Cbb$
-subspace $\dom \norm{\cdot}$ of $A$, such that 
\begin{enumerate}
	\item $\dom \norm{\cdot}$ is closed under the adjoint operation,
	\item $\{a \in \dom \norm{\cdot}: \norm{a} = 0\} = \Cbb 1_A$,
	\item there exists a constant $C>0$ such that for all $a, b \in \dom \norm{\cdot}$, we have
\begin{equation}
\label{Leibniz}
\norm{ab} 
\leq  C\big[\norm{a}_A \norm{b} + \norm{a} \norm{b}_A\big],
\end{equation}
\item there exists a state $\mu \in \S(A)$ such that the set $\{a \in \dom \norm{\cdot}: \norm{a} \leq 1, \mu(a) = 0\}$ is relatively compact in $A$ for the topology of the norm of $A$. 
\item $\norm{\cdot}$ is lower semicontinuous.
\end{enumerate}
If $\norm{\cdot}_\sa$ is the restriction of $\norm{\cdot}$ to $A_\sa \cap \dom\norm{\cdot}$, then $(A_\sa,\norm{\cdot}_\sa)$ is a $(C,0)$-quasi-Leibniz quantum compact metric space.
\end{prop}

\paragraph{New quantum compact metric spaces} Let $G$ be a connected Lie group equipped with a family $X=(X_1, \ldots, X_m)$ of left-invariant H\"ormander vector fields and a left Haar measure $\mu_G$. In this section, we suppose that $G$ is \textit{compact}. For the introduction of new quantum compact metric spaces, we need some preliminary results related to some estimates of the heat kernel. For any $s \in G$, a particular case of \cite[Theorem V.4.3]{VSCC92} gives 
\begin{equation}
\label{estimate-kernel}
0 \leq K_t(s)
\lesssim \frac{1}{t^{\frac{d}{2}}}, \quad 0 < t \leq 1 
\end{equation}
where the local dimension $d$ is defined in \eqref{local-dim}. 


The following is essentially \cite[pp.~339-341]{Rob91}. Since a point of \cite[pp.~339-341]{Rob91} is misleading and since it is fundamental for us, we give an argument relying on the same nice ideas.
 
\begin{lemma}
\label{Lemma-compact-resolvent}
The operator $\Delta_2 \co \dom(\Delta_2)  \subset \L^2(G) \to \L^2(G)$ has compact resolvent and we have the estimate
\begin{equation}
\label{estimates-L1-Linfty}
\norm{T_t}_{\L^1(G) \to \L^\infty(G)}
\lesssim \frac{1}{t^{\frac{d}{2}}}, \quad 0 <t \leq 1.
\end{equation} 
\end{lemma}

\begin{proof}
Note that for any $0 <t \leq 1$, we have
\begin{equation}
\label{Estimation-norm2-kernel}
\norm{K_t}_{\L^2(G)}
\lesssim \norm{K_t}_{\L^\infty(G)}
\ov{\eqref{estimate-kernel}}{\lesssim} \frac{1}{t^{\frac{d}{2}}}. 
\end{equation}
By translation invariance of the normalized Haar measure of $G$, we deduce that
\begin{equation}
\label{Schmidt-calculus}
\int_{G \times G} |K_t(s r^{-1})|^2 \d s\d r
=\int_{G} \bigg(\int_{G}|K_t(s r^{-1})|^2 \d s\bigg) \d r
=\int_{G} \bigg(\int_{G}|K_t(s)|^2 \d s\bigg) \d r
\ov{\eqref{Estimation-norm2-kernel}}{\leq} \frac{1}{t^d}.
\end{equation}
For any $t>0$, we deduce by \cite[Exercise 2.8.38 p.~170]{KaR97} and \eqref{Convolution-formulas} that $T_t \co \L^2(G) \to \L^2(G)$ is a Hilbert-Schmidt operator. 
By \cite[Theorem 4.29 p.~119]{EnN00}, we conclude that the operator $\Delta_2$ has compact resolvent.
Finally, for any $t>0$, we have
$$
\norm{T_t}_{\L^1(G) \to \L^\infty(G)}
\ov{\eqref{Dunford-Pettis-1}}{=} \underset{s,r \in G}{\esssup}\ |K_t(sr^{-1})|
\ov{\eqref{estimate-kernel}}{\lesssim} \frac{1}{t^{\frac{d}{2}}}.
$$
\end{proof}

\begin{lemma}
\label{Lemma-compactness}
The operator $\Delta_2^{-1} \co \L^2_0(G) \to \L^2_0(G)$ is compact. 
\end{lemma}

\begin{proof}
By Lemma \ref{Lemma-compact-resolvent}, the operator $\Delta_2 \co \dom \Delta_2 \subset \L^2(G) \to \L^2(G)$ has compact resolvent. Note that $\ker \Delta_2$ is an eigenspace, hence a reducing subspace for the selfadjoint operator $\Delta_2$. So for any $\lambda$ in the resolvent subset $\rho(\Delta_2)$, we have a well-defined operator $(\lambda-\Delta_2)^{-1} \co (\ker \Delta_2)^\perp \to (\ker \Delta_2)^\perp$ which is compact by composition. By the resolvent identity \cite[p.~273]{Haa06}, we deduce that $\Delta_2^{-1} \co (\ker \Delta_2)^\perp \to (\ker \Delta_2)^\perp$ is also compact by \cite[p.~3]{EcI18}. Recall that $(\ker \Delta_2)^\perp=\L^2_0(G)$. We conclude that $\Delta_2^{-1} \co \L^2_0(G) \to \L^2_0(G)$ is compact. 
\end{proof}

For the next statement, the domain of $\norm{\cdot}_{\dot{\L}^p_\alpha(G)}$ is defined as in \eqref{def-domaine}.

\begin{thm}
\label{Th-quantum-metric}
Let $G$ be a compact connected Lie group equipped with a family $(X_1, \ldots, X_m)$ of left-invariant H\"ormander vector fields. Suppose $0 < \alpha \leq 1$ (or $0<\alpha$ if $G$ is nilpotent) and $\frac{d}{\alpha} < p < \infty$ where $d$ is the local dimension defined in \eqref{local-dim}. Then $\big(\C(G),\norm{\cdot}_{\dot{\L}^p_\alpha(G)}\big)$ defines a $(C_{\alpha,p},0)$-quasi-Leibniz quantum compact metric space for some constant $C_{\alpha,p}>0$.
\end{thm}

\begin{proof}
We will prove the assumptions of Proposition \ref{Prop-carac-quantum2}.  The third point of Proposition \ref{Prop-carac-quantum2} is satisfied by \eqref{Leibniz-Coulhon} and the first two points by Lemma \ref{Lemma-recapitulatif}.

Since the normalized integral $\int_G \co \C(G) \to \Cbb$ is a state of the unital $\C^*$-algebra $\C(G)$, it suffices to show that
\begin{equation}
\label{equ-proof-prop-Fourier-quantum-compact-metric}
\left\{ f  \in \dom \norm{\cdot}_{\dot{\L}^p_\alpha(G)} : \norm{f}_{\dot{\L}^p_\alpha(G)} \leq 1 ,\: \int_G f = 0 \right\}
\end{equation} 
is relatively compact in $\C(G)$.

Note that \cite[p.~38]{DtER03} contains a proof of the existence of $\omega>0$ such that
\begin{equation}
\label{A-l-infini}
\norm{T_t}_{\L_0^1(G) \to \L^\infty(G)} 
\lesssim \e^{-\omega t}  \quad t \geq 1.
\end{equation}
Combined with \eqref{estimates-L1-Linfty}, we deduce the estimate
\begin{equation*}
\norm{T_t}_{\L_0^1(G) \to \L^\infty(G)} 
\lesssim \frac{1}{t^\frac{d}{2}}  \quad t > 0.
\end{equation*}
With Lemma \ref{Lemma-compactness}, we conclude that the assumptions of Proposition \ref{prop-interpolation-1} are satisfied. Using this result with $z=\frac{\alpha}{2}$ and $q=\infty$, the operator $\Delta^{-\frac\alpha2} \co \L^p_0(G) \to \L^\infty_0(G)$ is compact if $p > \frac{d}{\alpha}$. So the image $\mathcal{I}$ by $\Delta^{-\frac\alpha2}$ of the closed unit ball $\{g \in \ovl{\Ran \Delta_p} : \norm{g}_{\L^p(G)} \leq 1\}$ of $\L^p_0(G)=\ovl{\Ran \Delta_p}$ is relatively compact. Note that $\Ran \Delta_p^{\frac\alpha2} \subset \ovl{\Ran \Delta_p}$ by \eqref{inclusion-range}. Hence the subset (write $f=\Delta^{-\frac\alpha2}\Delta_p^{\frac\alpha2}f$)
$$
\Big\{ f \in \C(G)_0 \cap \dom \Delta_p^{\frac\alpha2} : \bnorm{\Delta_p^{\frac\alpha2}(f)}_{\L^p(G)} \leq 1 \Big\} 
$$ 
of $\mathcal{I}$ is relatively compact in $\C(G)$ where $\C(G)_0$ is the subspace of continuous functions with null integral. 
Since we have
\begin{align*}
\MoveEqLeft
\big\{f \in \C(G)_{0} \cap \dom \Delta_p^{\frac\alpha2}: \bnorm{\Delta_p^{\frac\alpha2}(f)}_{\L^p(G)} \leq 1 \big\} \\
&\ov{\eqref{Def-Lpalpha}\eqref{def-domaine}}{=} \bigg\{ f \in \dom \norm{\cdot}_{\dot{\L}^p_\alpha(G)} : \norm{f}_{\dot{\L}^p_\alpha(G)} \leq 1, \: \int_G f 
= 0 \bigg\}
\end{align*} 
we deduce that the subset \eqref{equ-proof-prop-Fourier-quantum-compact-metric} is also relatively compact in $\C(G)$. The proof is complete.
\end{proof}

\begin{remark} \normalfont
The result is sharp. Consider the abelian compact group $G = \T^2$ and the Laplacian $\Delta_2 \co \dom \Delta_2 \subset \L^2(\T^2) \to \L^2(\T^2)$, $ \e^{n\i \cdot} \ot \e^{m\i \cdot} \mapsto -(n^2 + m^2)\e^{n\i \cdot} \ot \e^{m\i \cdot}$ and $\alpha=1$. By Example \ref{Ex-dimension-torus}, the local dimension $d$ of $\T^2$ is 2. In \cite[Remark 5.3]{ArK22} and its proof, it is showed that the set
$$
\big\{ f \in \mathrm{C}(\T^2)_0 \cap \dom \Delta^{\frac12}_2 : \: \bnorm{\Delta_2^{\frac12}(f)}_{\L^2(\T^2)} \leq 1 \big\}
$$
is not bounded, in particular, not relatively compact. With the notation \eqref{Def-Lpalpha}, this set can be written
\begin{equation}
\left\{ f  \in \dom \norm{\cdot}_{\dot{\L}^2_1(\T^2)} : \norm{f}_{\dot{\L}^2_1(\T^2)} \leq 1 ,\: \int_{\T^2} f = 0 \right\}.
\end{equation} 
By Proposition \ref{Prop-carac-quantum}, we conclude that we does not have in general a quantum compact metric space under the critical condition $p=\frac{d}{\alpha}$.
\end{remark}

\begin{remark} \normalfont
The inequality \eqref{Leibniz-Coulhon} is open if $\alpha>1$. It would be interesting to find a counter-example. 
\end{remark}

\begin{remark} \normalfont
We can replace the subelliptic Laplacian $\Delta$ of \eqref{def-de-Delta} by a real second order subelliptic operator 
$H 
\ov{\mathrm{def}}{=} -\sum_{i,j=1}^{m} c_{ij} X_iX_j$ where $c_{ij} \in \R$, satisfying the condition $\frac{1}{2}(C+C^T) \geq \mu\I$ for some $\mu >0$ and $C=[c_{ij}]$. The $\L^p$-realization $H_p$ of this operator is a closed operator with domain $\dom H_p=\L'_{p,2}$. See \cite[Chapter II]{DtER03} for more information. In the case of a compact connected Lie group $G$, the boundedness of Riesz transforms is proved in \cite[p.~39]{DtER03}. Moreover, for any $f \in \L'_{2,1}$ we have by \cite[pp.~16-17]{DtER03}
$$
\Re\, \langle f, H_2f \rangle_{\L^2(G)}
\geq \mu \sum_{k=1}^{m} \norm{X_k f}_{\L^2(G)}^2.
$$
In particular, $H_2f=0$ if and only if for any $k \in \{1,\ldots,m\}$ we have $X_kf=0$. This observation is useful for obtaining a suitable generalization of the third point of Lemma \ref{Lemma-recapitulatif} (unfortunately, this argument only works in the case $p \geq 2$).

The generalization of the Leibniz's rule \eqref{Leibniz-Coulhon} for these operators for $\alpha \in ]0,1[$ is an open question.
\end{remark}

\begin{remark} \normalfont
It may be worthy to study the family of the quantum compact metric spaces $\big(\C(G),\norm{\cdot}_{\dot{\L}^p_\alpha(G)}\big)$ when $p \to \frac{d}{\alpha}$ from the perspective of the quantum Gromov-Haudorff distance. What can be said about the ``limit'' ?
\end{remark}

\begin{remark} \normalfont
It is unclear if there exists a formula for the restriction of the Monge-Kantorovich metric \eqref{MongeKant-def} on the subset of pure states of $\C(G)$, i.e. the map
$$
(s,s') \mapsto \sup \big\{|f(s)-f(s')| : f \in  \C(G,\R),\norm{f}_{\dot{\L}^p_\alpha(G)} \leq 1\big\}.
$$ 
on $G \times G$. It would be interesting to understand this quantity to equivalence with respect to a constant. The question is natural when we compare to the next situation of Theorem \ref{Th-recover-metric}.
\end{remark}

\section{Quantum locally compact metric spaces}
\label{Sec-locally}

\paragraph{Quantum locally compact metric spaces} 
The basic reference is \cite{Lat13}. A topography on a C*-algebra $A$ is an abelian C*-subalgebra $\frak{M}$ of $A$ containing an approximate identity for $A$. A topographic quantum space $(A,\frak{M})$ is an ordered pair of a C*-algebra $A$ and a topography $\frak{M}$ on $A$. Let $(A,\frak{M})$ be a topographic quantum space. A state $\varphi \in \S(A)$ is local when there exists a compact $K$ of the Gelfand spectrum of $\frak{M}$ such that $\varphi(1_K) = 1$. A Lipschitz triple $(A,\norm{\cdot},\frak{M})$ is a triple where $(A,\norm{\cdot})$ is a Lipschitz pair and $\frak{M}$ is a topography on $A$.



Let $(A,\norm{\cdot},\frak{M})$ be a Lipschitz triple. The definition of quantum locally compact quantum metric spaces of \cite{Lat13} is equivalent to say that $(A,\norm{\cdot},\frak{M})$ is a quantum locally compact quantum metric space if and only if for any compactly supported element $g, h \in \frak{M}$ and for some local state $\mu$ of $A$, the set:
\begin{equation*}
\big\{ g a h : a \in (uA)_\sa, \norm{a} \leq 1, \mu(a) = 0 \big\}
\end{equation*}
is relatively compact for the topology associated to $\norm{\cdot}_A$. Here we identify $\mu$ with its unique extension $a+\lambda 1 \mapsto \mu(a) +\lambda$ as a state of the unital C*-algebra $uA$.


\paragraph{Quasi-Leibniz quantum locally compact metric space} Unfortunately, Latr\'emoli\`ere did not generalize the notion of definition of quasi-Leibniz quantum compact metric spaces of Section \ref{Sec-quantum-compact-matric-spaces} to the locally compact case. We make an attempt by saying that a quantum locally compact quantum metric space $(A,\norm{\cdot},\frak{M})$ is a quasi-Leibniz quantum locally compact metric space if the restriction of $\norm{\cdot}$ on $A_{\sa}$ satisfies the points \eqref{Jordan-Lie} and \eqref{semicontinuous-1} which is slightly less general than \cite[Section 5.5]{ArK22}.

\paragraph{Criterion of relative compactness} The following is a locally compact group generalization \cite[Exercise 26 VIII.72]{Bou04} \cite[Problem 4 p.~283]{Dieu83} (see also \cite[Theorem A.4.1]{DtER03} for a particular case) of the classical Fr\'echet-Kolmogorov theorem on relative compactness.

\begin{thm}
\label{Th-compactness-Rob}
Let $G$ be a locally compact group equipped with a left Haar measure. Suppose $1 \leq p < \infty$. Let $\cal{F}$ be a subset of the Banach space $\L^p(G)$. Then $\cal{F}$ is relatively compact if and only if there exists $M>0$ such that
\begin{equation}
\label{}
\lim_{s \to e} \sup_{f \in \cal{F}} \norm{\lambda_s f-f}_{\L^p(G)}
=0,
\end{equation}
\begin{equation}
\label{}
\sup_{f \in \cal{F}} \norm{f}_{\L^p(G)} \leq M
\quad \text{and} \quad \lim_{r \to \infty} \sup_{f \in \cal{F}} \int_{G-B(e,r)} |f(s)|^p \d\mu_G(s)
=0.
\end{equation}
\end{thm}

Now, consider a connected Lie group $G$ equipped with a family $X$ of left-invariant H\"ormander vector fields with polynomial volume growth and local dimension $d$. We suppose that $G$ is \textit{not} compact. Let $K$ be a compact subset of $G$. We denote by $\C_K(G)$ the space of continuous functions on $G$ with support in $K$.

Suppose $1<p<\infty$ and $\alpha \geq 0$. Following essentially \cite[p.~287]{CRT01} we define the subspace
\begin{equation}
\label{Def-Sobolev}
\L^p_\alpha(G) 
\ov{\mathrm{def}}{=} \dom \Delta^{\frac{\alpha}{2}}_p
\end{equation}
of $\L^p(G)$. If $f \in \L^p_\alpha(G) $, we will use the notation
\begin{equation}
\label{Def-Lpalpha-bis}
\norm{f}_{\L^p_\alpha(G)}
\ov{\mathrm{def}}{=} \bnorm{\Delta_p^{\frac{\alpha}{2}}(f)}_{\L^p(G)}+\norm{f}_{\L^p(G)} 
\ov{\eqref{equiv-Egert}}{\approx} \bnorm{(\Id+\Delta_p)^{\frac{\alpha}{2}}(f)}_{\L^p(G)}.
\end{equation}
We refer to \cite[Section II.5]{EnN00} and \cite[Section 15.E]{KuW04} for the link with Sobolev towers. If $\alpha p>d$, we have by \cite[p.~287]{CRT01} \cite[Theorem 4.4 (c)]{BPTV19} a Sobolev embedding $\L^p_\alpha(G) \subset \L^\infty(G)$: 
\begin{equation}
\label{Sobolev-embedding-bis}
\norm{f}_{\L^\infty(G)} 
\lesssim \norm{f}_{\L_\alpha^p(G)}, \quad f \in \dom \Delta_p^{\frac{\alpha}{2}}.
\end{equation}
Note that by \cite[Proposition 3.2.3]{Haa06} the Bessel potential $(\Id+\Delta)^{-\alpha}$ is a \textit{bounded} operator on the Banach space $\L^p(G)$ for any $\alpha \in \Cbb$ with $\Re \alpha>0$. Consequently if $0 < \alpha \leq \beta$ it is obvious to check with \eqref{equiv-Egert} that
\begin{equation}
\label{increasing-Bessel}
\norm{f}_{\L_\alpha^p(G)} 
\lesssim \norm{f}_{\L_\beta^p(G)}.
\end{equation}
A contractive inclusion for the case $G=\R^n$ is proved in \cite[p.~135]{Ste70} with a different argument. A \textit{contractive} version of \eqref{increasing-Bessel} is stated without proof in the inequality following \cite[(3.1)]{BPTV19} but it is a mistake confirmed by the authors of this paper.

\begin{prop}
\label{Prop-loc-1}
Let $G$ be a non-compact connected Lie group equipped with a family $(X_1, \ldots, X_m)$ of left-invariant H\"ormander vector fields. Suppose that $G$ has polynomial volume growth. Let $\alpha \geq 1$ and $\max\{1,\frac{d}{\alpha}\} < p < \infty$. If $g \co G \to \Cbb$ is a compactly supported continuous function then the subset
\begin{equation}
\label{equ-proof-prop-Schur-quantum-compact-metric-3-bis}
g \big\{f \in \C_0(G) \cap \dom \Delta_p^{\frac\alpha2}  : \norm{f}_{\L^p_\alpha(G)} \leq 1, f(e)=0 \big\}
\end{equation}
is relatively compact in $\L^\infty(G)$.
\end{prop}

\begin{proof}
Let $K$ be a compact subset of $G$. For any $M \geq 0$, consider the subset
\begin{equation}
\label{Interminable-subset-bis}
E_{K,p,M} 
\ov{\mathrm{def}}{=} \Big\{ f \in \C_K(G) \cap \dom \Delta_p^{\frac{\alpha}{2}} : \norm{f}_{\L^p_\alpha(G)} \leq M \Big\} 
\end{equation} 
of the Banach space $\L^p(G)$. If $f \in E_{K,p,M}$, using the Sobolev embedding $\L^p_\alpha(G) \subset \L^\infty(G)$ we obtain
$$
\norm{f}_{\L^p(G)} 
\lesssim_{K,p} \norm{f}_{\L^\infty(G)} 
\ov{\eqref{Sobolev-embedding-bis}}{\lesssim} \norm{f}_{\L_\alpha^p(G)}
\leq M.
$$
Consequently, the subset $E_{K,p,M}$ is bounded in $\L^p(G)$. Moreover, using Lemma \ref{Lemma-useful-estimate}, we have for any function $f \in E_{K,p,M}$ and any $s \in G$
\begin{align*}
\MoveEqLeft
\norm{(\Id-\lambda_s)f}_{\L^p(G)}
\ov{\eqref{Ine-pratique}}{\leq} \dist_\CC^{p^*}(s,e) \bigg(\sum_{k=1}^m \norm{X_{k,p}(f)}_{\L^p(G)}^p \bigg)^{\frac{1}{p}}
\approx_p \dist_\CC^{p^*}(s,e) \sum_{k=1}^{m} \bnorm{X_{k,p}(f)}_{\L^p(G)} \\
&\ov{\eqref{Riesz1}}{\lesssim_{p}} \dist_\CC^{p^*}(s,e) \bnorm{\Delta_p^{\frac12}(f)}_{\L^p(G)}\ov{\eqref{Def-Lpalpha-bis}}{\leq} \dist_\CC^{p^*}(s,e) \norm{f}_{\L^p_1(G)} \\
&\ov{\eqref{increasing-Bessel}}{\lesssim} \dist_\CC^{p^*}(s,e) \norm{f}_{\L^p_\alpha(G)}
\ov{\eqref{Interminable-subset-bis}}{\leq} M \dist_\CC^{p^*}(s,e).
\end{align*} 
With Theorem \ref{Th-compactness-Rob}, we obtain the relative compactness of the subset $E_{K,p,M}$ in $\L^p(G)$ and of its subset
$$
F_{K,p,M} 
\ov{\mathrm{def}}{=} \Big\{ f \in \C_K(G) \cap \dom \Delta_p^{\frac{\alpha}{2}} : \norm{f}_{\L_\alpha^p(G)} \leq M, f(e)=0 \Big\}. 
$$
The operator $(\Id+\Delta)^{-\frac\alpha2} \co \L^p(G) \to \ovl{\Ran \Delta_\infty}$ is bounded by \eqref{Sobolev-embedding-bis} and \eqref{Def-Lpalpha-bis} (hence uniformly continuous). Applying this operator to the previous subset by writing $f=(\Id+\Delta)^{-\frac\alpha2}(\Id+\Delta_p)^{\frac\alpha2}f$, we obtain that the set $F_{K,p,M}$ is relatively compact in $\L^\infty(G)$, hence in $\C_0(G)$. Note that if $f$ belongs to $\C_0(G) \cap \dom \Delta_p^{\frac\alpha2}$ and satisfies $\norm{f}_{\L^p_\alpha(G)} \leq 1$ and if $g \in \C_c(G)$, we have
\begin{align*}
\MoveEqLeft
\norm{gf}_{\L^p_\alpha(G)}          
\ov{\eqref{Leibniz-Coulhon-bis}}{\lesssim_{p}} \norm{g}_{\L^p_\alpha(G)} \norm{f}_{\L^\infty(G)} + \norm{g}_{\L^\infty(G)} \norm{f}_{\L^p_\alpha(G)} \\ 
&\ov{\eqref{Sobolev-embedding-bis}}{\lesssim} \norm{f}_{\L^p_\alpha(G)} \Big[\norm{g}_{\L^\infty(G)} +  \norm{g}_{\L^p_\alpha(G)}\Big] 
\leq \norm{g}_{\L^\infty(G)} +  \norm{g}_{\L^p_\alpha(G)}.
\end{align*}
Consequently, if $\supp g \subset K$, we obtain that the subset \eqref{equ-proof-prop-Schur-quantum-compact-metric-3-bis} is included in some subset $F_{K,p,M}$ with $M \ov{\mathrm{def}}{=} \norm{g}_{\L^\infty(G)} +  \norm{g}_{\L^p_\alpha(G)}$. 
\end{proof}

\begin{remark} \normalfont
In \cite[Proposition 3 p.~138]{Ste70}, it is proved that a measurable function $f$ belongs to the space $\L^p_1(\R^n)$ if and only if $\norm{(\Id-\lambda_s)f}_{\L^p(\R^n)}=O(|s|)$. So, we are not confident in a possible generalization of Proposition \ref{Prop-loc-1} to the case $0<\alpha<1$. Note also that in \cite[Ex 6.1 p.~159]{Ste70}, it is stated that a measurable function $f$ belongs to the space $\L^p_1(\R^n)$ if and only if $f$ belongs to $\L^p(\R^n)$, $f$ is absolutely continuous and the partial derivatives $\frac{\partial f}{\partial x_1},\ldots,\frac{\partial f}{\partial x_n}$ belong to $\L^p(\R^n)$.
\end{remark}

\begin{remark} \normalfont
It is transparent for the author that some form of local ultracontractivity \cite[Definition 2.11]{GrT12} can be used to give some variants or generalizations of the previous proof to other contexts.
\end{remark}

We also define the seminorm $\norm{\cdot}_{\L^p_\alpha(G)}$ on the space $(\C_0(G) \cap \dom \Delta_p^{\frac{\alpha}{2}}) \oplus \mathbb{C}1$ by letting $\norm{f}_{\L^p_\alpha(G)} \ov{\mathrm{def}}{=} \norm{f_0}_{\L^p_\alpha(G)}$ for any element $f=f_0+ \lambda 1$ of the space $(\C_0(G) \cap \dom \Delta_p^{\frac{\alpha}{2}}) \oplus \Cbb 1$. Note that the latter space is a subspace of the unitization $\C_0(G) \oplus \Cbb 1$ of the non-unital algebra $\C_0(G)$. 


\begin{lemma}
\label{Lemma-lower-5}
The restriction of the seminorm $\norm{\cdot}_{\L^p_\alpha(G)}$ on the subspace $(\C_0(G) \cap \dom \Delta_p^{\frac{\alpha}{2}})$ is lower semicontinuous.
\end{lemma}

\begin{proof}
Let $f \in \C_0(G)$ and $(f_n)$ be a sequence of elements of $\C_0(G) \cap \dom \Delta^{\frac\alpha2}_p$ such that $(f_n)$ converges to $f$ for the norm topology of $\C_0(G)$ and $\norm{f_n}_{\L^p_\alpha(G)} \leq 1$ for any $n$, that is $\bnorm{\Delta_p^{\frac{\alpha}{2}}(f_n)}_{\L^p(G)}+\norm{f_n}_{\L^p(G)} \leq 1$ by \eqref{Def-Lpalpha-bis}. Note that in particular that the sequences $(f_n)$ and $(\Delta_p^{\frac{\alpha}{2}}(f_n))$ are bounded in the Banach space $\L^p(G)$. We have to prove that the function $f$ belongs to $\dom \Delta^{\frac\alpha2}_p$ and that $\norm{f}_{\L^p_\alpha(G)} \leq 1$. 

First, we show that the sequence $(f_n)$ converges to $f$ for the weak topology of the Banach space $\L^p(G)$. Indeed, for any function $g \in \C_c(G)$, we have
$$
\left|\int_G (f_n-f)g \d \mu_G\right| 
\leq \norm{f_n-f}_{\L^\infty(G)} \int_G |g| \d \mu_G
\xra[n \to +\infty]{} 0.
$$
Using the boundedness of the sequence $(f_n)$ in $\L^p(G)$, we obtain the claim with \cite[2.71 p.~234]{Meg98} since we have the convergence with any function $g$ of the dense subspace $\C_c(G)$ of the Banach space $\L^{p^*}(G)$.

Since the sequence $\big(\Delta_p^{\frac{\alpha}{2}}f_n\big)$ is bounded in the Banach space $\L^p(G)$ and since bounded sets are weakly relatively compact by \cite[Theorem 2.8.2]{Meg98}, there exists a weakly convergent subnet $\big(\Delta_p^{\frac{\alpha}{2}}f_{n_j}\big)$. Then $\big(f_{n_j},\Delta_p^{\frac{\alpha}{2}}f_{n_j}\big)$ is a weakly convergent net in the graph of the closed operator $\Delta_p^{\frac{\alpha}{2}}$. Note that this graph is closed and convex, hence weakly closed by \cite[Theorem 2.5.16]{Meg98}. Thus the limit of $\big(f_{n_j},\Delta_p^{\frac{\alpha}{2}}f_{n_j}\big)$ belongs again to the graph and is of the form $\big(g,\Delta_p^{\frac{\alpha}{2}}g\big)$ for some $g \in \dom \Delta_p^{\frac{\alpha}{2}}$. In particular, the net $(f_{n_j})$ converges weakly to $g$ and $\Delta_p^{\frac{\alpha}{2}}(f_{n_j})$ converges weakly to $\Delta_p^{\frac{\alpha}{2}}(g)$. We infer that $f=g$. We conclude that $f$ belongs to $\dom \Delta_p^{\frac{\alpha}{2}}$. Moreover, using the weakly lower semicontinuity of the norm \cite[Theorem 2.5.21]{Meg98}, we obtain
\begin{align*}
\MoveEqLeft
\norm{f}_{\L^p_\alpha(G)}
\ov{\eqref{Def-Lpalpha-bis}}{=} \bnorm{\Delta_p^{\frac{\alpha}{2}}(f)}_{\L^p(G)}+\norm{f}_{\L^p(G)} \leq \liminf_j \Big[\bnorm{\Delta_p^{\frac{\alpha}{2}}(f_{n_j})}_{\L^p(G)}+\liminf_j\norm{f_{n_j}}_{\L^p(G)}\Big]
\leq 1.
\end{align*}
\end{proof}

\begin{cor}
\label{Cor-global-1}
Let $G$ be a non-compact connected Lie group equipped with a family $(X_1, \ldots, X_m)$ of left-invariant H\"ormander vector fields. Suppose that $G$ has polynomial volume growth. Let $\alpha \geq 1$ and $\max\{1,\frac{d}{\alpha}\} < p < \infty$ where $d$ is the local dimension defined in \eqref{local-dim}. Then $\big(\C_0(G),\norm{\cdot}_{\L^p_\alpha(G)},\C_0(G)\big)$ defines a $(C_{\alpha,p},0)$-quasi-Leibniz quantum locally compact metric space for some constant $C_{\alpha,p}>0$.
\end{cor}

\begin{proof}
Parts 1 and 2 of Lemma \ref{Lemma-recapitulatif} say that $\dom \norm{\cdot}_{\dot\L^p_\alpha(G)}\ov{\eqref{def-domaine}}{=} \C_0(G) \cap \dom \Delta_p^{\frac{\alpha}{2}}$ is closed under the adjoint operation and dense in the space $\C_0(G)$. Consequently $(\C_0(G) \cap \dom \Delta_p^{\frac{\alpha}{2}}) \oplus \mathbb{C}1$ is also closed under the adjoint operation of the algebra $\C_0(G) \oplus \Cbb 1$ and dense in $\C_0(G) \oplus \Cbb 1$.

Let $f=f_0+ \lambda 1$ be an element of $(\C_0(G) \cap \dom \Delta_p^{\frac{\alpha}{2}}) \oplus \mathbb{C}1$. Suppose that $\norm{f}_{\L^p_\alpha(G)}=0$. Then by definition 
$$
\bnorm{\Delta_p^{\frac{\alpha}{2}}(f_0)}_{\L^p(G)}+\norm{f_0}_{\L^p(G)} 
\ov{\eqref{Def-Lpalpha-bis}}{=} \norm{f_0}_{\L^p_\alpha(G)}
=0.
$$ 
Hence $\norm{f_0}_{\L^p(G)}=0$ and finally $f_0=0$. We conclude that $f =\lambda 1$. So \eqref{point1} is satisfied. So we have a Lipschitz pair $(\C_0(G),\norm{\cdot}_{\L^p_\alpha(G)})$.

This Dirac measure $\delta_e$ is clearly a local state since it is supported by the compact $\{e\}$. The Leibniz rule is given by \eqref{Leibniz-Coulhon-bis}. The lower semicontinuity is given by Lemma \ref{Lemma-lower-5}. We conclude with Proposition \ref{Prop-loc-1}.
\end{proof}

In the end of this section, we will investigate what happens when we replace the operator $\Id+\Delta_p$ by the subelliptic Laplacian $\Delta_p$ in one case. The obtained result of Proposition \ref{Prop-less} is a bit different. Indeed, it is obvious that the addition of the identity to the operator $\Delta$ removes global phenomenons.

Suppose that the connected Lie group $G$ is equipped with a family $(X_1, \ldots, X_m)$ of left-invariant H\"ormander vector fields and has polynomial growth with $d<D$. Such group is not compact. For example by \cite[p.~273]{Rob91}, this condition is satisfied if $G$ is simply connected, nilpotent with $G \not\approx \R^d$. 
Consider some $1 < p < \infty$ and some $\alpha >0$. If $d<\alpha p< D$, it is stated in \cite[p.~288]{CRT01} and \cite[p.~197]{CoS91} that
\begin{equation}
\label{Sobolev-embedding}
\norm{f}_{\L^\infty(G)} 
\lesssim \norm{f}_{\dot{\L}_\alpha^p(G)},\quad f \in \C_c^\infty(G).
\end{equation}
Now, we prove an analogue of Proposition \ref{Prop-loc-1}.

\begin{prop}
\label{Prop-less}
Let $G$ be a connected Lie group equipped with a family $(X_1, \ldots, X_m)$ of left-invariant H\"ormander vector fields. Suppose that $G$ has polynomial growth with $d<D$. Assume that $d<p<D$. If $g \co G \to \Cbb$ is a compactly supported continuous function then the subset
\begin{equation}
\label{equ-proof-prop-Schur-quantum-compact-metric-3}
g \big\{f \in \C_0(G) \cap \dom \Delta_p^{\frac12}  : \norm{f}_{\dot{\L}^p_1(G)} \leq 1, f(e)=0 \big\}
\end{equation}
is relatively compact in $\L^\infty(G)$.
\end{prop}

\begin{proof}
Let $K$ be a compact subset of $G$. For any $M \geq 0$, consider the subset
\begin{equation}
\label{Interminable-subset}
E_{K,p,M} 
\ov{\mathrm{def}}{=} \Big\{ f \in \C_K(G) \cap \dom \Delta_p^{\frac{1}{2}} : \norm{f}_{\dot{\L}^p_1(G)} \leq M \Big\} 
\end{equation} 
of the space $\L^p(G)$. If $f \in E_{K,p,M}$, using the Sobolev embedding $\dot{\L}^p_1(G) \subset \L^\infty(G)$ of \eqref{Sobolev-embedding}, we obtain 
$$
\norm{f}_{\L^p(G)} 
\lesssim_{K,p} \norm{f}_{\L^\infty(G)} 
\ov{\eqref{Sobolev-embedding}}{\lesssim} \norm{f}_{\dot{\L}_1^p(G)} 
\leq M.
$$ 
We infer that the subset $E_{K,p,M}$ is bounded in $\L^p(G)$. Furthermore, using Lemma \ref{Lemma-useful-estimate}, we have for any function $f \in E_{K,p,M}$
\begin{align*}
\MoveEqLeft
\norm{(\Id-\lambda_s)f}_{\L^p(G)}
\ov{\eqref{Ine-pratique}}{\leq} \dist_\CC^{p^*}(s,e) \bigg(\sum_{k=1}^m \norm{X_{k,p}(f)}_{\L^p(G)}^p \bigg)^{\frac{1}{p}}
\approx_p \dist_\CC^{p^*}(s,e) \sum_{k=1}^{m} \bnorm{X_{k,p}(f)}_{\L^p(G)} \\
&\ov{\eqref{Riesz1}}{\lesssim_{p}}\dist_\CC^{p^*}(s,e) \bnorm{\Delta_p^{\frac12}(f)}_{\L^p(G)} 
\ov{\eqref{Def-Lpalpha}}{=} \dist_\CC^{p^*}(s,e) \norm{f}_{\dot{\L}^p_1(G)}
\ov{\eqref{Interminable-subset}}{\leq} M\dist_\CC^{p^*}(s,e).
\end{align*} 
By Theorem \ref{Th-compactness-Rob}, the subset $E_{K,p,M}$ is relatively compact in $\L^p(G)$. Hence, its subset
$$
F_{K,p,M} 
\ov{\mathrm{def}}{=} \Big\{ f \in \C_K(G) \cap \dom \Delta_p^{\frac{1}{2}} : \norm{f}_{\dot{\L}_1^p(G)} \leq M, f(e)=0 \Big\} 
$$
is also relatively compact in $\L^p(G)$. The operator $\Delta^{-\frac12} \co \L^p(G) \to \ovl{\Ran \Delta_\infty}$ is bounded by \eqref{Sobolev-embedding}, hence uniformly continuous. Applying this operator to the previous subset by writing $f=\Delta^{-\frac12}\Delta_p^{\frac12}f$ we obtain that the set $F_{K,p,M}$ is relatively compact in $\L^\infty(G)$, hence in the space $\C_0(G)$. Note that we have
\begin{align*}
\MoveEqLeft
\norm{gf}_{\dot{\L}^p_1(G)}          
\ov{\eqref{Leibniz-Coulhon}}{\lesssim_{p}} \norm{g}_{\dot{\L}^p_1(G)} \norm{f}_{\L^\infty(G)} + \norm{g}_{\L^\infty(G)} \norm{f}_{\dot{\L}^p_1(G)} \\ 
&\ov{\eqref{Sobolev-embedding}}{\lesssim} \norm{f}_{\L^p_1(G)} \Big[\norm{g}_{\L^\infty(G)} +  \norm{g}_{\dot{\L}^p_1(G)}\Big] 
\ov{\eqref{Interminable-subset}}{\lesssim} \norm{g}_{\L^\infty(G)} +  \norm{g}_{\dot{\L}^p_1(G)}.
\end{align*}
Consequently, if $\supp g \subset K$, we obtain that the subset \eqref{equ-proof-prop-Schur-quantum-compact-metric-3} is included in some subset $F_{K,p,M}$ with $M \ov{\mathrm{def}}{=} \norm{g}_{\L^\infty(G)} +  \norm{g}_{\dot{\L}^p_1(G)}$. 
\end{proof}

Similarly to Corollary \ref{Cor-global-1}, we obtain the following result where the seminorm $\norm{\cdot}_{\dot{\L}^p_\alpha(G)}$ is defined on $\dom \norm{\cdot}_{\dot{\L}^p_1(G)} \ov{\eqref{def-domaine}}{=} \C_0(G) \cap \dom \Delta_p^{\frac{1}{2}}$. Unfortunately, we are not able to prove the lower semicontinuity of the seminorm $\norm{\cdot}_{\dot{\L}^p_1(G)}$. So we cannot make the statement that we have a \textit{quasi-Leibniz} quantum locally compact metric space.

\begin{cor}
\label{Cor-global-2} 
Let $G$ be a connected Lie group equipped with a family $(X_1, \ldots, X_m)$ of left-invariant H\"ormander vector fields. Suppose that $G$ has polynomial growth with $d<D$. Then the triple $\big(\C_0(G),\norm{\cdot}_{\dot{\L}^p_1(G)},\C_0(G)\big)$ defines a quantum locally compact metric space.
\end{cor}

\section{Compact spectral triples and spectral dimension}
\label{Sec-spectral-triples}

\paragraph{Possibly kernel-degenerate compact spectral triples}
Consider a triple $(A,Y,\slashed{D})$ constituted of the following data: a Banach space $Y$, a closed unbounded operator $\slashed{D}$ on $Y$ with dense domain $\dom \slashed{D} \subset Y$, an algebra $A$ equipped with a homomorphism $\pi \co A \to \B(Y)$. In this case, we define the Lipschitz algebra
\begin{align}
\label{Lipschitz-algebra-def}
\MoveEqLeft
\Lip_\slashed{D}(A) 
\ov{\mathrm{def}}{=} \big\{a \in A : \pi(a) \cdot \dom \slashed{D} \subset \dom \slashed{D} 
\text{ and the unbounded operator } \\
&\qquad \qquad  [\slashed{D},\pi(a)] \co \dom \slashed{D} \subset Y \to Y \text{ extends to an element of } \B(Y) \big\}. \nonumber
\end{align} 
We say that $(A,Y,\slashed{D})$ is a (possibly kernel-degenerate) compact spectral triple if in addition $Y$ is a Hilbert space $H$, $A$ is a $\C^*$-algebra, $D$ is a selfadjoint operator on $Y$ and if we have  
\begin{enumerate}
\item{} $\slashed{D}^{-1}$ is a compact operator on $\ovl{\Ran \slashed{D}} \ov{\eqref{lien-ker-image}}{=} (\ker \slashed{D})^\perp$,
\item{} the subset $\Lip_\slashed{D}(A)$ is dense in $A$.
\end{enumerate}


We essentially follow \cite[Definition 2.1]{CGIS14} and \cite[Definition 5.10]{ArK22}. Note that there exists different variations of this definition in the literature, see e.g. \cite[Definition 1.1]{EcI18}. Moreover, we can replace $\slashed{D}^{-1}$ by $|\slashed{D}|^{-1}$ in the first point by an elementary functional calculus argument.

We equally refer to \cite[Definition 5.10]{ArK22} for the notion of compact Banach spectral triple which is a generalization for the case of an operator $\slashed{D}$ acting on a Banach space $Y$ instead of a Hilbert space $H$.

\begin{example} \normalfont
If $M$ is a compact oriented Riemannian manifold $M$, we can associate the spectral triple $(\C(M),\L^2(\wedge T^*M),D)$ where $\L^2(\wedge T^*M)$ is the Hilbert space of square-integrable complex-valued forms on $M$ and where $D$ is the Hodge-Dirac operator (also called Hodge-de Rham operator). If $M$ is in addition a spin manifold, we can also consider the spectral triple $(\C(M),\L^2(M,S),\scr{D})$ obtained by using the Hilbert space $\L^2(M,S)$ the space of square-integrable spinors on $M$, and the Dirac operator $\scr{D}$. In both cases, the functions of $\C(M)$ act on the Hilbert space by multiplication operators.
\end{example}

\paragraph{Spectral dimension}
Let $(A,H,\slashed{D})$ be a compact spectral triple. By \cite[Proposition 5.3.38]{Ped89}, we have $\ker |\slashed{D}|=\ker \slashed{D}$. Moreover, the operator $|\slashed{D}|^{-1}$ is well-defined on $\ovl{\Ran \slashed{D}} \ov{\eqref{lien-ker-image}}{=} (\ker \slashed{D})^\perp$. Furthermore, we can extend it by letting $|\slashed{D}|^{-1}=0$ on $\ker \slashed{D}$. Following \cite[p.~4]{EcI18}, we say that a compact spectral triple $(A,H,\slashed{D})$ is $\alpha$-summable for some $\alpha > 0$ if $\tr |\slashed{D}|^{-\alpha} < \infty$, 
that is if the operator $|\slashed{D}|^{-1}$ belongs to the Schatten class $S^\alpha(H)$. In this case, the spectral dimension of the spectral triple is defined by
\begin{equation}
\label{Def-spectral-dimension}
\dim(A,H,\slashed{D})  \ov{\mathrm{def}}{=} \inf\big\{\alpha > 0 : \tr |\slashed{D}|^{-\alpha} < \infty\big\}.
\end{equation}
See also \cite[p.~450]{GVF01} and \cite[p.~38 and Definition 6.2 p.~47]{CPR11} for a variation of this definition. 

We will use the following lemma which is a slight variation of \cite[Lemma 10.8 p.~450]{GVF01}.

\begin{lemma} 
\label{Lemma-theta-summable}
If $|\slashed{D}|^{-\alpha}$ is trace-class then for any $t>0$ the operator $\e^{-t\slashed{D}^2}$ is trace-class and we have
\begin{equation}
\label{ul-esti2}
\tr \e^{-t\slashed{D}^2}
\lesssim \frac{1}{t^{\frac{\alpha}{2}}}, \quad t>0.
\end{equation}
\end{lemma}

\begin{proof}
Note that here the operator $|\slashed{D}|^{-\alpha}$ is defined and bounded on $\ovl{\Ran \slashed{D}}$. However, we can extend it by letting $|\slashed{D}|^{-\alpha}=0$ on $\ker \slashed{D}$. For any $t>0$, we have $\e^{-t\slashed{D}^2}=|\slashed{D}|^{\alpha}\e^{-t\slashed{D}^2}|\slashed{D}|^{-\alpha}$. An elementary study of the function $f \co \lambda \mapsto \lambda^{\alpha}\e^{-t\lambda^2}$ on $\R^+$ shows that $f'(\lambda)=\lambda^{\alpha-1}\e^{-t\lambda^2} (\alpha-2\lambda^2 t)$ for any $\lambda \geq 0$ and consequently that $f$ is bounded and that its maximum is $(\frac{\alpha}{2t})^{\frac{\alpha}{2}}\e^{-\frac{\alpha}{2}}$ in $\lambda=\sqrt{\frac{\alpha}{2 t}}$. We conclude by functional calculus that the operator $|\slashed{D}|^{\alpha}\e^{-t\slashed{D}^2}$ is bounded and that
$$
\tr \e^{-t\slashed{D}^2}
=\bnorm{\e^{-t\slashed{D}^2}}_{S^1(H)}
\leq \bnorm{|\slashed{D}|^{\alpha}\e^{-t\slashed{D}^2}}_{\B(H)} \norm{|\slashed{D}|^{-\alpha}}_{S^1(H)}
\lesssim \frac{1}{t^{\frac{\alpha}{2}}}.
$$
\end{proof}

\paragraph{Hodge-Dirac operator}
Let $G$ be a unimodular connected Lie group equipped with a family $(X_1, \ldots, X_m)$ of left-invariant H\"ormander vector fields and consider a Haar measure $\mu_G$ on $G$. Suppose $1 \leq p \leq \infty$. Recall that we have a canonical isometry $\ell^p_m(\L^p(G))=\L^p(G,\ell^p_m)$. We define the unbounded closed operator $\nabla_p$ from $\L^p(G)$ into the space $\L^p(G,\ell^p_m)$ by $\dom \nabla_p =\dom X_{1,p} \cap \cdots \cap \dom X_{m,p}$ and 
\begin{equation}
\label{Def-grad}
\nabla_p f
\ov{\mathrm{def}}{=} \big(X_{1,p} f, \ldots, X_{m,p} f\big), \quad f \in \dom \nabla_p.
\end{equation} 
If $1<p<\infty$, note that $\dom \nabla_p=\dom \Delta_p^{\frac12}$ by Proposition \ref{Prop-domain}. For any functions $f,g$ of $\dom \nabla_p \cap \L^\infty(G)$ then $fg$ belongs to $\dom \nabla_p \cap \L^\infty(G)$ and we have
\begin{equation}
\label{Leib-gradient}
\nabla_p(fg)            
=g\cdot\nabla_p(f) +f\cdot\nabla_p(g) 
\end{equation}
where $f\cdot (h_1,\ldots,h_m) \ov{\mathrm{def}}{=}(fh_1,\ldots,fh_m)$. See \cite[p.~289]{CRT01} for a generalization. 

If $1<p<\infty$, we introduce the unbounded closed operator
\begin{equation}
\label{Hodge-Dirac-I}
\slashed{D}_p
\ov{\mathrm{def}}{=}
\begin{bmatrix} 
0 & (\nabla_{p^*})^* \\ 
\nabla_p & 0 
\end{bmatrix}.
\end{equation}
on the Banach space $\L^p(G) \oplus_p \L^p(G,\ell^p_m)$ defined by
\begin{equation}
\label{Def-D-psi}
\slashed{D}_{p}(f,g)
\ov{\mathrm{def}}{=}
\big((\nabla_{p^*})^*(g),
\nabla_{p}(f)\big), \quad f \in \dom \nabla_{p},\  g \in \dom (\nabla_{p^*})^*.
\end{equation}
We call it the Hodge-Dirac operator of the subelliptic Laplacian $\Delta=-(X_1^2+\cdots+X_m^2)$. These operators are related by the computation
\begin{equation}
\label{carre-de-D}
\slashed{D}_{p}^2
\ov{\eqref{Hodge-Dirac-I}}{=}
\begin{bmatrix} 
0 & (\nabla_{p^*})^* \\ 
\nabla_p & 0 
\end{bmatrix}^2
=\begin{bmatrix} 
(\nabla_{p^*})^*\nabla_p & 0 \\ 
0 & \nabla_p (\nabla_{p^*})^*
\end{bmatrix}
=\begin{bmatrix} 
\Delta_p & 0 \\ 
0 & \nabla_p (\nabla_{p^*})^*
\end{bmatrix}. 
\end{equation} 
The operator $\slashed{D}_2$ is identical to the operator $\Pi$ of \cite[proof of Theorem 1.2]{BEM13} with $b=\Id$. 

We will use just the following lemma which describes a tractable subspace for the adjoint operator $(\nabla_{p^*})^*$. If $(\varphi_j)$ is a Dirac net of functions of $\C^\infty_c(G)$ and if $h=(h_1,\ldots,h_m)$, we will use the notation $\Reg_j h \ov{\mathrm{def}}{=}(h_1*\varphi_j,\ldots,h_m*\varphi_j)$ as soon as it makes sense.

\begin{lemma}
\label{Lemma-core}
Suppose $1<p<\infty$. The subspace $\C^\infty_c(G) \ot \ell^p_m$ is a core of the unbounded operator $(\nabla_{p^*})^*$.
\end{lemma}

\begin{proof}
Is is easy to check (use \cite[Problem 5.24 p.~168]{Kat76}) that $\C^\infty_c(G) \ot \ell^p_m$ is a subset of $\dom (\nabla_{p^*})^*$. We consider a Dirac net $(\varphi_j)$ of functions of $\C^\infty_c(G)$. Let $h=(h_1,\ldots,h_m)$ be an element of $\dom(\nabla_{p^*})^*$. Then $\Reg_j h = (h_1*\varphi_j,\ldots,h_m*\varphi_j)$ belongs to $\C^\infty_c(G) \ot \ell^p_m$. It remains to show that $(\Reg_j h)$ converges to $h$ in the graph norm of $(\nabla_{p^*})^*$. By \cite[Proposition 20  VIII.44]{Bou04}, the net $(\Reg_j h)$ converges to $h$ in $\L^p(G,\ell^p_m)$. For any $1 \leq k \leq m$, we put $a_k \ov{\mathrm{def}}{=} X_k(e)$. If $f \in \C_c^\infty(G)$, using \cite[(9.19)]{Sch20} and the equalities $X_{k}=\d \lambda(a_k)$ \cite[p.~14]{DtER03}
in the second equality and \cite[Proposition 3.14]{Mag92} 
in the third equality, we have
\begin{align}
\MoveEqLeft
\label{commute-tronc}
\nabla\Reg_j f
\ov{\eqref{Def-grad}}{=} \big(X_{1} (f*\varphi_j), \ldots, X_{m} (f*\varphi_j)\big)  
=\big(\d \lambda(a_1) (\lambda_f(\varphi_j)), \ldots, \d \lambda(a_m) (\lambda_f(\varphi_j))\big)\\
&=\big(\lambda_{X_1f}(\varphi_j), \ldots, \lambda_{X_mf}(\varphi_j)\big) 
=\big( (X_{1}f)*\varphi_j, \ldots, (X_{m}f)*\varphi_j\big) \nonumber\\
&=\Reg_j(X_{1} f, \ldots, X_{m} f) 
=\Reg_j(\nabla f) \nonumber           
\end{align} 
where $\lambda_f(g) \ov{\mathrm{def}}{=}f*g$ and where $\d\lambda$ is the derived representation \cite[Definition 3.12]{Mag92} of the left regular representation $\lambda$. Moreover, for any $g \in \C_c^\infty(G)$, we have using \cite[(14.10.9)]{Dieu83} in the second and the last inequalities
\begin{align*}
\MoveEqLeft
\big\langle (\nabla_{p^*})^* \Reg_jh, g \big\rangle_{\L^p(G),\L^{p^*}(G)} 
=\big\langle \Reg_j h, \nabla_{p^*}g \big\rangle_{\L^p(G,\ell^p_m),\L^{p^*}(G,\ell^{p^*}_m)} 
=\big\langle h, \Reg_j(\nabla_{p^*}g) \big\rangle \\
&\ov{\eqref{commute-tronc}}{=} \big\langle h, \nabla_{p^*} (\Reg_j g) \big\rangle
=\big\langle (\nabla_{p^*})^*(h), \Reg_j g \big\rangle 
=\big\langle \Reg_j (\nabla_{p^*})^*(h) , g \big\rangle_{\L^p(G),\L^{p^*}(G)}
\end{align*}
where here we use the bracket $\langle f, g \rangle_{\L^p(G),\L^{p^*}(G)}=\int_G f(s)g(s^{-1}) \d s$. Note that the use of the inversion map $G \to G$, $s \mapsto s^{-1}$ in the bracket simplifies \cite[(14.10.9)]{Dieu83}. By density and duality, we infer that $
(\nabla_{p^*})^*\Reg_j h 
= \Reg_j ((\nabla_{p^*})^*h)$ which converges to $(\nabla_{p^*})^*(h)$ in $\L^p(G)$.
\end{proof}


If $f \in \L^\infty(G)$, we define the bounded operator $\pi(f) \co \L^p(G) \oplus_p \L^p(G,\ell^p_m) \to \L^p(G) \oplus_p \L^p(G,\ell^p_m)$ by
\begin{equation}
\label{Def-pi-a}
\pi(f)
\ov{\mathrm{def}}{=} \begin{bmatrix}
    \M_f & 0  \\
    0 & \tilde{\M}_f  \\
\end{bmatrix}, \quad f \in \L^\infty(G)
\end{equation}
where the linear map $\M_f \co \L^p(G) \to \L^p(G)$, $g \mapsto fg$ is the multiplication operator by the function $f$ and where 
$$
\tilde{\M}_f \ov{\mathrm{def}}{=} \Id_{\ell^p_m} \ot \M_f  \co \ell^p_m(\L^p(G)) \to \ell^p_m(\L^p(G)),\, (h_1,\ldots,h_m) \mapsto (fh_1,\ldots,fh_m)
$$ 
is also a multiplication operator (by the function $(f,\ldots,f)$ of $\ell^\infty_m(\L^\infty(G))$). Using \cite[Proposition 4.10 p.~31]{EnN00}, it is (really) easy to check that $\pi \co \L^\infty(G) \to \B(\L^p(G) \oplus_p \L^p(G,\ell^p_m))$ is an isometric homomorphism. Moreover, it is obviously continuous when the algebra $\L^\infty(G)$ is equipped with the weak* topology and when the space $\B(\L^p(G)) \oplus_p \L^p(G,\ell^p_m)$ is equipped with the weak operator topology. Note $\B(\L^p(G) \oplus_p \L^p(G,\ell^p_m))$ is a dual Banach space whose predual is the projective tensor
product $\L^p(G) \oplus_p \L^p(G,\ell^p_m)) \otpb (\L^{p^*}(G) \oplus_{p^*} \L^{p^*}(G,\ell^{p^*}_m))$. Using \cite[Theorem A.2.5 (2)]{BLM04}, it is not difficult to prove that $\pi$ is \textit{even} weak* continuous when we equip the space $\B(\L^p(G) \oplus_p \L^p(G,\ell^p_m))$ with the weak* topology.

\begin{prop}
\label{First-spectral-triple}
Let $G$ be a unimodular connected Lie group equipped with a family $X$ of left-invariant H\"ormander vector fields. Consider a Haar measure $\mu_G$ on $G$. Suppose $1<p<\infty$.
\begin{enumerate}
\item We have $(\slashed{D}_p)^*=\slashed{D}_{p^*}$. In particular, the unbounded operator $\slashed{D}_2$ is selfadjoint.

\item We have
\begin{equation}
\label{Lip-algebra-I-Schur}
\dom \nabla_{\infty}
\subset \Lip_{\slashed{D}_p}(\L^\infty(G)).
\end{equation}

\item For any $f \in \dom \nabla_{\infty}$, we have 
\begin{equation}
\label{norm-commutator}
\norm{\big[\slashed{D}_p,\pi(f)\big]}_{\L^p(G) \oplus_p \L^p(G,\ell^p_m) \to \L^p(G) \oplus_p \L^p(G,\ell^p_m)}
= \bnorm{\nabla_{\infty}(f)}_{\L^\infty(G,\ell^p_m)}.
\end{equation}

\end{enumerate}
\end{prop}

\begin{proof}
1. By definition, an element $(z,t)$ of the Banach space $\L^{p^*}(G) \oplus_{p^*} \L^{p^*}(G,\ell^{p^*}_m)$ belongs to $\dom (\slashed{D}_{p})^*$ if and only if there exists $(h,k) \in \L^{p^*}(G) \oplus_{p^*} \L^{p^*}(G,\ell^{p^*}_m)$ such that for any $(f,g) \in \dom \nabla_{p} \oplus \dom (\nabla_{p^*})^*$ we have
\begin{align*}
\MoveEqLeft
\left\langle \begin{bmatrix} 
0 & (\nabla_{p^*})^* \\ 
\nabla_{p} & 0 
\end{bmatrix}\begin{bmatrix}
      f  \\
      g  \\
\end{bmatrix},\begin{bmatrix}
      z  \\
      t  \\
\end{bmatrix}\right\rangle            
=\left\langle 
\begin{bmatrix}
    f    \\
    g    \\
\end{bmatrix},\begin{bmatrix}
    h   \\
    k    \\
\end{bmatrix}\right\rangle, 
\end{align*}
that is 
\begin{equation}
\label{Relation99-group}
\big\langle (\nabla_{p^*})^*(g), z\big\rangle
+\big\langle \nabla_{p}(f),t \big\rangle
=\langle g,k\rangle+\langle f,h\rangle. 
\end{equation}
If $z \in \dom \nabla_{p^*}$ and if $t \in \dom (\nabla_{p})^*$ the latter holds with $k=\nabla_{p^*}(z)$ and $h=(\nabla_{p})^*(t)$. This proves that $\dom \nabla_{p^*} \oplus \dom (\nabla_{p})^* \subset \dom (\slashed{D}_{p})^*$ and that 
\begin{align*}
\MoveEqLeft
(\slashed{D}_{p})^*(z,t)
=\big((\nabla_{p})^*(t),\nabla_{p^*}(z)\big) 
= \begin{bmatrix} 
0 & (\nabla_{p})^* \\ 
\nabla_{p^*} & 0 
\end{bmatrix}
\begin{bmatrix}
    z    \\
    t    \\
\end{bmatrix} 
\ov{\eqref{Def-D-psi}}{=} \slashed{D}_{p^*}(z,t).            
\end{align*} 
Conversely, if $(z,t) \in \dom (\slashed{D}_{p})^*$, choosing $g=0$ in \eqref{Relation99-group} we obtain $t \in \dom (\nabla_{p})^*$ and taking $f=0$ we obtain $z \in \dom \nabla_{p^*}$. 

2. Let $f \in \C^\infty_c(G)$. A standard calculation shows that
\begin{align*}
\MoveEqLeft
\big[\slashed{D}_{p},\pi(f)\big]
\ov{\eqref{Hodge-Dirac-I}\eqref{Def-pi-a}}{=}\begin{bmatrix} 
0 & (\nabla_{p^*})^* \\ 
\nabla_{p}& 0
\end{bmatrix}\begin{bmatrix}
    \M_f & 0  \\
    0 & \tilde{\M}_f  \\
\end{bmatrix}
-\begin{bmatrix}
    \M_f & 0  \\
    0 & \tilde{\M}_f  \\
\end{bmatrix}
\begin{bmatrix} 
0 & (\nabla_{p^*})^* \\ 
\nabla_{p}& 0
\end{bmatrix}\\
&=\begin{bmatrix}
   0 &  (\nabla_{p^*})^*\tilde{\M}_f \\
   \nabla_{p}\M_f  &0  \\
\end{bmatrix}-\begin{bmatrix}
    0 &  \M_f (\nabla_{p^*})^*\\
    \tilde{\M}_f\nabla_{p} & 0  \\
\end{bmatrix} \\
&=\begin{bmatrix}
    0 &  (\nabla_{p^*})^*\tilde{\M}_f-\M_f (\nabla_{p^*})^* \\
   \nabla_{p}\M_f-\tilde{\M}_f\nabla_{p}  & 0  \\
\end{bmatrix}.
\end{align*}
We calculate the two non-zero entries of the commutator. For the lower left corner, if $g \in \C^\infty_c(G)$ we have
\begin{align}
\label{Bon-commutateur}
\MoveEqLeft
(\nabla_{p}\M_f-\tilde{\M}_f\nabla_{p})(g)           
=\nabla_{p}\M_f(g)-\tilde{\M}_f\nabla_{p}(g)
=\nabla_{p}(fg)-f\cdot\nabla_{p}(g) \\
&\ov{\eqref{Leib-gradient}}{=} g\cdot \nabla_p(f)
= \big(gX_{1,m}(f),\ldots,gX_{m,p}(f)\big)
=\M_{\nabla f}J(g) \nonumber
\end{align}
where $J \co \L^p(G) \to \ell^p_m(\L^p(G))$, $g \mapsto (g,\ldots,g)$ and where $\M_{\nabla(f)}$ is the multiplication operator on the $\L^p$-space $\ell^p_m(\L^p(G))$ by $\nabla(f)$. For the upper right corner, note that for any $h \in \C^\infty_c(G) \ot \ell^p_m$ and any $g \in \C^\infty_c(G)$ we have
\begin{align*}
\MoveEqLeft
\big\langle \big((\nabla_{p^*})^*\tilde{\M}_f-\M_f (\nabla_{p^*})^*\big)(h),g \big\rangle_{}            
=\big\langle (\nabla_{p^*})^*\tilde{\M}_f(h),g \big\rangle -\big\langle \M_f (\nabla_{p^*})^*(h),g \big\rangle_{}\\
&=\big\langle \tilde{\M}_f(h),\nabla_{p^*}(g) \big\rangle -\big\langle  (\nabla_{p^*})^*(h),\M_{f}(g) \big\rangle_{}
=\big\langle h,\tilde{\M}_{f}\nabla_{p^*}(g) \big\rangle -\big\langle h,\nabla_{p^*}\M_{f}(g) \big\rangle_{} \\
&=\big\langle h,\tilde{\M}_{f}\nabla_{p^*}(g)-\nabla_{p^*}\M_{f}(g) \big\rangle_{}
=\big\langle h,f \cdot \nabla_{p^*}(g)-\nabla_{p^*}(fg) \big\rangle_{} \\
&\ov{\eqref{Leib-gradient}}{=} - \big\langle h,g \cdot \nabla(f) \big\rangle_{}
=\big\langle h, -{\M}_{\nabla f}J(g)\big\rangle_{}
= \big\langle h, \M_{\nabla f}J(g) \big\rangle_{} \\
&=\big\langle {\M}_{\nabla f}(h), J(g) \big\rangle
=\big\langle J^*{\M}_{\nabla f}(h),g\big\rangle_{\L^p(G),\L^{p^*}(G)}.
\end{align*} 
We conclude that
\begin{equation}
\label{Commutateur-etrange}
\big((\nabla_{p^*})^*\tilde{\M}_f-\M_f (\nabla_{p^*})^*\big)(h)
=J^*{\M}_{\nabla f}(h), \quad h \in \C^\infty_c(G) \ot \ell^p_m.
\end{equation}
The two non-zero components of the commutator are bounded linear operators on $\C^\infty_c(G)$ and on $\C^\infty_c(G) \ot \ell^p_m$. We deduce that $\big[\slashed{D}_{p},\pi(f)\big]$ is bounded on the core $(\C^\infty_c(G) \oplus \ell^p_m) \ot \C^\infty_c(G)$ of the unbounded operator $\slashed{D}_{p}$ (here we use Lemma \ref{Lemma-core}). By \cite[Proposition 26.5]{ArK22}, this operator extends to a bounded operator on the Banach space $\L^p(G) \oplus_p \L^p(G,\ell^p_m)$. Hence $\C^\infty_c(G)$ is a subset of $\Lip_{\slashed{D}_p}(\L^\infty(G))$.

If $(g,h) \in \dom \slashed{D}_p$ and $f \in \C^\infty_c(G)$, we have in addition
\begin{align*}
\MoveEqLeft
\norm{\M_{\nabla f}J}_{\L^p(G) \to \ell^p_m(\L^p(G))} 
=\sup_{\norm{g}_{\L^p(G)}=1} \norm{\M_{\nabla(f)}J(g)}_{\ell^p_m(\L^p(G))}\\
&=\sup_{\norm{g}_{\L^p(G)}=1} \norm{\big(X_{1,p}(f)g, \ldots, X_{m,p}(f)g\big)}_{\ell^p_m(\L^p(G))}\\   
&=\sup_{\norm{g}_{\L^p(G)}=1} \norm{\big(X_{1,p}(f)g, \ldots, X_{m,p}(f)g\big)}_{\L^p(G,\ell^p_m)} \\ 
&=\sup_{\norm{g}_{\L^p(G)}=1} \bigg(\int_G |(X_{1,p}f)(s)g(s)|^p+\cdots+ |(X_{m,p}f)(s)g(s)|^p \d \mu_G(s)\bigg)^{\frac{1}{p}}\\
&=\sup_{\norm{g}_{\L^p(G)}=1} \bigg(\int_G \big[|(X_{1,p}f)(s)|^p+\cdots+ |(X_{m,p}f)(s)|^p\big] |g(s)|^p\d \mu_G(s)\bigg)^{\frac{1}{p}}\\ 
&=\sup_{\norm{h}_{\L^1(G)}=1, h \geq 0} \bigg(\int_G \big[|(X_{1,p}f)(s)|^p+\cdots+ |(X_{m,p}f)(s)|^p\big] h(s)\d \mu_G(s)\bigg)^{\frac{1}{p}}\\ 			
&=\bigg(\sup_{\norm{h}_{\L^1(G)}=1, h \geq 0} \big\la \norm{\nabla f}_{\ell^p_m}^p,h \big\ra_{\L^\infty(G),\L^1(G)}\bigg)^{\frac{1}{p}} \\
&=\norm{\norm{\nabla f}_{\ell^p_m}^p}_{\L^\infty(G)} ^{\frac{1}{p}}
=\norm{\nabla f}_{\L^\infty(G,\ell^p_m)}.
\end{align*} 
By duality, the second non-null entry of the commutator has the same norm. So we have proved \eqref{norm-commutator} in the case where $f \in \C^\infty_c(G)$.

Let $f \in \dom \nabla_{\infty}$. Since $\C^\infty_c(G)$ is a weak* core of the operator $\nabla_{\infty}$, we can consider a net $(f_j)$ in $\C^\infty_c(G)$ such that $f_j \to f$ and $\nabla_{\infty}(f_j) \to \nabla_{\infty}(f)$ both for the weak* topology of $\L^\infty(G)$. By \cite[Lemma 1.6]{ArK22}, we can suppose that the nets $(f_j)$ and $(\nabla_{\infty}(f_j))$ are bounded. By \cite[Proposition 5.11 4.]{ArK22}, we deduce that $f \in \Lip_{\slashed{D}_{p}}(\L^\infty(G))$. By continuity of $\pi$, note that $\pi(f_j) \to \pi(f)$ for the weak operator topology. For any $\xi \in \dom \slashed{D}_{p}$ and any $\zeta \in \dom (\slashed{D}_{p})^*$, we have
\begin{align*}
\MoveEqLeft
\big\langle [\slashed{D}_{p},\pi(f_j)]\xi,\zeta \big\rangle_{\L^p(G) \oplus_p \L^{p}(G,\ell^{p}_m),\L^{p^*}(G) \oplus_{p^*} \L^{p^*}(G,\ell^{p^*}_m)}            
=\big\langle (\slashed{D}_{p}\pi(f_j)-\pi(f_j)\slashed{D}_{p})\xi, \zeta\big\rangle_{} \\
&=\big\langle \slashed{D}_{p}\pi(f_j)\xi, \zeta\big\rangle_{}-\big\langle \pi(f_j) \slashed{D}_{p}\xi, \zeta\big\rangle_{} 
=\big\langle \pi(f_j)\xi, (\slashed{D}_{p})^*\zeta\big\rangle-\big\langle \pi(f_j)\slashed{D}_{p}\xi, \zeta \big\rangle_{} \\
&\xra[j]{} \big\langle \pi(f)\xi, (\slashed{D}_{p})^*\zeta\big\rangle-\big\langle \pi(f)\slashed{D}_{p}\xi, \zeta \big\rangle_{} 
=\big\langle [\slashed{D}_{p},\pi(f)]\xi,\zeta \big\rangle_{}.
\end{align*}
The net $([\slashed{D}_{p},\pi(f_j)])$ is bounded since 
$$
\norm{\big[\slashed{D}_p,\pi(f_j)\big]}_{p \to p}
\ov{\eqref{norm-commutator}}{=} \bnorm{\nabla_{\infty}(f_j)}_{\L^\infty(G,\ell^p_m)}
\lesssim_{m,p} \bnorm{\nabla_{\infty}(f_j)}_{\L^\infty(G,\ell^\infty_m)}
\lesssim 1.
$$ 
We deduce that the net $([\slashed{D}_{p},\pi(f_j)])$ converges to $[\slashed{D}_{p},\pi(f)]$ for the weak operator topology by a ``net version'' of \cite[Lemma 3.6 p.~151]{Kat76}. Furthermore, it is (really) easy to checkthat $\M_{\nabla_{\infty}(f_j)}J \to \M_{\nabla_{\infty}(f)}J$ and $-\E{\M}_{\nabla_{\infty}(f_j)} \to -\E {\M}_{\nabla_{\infty}(f)}$ both for the weak operator topology. Indeed, recall that the composition of operators is separately continuous for the weak operator topology. By uniqueness of the limit, we deduce that the commutator is given by the same formula that in the case of elements of $\C^\infty_c(G)$. From here, we obtain \eqref{norm-commutator} as before.
\end{proof}

\begin{remark} \normalfont
\label{Intrigating-rem}
The inclusion \eqref{Lip-algebra-I-Schur} is probably an equality. We leave this intriguing question open. We sketch an incomplete proof. Let $f$ an element of $\Lip_{\slashed{D}_p}(\L^\infty(G))$. We consider a Dirac net $(\varphi_j)$ of functions of $\C^\infty(G)$. For any $j$, we let $f_j \ov{\mathrm{def}}{=} \Reg_j f$. By an obvious ``net version'' of \cite[14.11.1]{Dieu83}, the net $(f_j)$ converges to $f$ in $\L^\infty(G)$ for the weak* topology. The point is to prove that $(\nabla_\infty f_j)$ is a bounded net. If it is true, using Banach-Alaoglu theorem, we can suppose that $\nabla_{\infty}(f_j)\to g$ for the weak* topology for some function $g \in \L^\infty(G)$. Since the graph of the unbounded operator $\nabla_\infty$ is weak* closed, we would conclude that $f$ belongs to the subspace $\dom \nabla_\infty$. 

For the proof of the boundedness of the net, the writing 
$$
\norm{\nabla_\infty f_j}_{\L^\infty(G,\ell^\infty_m)}
\approx_{m,p} \norm{\nabla_\infty f_j}_{\L^\infty(G,\ell^p_m)}
\ov{\eqref{norm-commutator}}{=} \sup_{\norm{\xi} \leq 1, \: \norm{\eta} \leq 1} \left| \big\langle [\slashed{D}_p,\pi(f_j)] \xi, \eta \big\rangle \right|
$$ 
and maybe \cite[(14.10.9)]{Dieu83} could be useful. 

\end{remark}

\begin{thm}
\label{thm-spectral-triple}
Let $G$ be a compact connected Lie group equipped with a family $(X_1, \ldots, X_m)$ of left-invariant H\"ormander vector fields and consider the normalized Haar measure $\mu_G$ on $G$. The triple $(\C(G),\L^2(G) \oplus_2 \L^2(G,\ell^2_m), \slashed{D}_2)$ is a compact spectral triple.
\end{thm}

\begin{proof}
Here, we use the notation $\nabla \ov{\mathrm{def}}{=} \nabla_2$. On the Hilbert space $(\ker \nabla^* \nabla)^\perp \oplus_2 (\ker \nabla \nabla^*)^\perp$ we have
\begin{align}
\label{D-1}
\MoveEqLeft
|\slashed{D}_2|^{-1}
=\big(\slashed{D}_2^2\big)^{-\frac{1}{2}}
\ov{\eqref{carre-de-D}}{=} \begin{bmatrix} 
\nabla^* \nabla & 0 \\ 
0 & \nabla \nabla^*
\end{bmatrix}^{-\frac{1}{2}}
=\begin{bmatrix} 
\big(\nabla^* \nabla\big)^{-\frac{1}{2}} & 0 \\ 
0 & \big(\nabla \nabla^*\big)^{-\frac{1}{2}}
\end{bmatrix}.            
\end{align} 
By Theorem \ref{Th-unit-eq}, we know that the operators $\nabla^* \nabla |_{(\ker \nabla)^\perp}$ and $\nabla \nabla^*|_{(\ker \nabla^*)^\perp}$ are unitarily equivalent. Moreover, we have
$$
(\ker \nabla)^\perp
\ov{\eqref{lien-ker-image}}{=} \ovl{\Ran \nabla^*}
\ov{\eqref{Kadison-image}}{=}
\ovl{\Ran \nabla^* \nabla}
\ov{\eqref{lien-ker-image}}{=} (\ker \nabla^* \nabla)^\perp
$$
and
$$
(\ker \nabla \nabla^*)^\perp
\ov{\eqref{Kadison-image}}{=} (\ker \nabla^*)^\perp.
$$
Consequently $(\nabla^* \nabla)^{-\frac{1}{2}}|_{(\ker \nabla^* \nabla)^\perp}$ and $(\nabla \nabla^*)^{-\frac{1}{2}}|_{(\ker \nabla \nabla^*)^\perp}$ are also unitarily equivalent. By Proposition \ref{Lemma-compactness}, the operator $(\nabla^* \nabla)^{-\frac12}=\Delta_2^{-\frac12} \co \ovl{\Ran \Delta_2} \to \ovl{\Ran \Delta_2}$ is compact (the square root does not change the compactness by \cite[Lemma 9.3]{Nee22}) on $\ovl{\Ran \Delta_2} \ov{\eqref{lien-ker-image}}{=} (\ker \nabla^* \nabla)^\perp$. Hence the operator $(\nabla \nabla^*)^{-\frac{1}{2}}|_{(\ker \nabla \nabla^*)^\perp}$ is also compact. We conclude that the operator $|\slashed{D}_2|^{-1}$ is compact. 
\end{proof}

\begin{remark} \normalfont
Recall that a compact spectral triple $(A,H,\slashed{D})$ is even if there exists a selfadjoint unitary operator $\gamma \co H \to H$ such that $\gamma \slashed{D}=-\slashed{D}\gamma$ and $\gamma \pi(a)=\pi(a)\gamma$ for any $a \in A$. Note that the spectral triple $(\C(G),\L^2(G) \oplus_2 \L^2(G,\ell^2_m),\slashed{D})$ is even. Indeed, the Hodge-Dirac operator $\slashed{D}_{p}$ anti-commutes with the involution 
\begin{align*}
\MoveEqLeft
\gamma_p
\ov{\mathrm{def}}{=}
\begin{bmatrix} 
-\Id_{\L^p(G)} & 0 \\ 
0& \Id_{\L^p(G,\ell^p_m)} 
\end{bmatrix} \co \L^p(G) \oplus_p \L^p(G,\ell^p_m) \to \L^p(G) \oplus_p \L^p(G,\ell^p_m) 
\end{align*}  
(which is selfadjoint if $p=2$), since
\begin{align*}
\MoveEqLeft
\slashed{D}_{p}\gamma_p+\gamma_p\slashed{D}_{p}
\ov{\eqref{Hodge-Dirac-I}}{=} 
\begin{bmatrix} 
0 & (\nabla_{p^*})^* \\ 
\nabla_{p}& 0 
\end{bmatrix}
\begin{bmatrix} 
-\Id_{} & 0 \\ 
0& \Id_{}
\end{bmatrix}
+\begin{bmatrix} 
-\Id_{} & 0 \\ 
0& \Id_{}
\end{bmatrix}
\begin{bmatrix} 
0 & (\nabla_{p^*})^* \\ 
\nabla_{p}& 0 
\end{bmatrix} \\
&=\begin{bmatrix} 
0 & (\nabla_{p^*})^* \\ 
-\nabla_{p}&  0
\end{bmatrix}+\begin{bmatrix} 
0 & -(\nabla_{p^*})^* \\ 
\nabla_{p}&  0
\end{bmatrix}
=0.            
\end{align*} 
Moreover, for any $f \in \L^\infty(G)$, we have
\begin{align*}
\MoveEqLeft
\gamma_p\pi(f)
\ov{\eqref{Def-pi-a}}{=} \begin{bmatrix} 
-\Id_{} & 0 \\ 
0& \Id_{}
\end{bmatrix} 
\begin{bmatrix}
    \M_f & 0  \\
    0 & \tilde{\M}_f  \\
\end{bmatrix}  
=\begin{bmatrix}
    -\M_f & 0  \\
    0 & \tilde{\M}_f  \\
\end{bmatrix}        
=\begin{bmatrix}
    \M_f & 0  \\
    0 & \tilde{\M}_f  \\
\end{bmatrix}\begin{bmatrix} 
-\Id_{} & 0 \\ 
0& \Id_{}
\end{bmatrix}
\ov{\eqref{Def-pi-a}}{=} \pi(f)\gamma_p.
\end{align*} 
\end{remark}

In the sequel, if $1 \leq p< \infty$ and if $H$ is a Hilbert space, we use the notation $S^p(H)$ for the space of the compact operators $T \co H \to H$ such that $\norm{T}_{S^p(H)} \ov{\mathrm{def}}{=}\big(\tr |T|^p\big)^{\frac{1}{p}}<\infty$. Moreover, recall that the local dimension $d$ of $(G,X)$ is defined in \eqref{local-dim}. 

\begin{prop}
\label{Prop-trace-class}
Assume that $G$ is compact. If $\alpha > d$, the operator $|\slashed{D}|^{-\alpha}$ is trace-class. 
\end{prop}

\begin{proof}
We consider the canonical projection $Q \co \L^2(G) \to \L^2(G)$ on $\L^2_0(G)$. We have $Q=\Id-\E$ where the conditional expectation $\E$ is defined in Section \ref{sec-preliminaries} by $\E(f)=\big(\int_G f\big)1$. For any $t >0$ and any function $f \in \L^2(G)$, we have
\begin{align}
\MoveEqLeft
\label{Sans-fin-3}
T_t Qf
=T_t(\Id-\E)f            
=T_tf-T_t\E f 
=K_t*f-\int_G f.
\end{align} 
For almost all $s \in G$, we infer that using the translation invariance of the Haar measure
\begin{align*}
\MoveEqLeft
T_t Qf(s)            
\ov{\eqref{Sans-fin-3}}{=} (K_t*f)(s)-\int_G f
\ov{\eqref{Convolution-formulas}}{=} \int_G K_t(r)f(r^{-1}s) \d \mu_G(r)-\int_G f(r) \d \mu_G(r) \\
&=\int_G \big[K_t(r)-1\big]f(r^{-1}s) \d \mu_G(r).
\end{align*} 
We conclude that $T_tQ \co \L^2(G) \to \L^2(G)$ is a convolution operator by the function $K_t-1$. 

For any $t>0$, we have using \cite[Exercise 2.8.38 (ii) p.~170]{KaR97} 
\begin{align}
\MoveEqLeft
\label{Infinite-2}
\norm{T_{t}}_{S^1(\L^2_0(G))}
=\bnorm{T_{\frac{t}{2}}^2}_{S^1(\L^2_0(G))}          
= \bnorm{T_\frac{t}{2}}_{S^2(\L^2_0(G))}^2
=\bnorm{T_\frac{t}{2}Q}_{S^2(\L^2(G))}^2\\
&=\bnorm{K_{\frac{t}{2}}-1}_{\L^2(G)}^2
=\int_G |(K_{\frac{t}{2}}-1)(r)|^2 \d\mu_G(r) \nonumber
\end{align}  
(for the third equality, consider an orthonormal basis $(e_i)_{i \in I}$ of the Hilbert space $\L^2(G)$ adapted to the closed subspace $\L^2_0(G)$ and use the equality $\norm{T}_{S^2(\L^2(G))}=\big(\sum_{i \in I} \norm{T(e_i)}_{\L^2(G)}^2\big)^{\frac{1}{2}}$). Now, by translation invariance of the Haar measure and Dunford-Pettis theorem, we obtain 
\begin{align*}
\MoveEqLeft
\norm{T_{t}}_{S^1(\L^2_0(G))} 
=\esssup_{s \in G} \int_G |(K_{\frac{t}{2}}-1)(sr^{-1})|^2 \d\mu_G(r) \\
&\ov{\eqref{Dunford-Pettis-2}}{=} \bnorm{T_\frac{t}{2}Q}_{\L^2(G) \to \L^\infty(G)}^2 
=\bnorm{T_\frac{t}{2}}_{\L^2_0(G) \to \L^\infty(G)}^2.
\end{align*}
By interpolation, we have by \eqref{Rnpq} combinated with \eqref{estimates-L1-Linfty} the estimate 
$$
\bnorm{T_\frac{t}{2}}_{\L^2_0(G) \to \L^\infty(G)} \lesssim \frac{1}{t^{\frac{d}{4}}}, \quad 0 < t \leq 1.
$$ 
We infer that 
\begin{equation}
\label{magic-estimates}
\norm{T_{t}}_{S^1(\L^2_0(G))}
\lesssim \frac{1}{t^{\frac{d}{2}}}, \quad  0 < t \leq 1.
\end{equation}
Now, if $t \geq 1$, we have since $G$ is compact
\begin{align*}
\MoveEqLeft
\norm{T_{t}}_{S^1(\L^2_0(G))}                    
\ov{\eqref{Infinite-2}}{=} \bnorm{K_{\frac{t}{2}}-1}_{\L^2(G)}^2 
\leq  \bnorm{K_{\frac{t}{2}}-1}_{\L^\infty(G)}^2 \\
&\ov{\eqref{Dunford-Pettis-1}}{=} \bnorm{T_\frac{t}{2}Q}_{\L^1(G) \to \L^\infty(G)}^2 
=\bnorm{T_\frac{t}{2}}_{\L^1_0(G) \to \L^\infty(G)}^2.
\end{align*}
With the estimate \eqref{A-l-infini}, we conclude that 
\begin{equation}
\label{magic-estimates-2}
\norm{T_{t}}_{S^1(\L^2_0(G))}
\lesssim \e^{-2\omega t},\quad   t \geq 1.
\end{equation}
Observe that the map is $\R^+ \mapsto \B(\L^2(G))$, $t \mapsto T_t$ is strong operator continuous hence weak operator continuous. Moreover, if $\alpha >d$, we have
\begin{align*}
\MoveEqLeft
\int_{0}^{\infty} t^{\frac{\alpha}{2}-1} \norm{T_t}_{S^1(\L^2_0(G))} \d t
=\int_{0}^{1} t^{\frac{\alpha}{2}-1} \norm{T_t}_{S^1(\L^2_0(G))} \d t
+\int_{1}^{\infty} t^{\frac{\alpha}{2}-1} \norm{T_t}_{S^1(\L^2_0(G))} \d t \\         
&\ov{\eqref{magic-estimates}}{=}\int_{0}^{1} t^{\frac{\alpha}{2}-1-\frac{d}{2}} \d t
+\int_{1}^{\infty} t^{\frac{\alpha}{2}-1} \e^{-2\omega t} \d t 
<\infty.
\end{align*}
By \cite[Lemma 2.3.2]{SuZ18}, we deduce that the operator $\int_{0}^{\infty} t^{\frac{\alpha}{2}-1} T_t \d t$ acting on the Hilbert space $\L^2_0(G)$ is well-defined and trace-class. Furthermore, we have $\norm{T_{t}}_{\L^2_0(G) \to \L^2_0(G)} \ov{\eqref{magic-estimates-2}}{\lesssim} \e^{-2\omega t}$ if $t \geq 1$, that means that $(T_t)_{t \geq 0}$ is an exponentially stable semigroup on $\L^2_0(G)$. Consequently, we know by \cite[Corollary 3.3.6]{Haa06} that
$$
\Delta^{-\frac{\alpha}{2}}
=\frac{1}{\Gamma(\frac{\alpha}{2})} \int_{0}^{\infty} t^{\frac{\alpha}{2}-1} T_t \d t.
$$
We obtain that if $\alpha>d$ then the operator $\Delta^{-\frac{\alpha}{2}}$ is trace-class. The operator $( \nabla \nabla^*)^{-\alpha}$ is also trace-class since unitarily equivalent to $\Delta^{-\frac{\alpha}{2}}$ as observed in the proof of Theorem \ref{thm-spectral-triple}. Finally, note that 
\begin{align}
\label{D-1}
\MoveEqLeft
|\slashed{D}_2|^{-\alpha}
\ov{\eqref{carre-de-D}}{=} \begin{bmatrix} 
\Delta_2^{-\frac{\alpha}{2}} & 0 \\ 
0 & (\nabla \nabla^*)^{-\frac{\alpha}{2}}
\end{bmatrix}.            
\end{align}
We conclude that the operator $|\slashed{D}_2|^{-\alpha}$ is trace-class if $\alpha >d$. 
\end{proof}

\begin{thm}
\label{}
Assume that $G$ is compact. The spectral dimension \eqref{Def-spectral-dimension} of the spectral triple $(\C(G),\L^2(G) \oplus_2 \L^2(G,\ell^2_m),\slashed{D}_2)$ is equal to the local dimension $d$ of $(G,X)$.
\end{thm}

\begin{proof}
Note Proposition \ref{Prop-trace-class}. Now, suppose that the operator $|\slashed{D}|^{-\alpha}$ is trace-class. By Lemma \ref{Lemma-theta-summable}, we have $\tr \e^{-t\slashed{D}^2} \lesssim \frac{1}{t^{\frac{\alpha}{2}}}$ for any $t > 0$. Using \cite[Proposition II.3.1 p.~20]{DtER03}, the relation $K_{t}=\check{K_t}$ in the second equality, \cite[Exercise 2.8.38 (ii) p.~170]{KaR97} in the fourth equality, the selfadjointness of $T_t$ and finally \eqref{carre-de-D} in the last inequality, we deduce that
\begin{align*}
\MoveEqLeft
K_t(e)
=\int_G K_{\frac{t}{2}}(r)K_{\frac{t}{2}}(r^{-1}) \d \mu_G(r)
=\int_G K_{\frac{t}{2}}^2
=\bnorm{K_{\frac{t}{2}}}_{\L^2(G)}^2
=\bnorm{T_{\frac{t}{2}}}_{S^2(\L^2(G))}^2 \\
&=\bnorm{T_{\frac{t}{2}}^2}_{S^1(\L^2(G))} 
=\norm{T_t}_{S^1(\L^2(G))}
=\tr T_t 
=\tr \e^{-t\Delta_2}
\leq \frac{1}{t^{\frac{\alpha}{2}}}.
\end{align*}  
By \cite[(3) p.~113]{VSCC92} (see also \cite[p.~174]{DtER03} for a more general statement for selfadjoint subelliptic operators), we have 
$$
\frac{1}{t^\frac{d}{2}} 
\ov{\eqref{local-dim}}{\approx}  \frac{1}{V(\sqrt{t})} \lesssim K_t(e), \quad 0< t \leq 1.
$$ 
We conclude that $\frac{1}{t^\frac{d}{2}} \lesssim \frac{1}{t^{\frac{\alpha}{2}}}$ for any $0< t \leq 1$ and consequently $\alpha \geq d$.
\end{proof}


%

\section{Some remarks on Carnot-Carath\'eodory distances}
\label{Section-distances}

Let $G$ be a connected unimodular Lie group equipped with a family $(X_1, \ldots, X_m)$ of left-invariant H\"ormander vector fields. In this section, we will show in Theorem \ref{Th-recover-metric} that the Connes spectral pseudo-distance associated to our Hodge-Dirac operator allows us to recover the Carnot-Carath\'eodory distance. If $(A,Y,\slashed{D})$ is a triple as that precedes \eqref{Lipschitz-algebra-def}, recall that it is defined by 
\begin{equation}
\label{Connes-distance}
\dist_{\slashed{D}}(\varphi,\psi)
\ov{\mathrm{def}}{=}\sup \big\{ |\varphi(a) - \psi(a)| : a \in \Lip_\slashed{D}(A) \text{ and }\norm{[\slashed{D},\pi(a)]} \leq 1 \big\}
\end{equation}
where $\varphi$ and $\psi$ are two states of the algebra $A$ and where $\Lip_\slashed{D}(A)$ is defined in \eqref{Lipschitz-algebra-def}. The term pseudo-metric is used since $\dist(\varphi,\psi)$ is not necessarily finite. In general, $\Lip_\slashed{D}(A)$ is unknown and we replace this space by a dense subset of $A$ (or a weak* dense subset if $A$ is a dual space) which is contained in $\Lip_\slashed{D}(A)$. See the papers \cite{Lat21}, \cite{Pav98} and \cite{Rie98} for more information.

In \cite[Lemma 2.3 p.~265]{Rob91} and \cite[p.~24]{DtER03}, the following formula is stated for the Carnot-Carath\'eodory distance, i.e. the case $p=2$ of \eqref{distance-Carnot}. For any $s,s' \in G$, it is written that
\begin{equation}
\label{dual-carthodory}
\dist_\CC(s,s') 
\ov{\mathrm{def}}{=} \sup \left\{ |f(s) - f(s')| : f \in \C^\infty_c(G), \norm{\nabla f}_{\L^\infty(G,\ell^2_m)} \leq 1 \right\}.
\end{equation}
We will see that this formula is strongly related to the distance \eqref{Connes-distance} in our setting. Unfortunately, we are unable to understand the sketched proof. The writings ``Therefore'' and ``by a slight modification of the $\psi_n$ one can arrange'' of \cite[Lemma 2.3 p.~265]{Rob91} are obscure for us. 
Moreover, the same reference \cite[p.~24]{DtER03} says without proof that we can replace the space $\C^\infty_c(G)$ by the subspace $\C^\infty_c(G,\R)$ of real-valued compactly supported continuous functions in this formula. 

Indeed, we can use the following elementary argument. We fix $s,s' \in G$. We write $f(s)-f(s')=|f(s)-f(s')|\e^{\i \theta}$ for some $\theta \in \R$. We consider the real-valued function $\tilde{f} \ov{\mathrm{def}}{=} \frac{1}{2}[f\e^{-\i\theta}+\ovl{f}\e^{\i\theta}]$. We have
$$
\norm{\nabla \tilde{f}}_{\L^\infty(G,\ell^2_m)}
=\frac{1}{2}\bnorm{\e^{-\i\theta}\nabla(f)+\e^{-\i\theta}\nabla(f)}_{\L^\infty(G,\ell^2_m)} 
\leq \norm{\nabla f}_{\L^\infty(G,\ell^2_m)}
$$ 
and
\begin{align*}
\MoveEqLeft
|\tilde{f}(s)-\tilde{f}(s')|           
=\frac{1}{2} \big|f(s)\e^{-\i\theta}+\ovl{f}(s)\e^{\i\theta}-f(s')\e^{-\i\theta}-\ovl{f}(s')\e^{\i\theta}\big| \\
&=\frac{1}{2} \big|\e^{-\i\theta}[f(s)-f(s')]+ \e^{\i\theta}[\ovl{f}(s)-\ovl{f}(s')]\big| \\
&=\frac{1}{2}\big||f(s)-f(s')|+\ovl{|f(s)-f(s')|} \big|
=|f(s)-f(s')|.
\end{align*}

Now, we introduce the following definition.

\begin{defi} 
\label{def Lip_CC}
Suppose $1<p<\infty$. Let $f \co G \to \Cbb$ be a function. The number
\begin{equation}
\label{Def-constant-Lip}
\Lip_{\CC}^p(f) 
\ov{\mathrm{def}}{=} \sup \left\{\frac{|f(s)-f(s')|}{\dist_{\CC}^p(s,s')}: s,s' \in G, 
s \neq s'\right\}
\end{equation}
of $[0,\infty]$ is called the $p$-Carnot-Carath\'{e}odory-Lipschitz constant of $f$. If $\Lip_{\CC}^p(f)$ is finite, we call $f$ a $p$-Carnot-Carath\'{e}odory-Lipschitz function.
\end{defi}

In \cite[Proposition 2.5 (i)]{BEM13}, it is stated that the domain $\dom \nabla_\infty$ is the space of the equivalence classes of \textit{bounded} $2$-Carnot-Carath\'{e}odory-Lipschitz function on $G$. Here, we complete this fact and we give a variant of \cite[Proposition 2.5]{BEM13}.


\begin{lemma} 
\label{lemmaLipCC-grad}
Suppose $1<p<\infty$. Then an essentially bounded function $f \co G \to \Cbb$ is a $p$-Carnot-Carath\'{e}odory-Lipschitz function if and only if its equivalence class belongs to the space $\dom \nabla_\infty$. In this case, we have
$$
\Lip_{\CC}^p(f) 
= \norm{\nabla f}_{\L^\infty(G,\ell^{p^*}_m)}.
$$
\end{lemma}

\begin{proof}
Suppose that $f \in \C_c^\infty(G)$. Let $s,s' \in G$ and let $\gamma \co [0,1] \mapsto G$ be an absolutely continuous path from $s$ to $s'$. We have
\begin{align*}
\label{}
\MoveEqLeft
f(s)-f(s') 
=f(\gamma(0))-f(\gamma(1)) 
=-\int_{0}^{1} \frac{\d}{\d t}f(\gamma(t)) \d t
\ov{\eqref{Chain-rule}}{=}-\int_{0}^{1} \sum_{k=1}^{m} \dot\gamma_k(t) (X_kf)(\gamma(t)) \d t.
\end{align*} 
Consequently, using H\"older's inequality, we obtain
\begin{align*}
\MoveEqLeft
\big|f(s)-f(s') \big|  
\leq \int_{0}^{1} \bigg|\sum_{k=1}^{m} \dot\gamma_k(t) (X_kf)(\gamma(t))\bigg| \d t \\       
&\leq \int_{0}^{1}  \bigg(\sum_{k=1}^{m} |\dot\gamma_k(t)|^p\bigg)^{\frac{1}{p}} \bigg(\sum_{k=1}^{m} \big|(X_k f)(\gamma(t))\big|^{p^*}\bigg)^{\frac{1}{p^*}} \d t.
\end{align*}
We deduce that
\begin{align*}
\MoveEqLeft
\big|f(s)-f(s') \big|            
\leq \int_{0}^{1}  \bigg(\sum_{k=1}^{m} |\dot\gamma_k(t)|^p\bigg)^{\frac{1}{p}} \d t\, \bnorm{(X_1f,\ldots,X_{m}f)}_{\L^\infty(G,\ell^{p^*}_m)} 
\ov{\eqref{Def-lpgamma}}{=} \bnorm{\nabla f}_{\L^\infty(G,\ell^{p^*}_m)}\ell_p(\gamma).
\end{align*} 
Passing to the infimum, we obtain $|f(s)-f(s')| 
\leq \norm{\nabla f}_{\L^\infty(G,\ell^{p^*}_m)} \dist_\CC^p(s,s')$ by \eqref{distance-Carnot}. Consequently, we have the inequality $\Lip_{\CC}^p(f) \leq \norm{\nabla f}_{\L^\infty(G,\ell^{p^*}_m)}$.  
We conclude with a regularization argument for the general case of a function $f$ of $\dom \nabla_\infty$. 

Now, we prove the reverse inequality. Suppose that $f \co G \to \Cbb$ is a $p$-Carnot-Carath\'{e}odory-Lipschitz function. Let $\xi \in \ell_m^{p}$ with $\norm{\xi}_{\ell_m^{p}}=1$. For any $1 \leq k \leq m$, we put $a_k \ov{\mathrm{def}}{=} X_k(e)$. Consider some $s_0 \in G$ and the path 
\begin{equation}
\label{Def-gamma}
\gamma(t)
\ov{\mathrm{def}}{=} s_0\exp\bigg(t\sum_{k=1}^{m}\xi_ka_k\bigg), \quad t \in \R.
\end{equation}
By \cite[(19.8.11)]{Dieu77}, for any $t \in \R$ we have
\begin{align}
\label{gamma-prime}
\MoveEqLeft
\dot\gamma(t)            
=\sum_{k=1}^{m} \xi_k X_k(\gamma(t)).
\end{align}  
We infer that $\dot\gamma(t)$ belongs to the subspace $\Span \{ X_1|_{\gamma(t)}, \ldots, X_m|_{\gamma(t)} \}$ for all $t \in \R$. Moreover, the $p$-length of the restriction $\gamma|[c,d]$ is given by
\begin{align*}
\MoveEqLeft
\ell_p(\gamma|[c,d]) 
\ov{\eqref{Def-lpgamma}}{=} \int_c^d  \Big( \sum_{k=1}^m |\dot\gamma_k(t)|^p \Big)^{\frac{1}{p}} \d t
\ov{\eqref{gamma-prime}}{=} \int_c^d \norm{\xi}_{\ell_m^p} \d t 
=|c-d|\norm{\xi}_{\ell_m^{p}}
=|c-d|.
\end{align*} 
By \eqref{distance-Carnot}, we deduce that
\begin{equation}
\label{dist-7}
\dist_\CC^p(\gamma(c),\gamma(d)) 
\leq |d-c|.
\end{equation}
We consider the function $g \co \R \to \R$, $\dsp t \mapsto f(\gamma(t))$. Since $f$ is a $p$-Carnot-Carath\'{e}odory-Lipschitz function, we have
\begin{align*}
\MoveEqLeft
|g(c)-g(d)|
=|f(\gamma(c))-f(\gamma(d))|           
\leq \Lip_{\CC}^p(f) \dist_{\CC}^p(\gamma(c),\gamma(d)) 
\ov{\eqref{dist-7}}{\leq} \Lip_{\CC}^p(f) |d-c|.
\end{align*}  
Hence the function $g$ is a Lipschitz function on $\R$, hence differentiable almost everywhere. It is left to the reader to show that $X_jf$ exists almost everywhere on $G$. 

If $t>0$, using the notation $s_t \ov{\mathrm{def}}{=} \exp\big(t\sum_{k=1}^{m}\xi_ka_k\big)$ we deduce that
\begin{align*}
\MoveEqLeft
\Lip_{\CC}^p(f)            
\ov{\eqref{Def-constant-Lip}}{\geq} \frac{|f(\gamma(t))-f(s_0)|}{\dist_\CC^p(\gamma(t),s_0)} 
\ov{\eqref{dist-7}}{\geq} \frac{|f(\gamma(t))-f(s_0)|}{t} 
=\left|\frac{1}{t}\big((\rho_{s_t}-\Id)f\big)(s_0)\right|
\end{align*} 
where $\rho$ is the right regular representation of $G$. Consequently, for any $t>0$ we obtain
\begin{equation}
\label{Passing}
\norm{\frac{1}{t}(\rho_{s_t}-\Id)f}_{\L^\infty(G)} 
\leq \Lip_{\CC}^p(f).
\end{equation}
Now, $\left|\frac{1}{t}\big((\rho_{s_t}-\Id)f\big)(s_0)\right|=\frac{|f(\gamma(t))-f(s_0)|}{t}$ converges almost everywhere when $t \to 0$. Using dominated convergence theorem, we conclude that $\frac{1}{t}(\rho_{s_t}-\Id)f$ converges in $\L^\infty(G)$ for the weak* topology. 

Using \cite[p.~14]{DtER03}, we infer that the class of $f$ belongs to $\dom \nabla_\infty$ and that $\frac{1}{t}\big((\rho_{s_t}-\Id)f \to \sum_{k=1}^{m} \xi_k Y_{k,\infty}f$  when $t \to 0$ for the weak* topology of $\L^\infty(G)$ where $Y_{k}$ is the right invariant vector field associated to the element $a_k$. Passing to the limit in \eqref{Passing} when $t \to 0$, using the weak* lower semicontinuity of the norm \cite[Th 2.6.14 p.~227]{Meg98}, we obtain
$$
\norm{\sum_{k=1}^{m} \xi_k Y_{k,\infty}f}_{\L^\infty(G)}
\leq \Lip_{\CC}^p(f).
$$
Since $G$ is unimodular, we have $\norm{\sum_{k=1}^{m} \xi_k X_{k,\infty}f}_{\L^\infty(G)}=\norm{\sum_{k=1}^{m} \xi_k Y_{k,\infty}f}_{\L^\infty(G)}$ (if $I \co G \to G$, $s\mapsto s^{-1}$ is the inversion map and $I_*$ its associated push-forward map on vector fields, we have $I_*X_k=-Y_k$). We conclude by duality that
$$
\norm{\nabla f}_{\L^\infty(G,\ell^{p^*}_m)}
=\norm{(X_1f,\ldots,X_{m}f)}_{\L^\infty(G,\ell^{p^*}_m)}
\leq \Lip_{\CC}^p(f).
$$ 
\end{proof}


\begin{lemma} 
\label{Lemma-dCC}
Suppose $1<p< \infty$. For any $s,s' \in G$, we have
$$
\dist_{\CC}^p(s,s') 
= \sup \big\{ |f(s)-f(s')|: f \in \dom \nabla_\infty, \Lip_{\CC}^p(f) \leq 1 \big\}.
$$
Moreover, we can replace by $\dom \nabla_\infty$ by the space $\C_c^\infty(G)$.
\end{lemma}

\begin{proof}
Let $f \in \dom \nabla_\infty$ with $\Lip_{\CC}^p(f) \leq 1$. For any $s,s' \in G$, we have by definition
\begin{equation*} 
\left| f(s) - f(s') \right| 
\leq \Lip_{\CC}^{p}(f) \cdot \dist_{\CC}^{p}(s,s') 
\leq \dist_{\CC}^{p}(s,s').
\end{equation*}
We deduce that $\sup \left\{ |f(s)-f(s')|: f \in \dom \nabla_\infty, \Lip_{\CC}^p(f) \leq 1 \right\} \leq \dist_{\CC}^p(s,s')$. Now, we prove the reverse inequality. We fix $s \in G$. We consider the function $h \co G \to \R$, $s' \mapsto \dist_{\CC}^{p}(s,s')$. Since $\dist_{\CC}^p$ is a distance on $G$ we have for any $s'' \in G$
$$
\left|h(s')-h(s'')\right| 
= \left|\dist_{\CC}^{p}(s,s') - \dist_{\CC}^{p}(s,s'')\right| 
\leq \dist_{\CC}^{p}(s',s'').
$$
We infer that that $h$ is $p$-Carnot-Carath\'{e}odory-Lipschitz function (hence its class belongs to $\dom \nabla_\infty$ by Lemma \ref{lemmaLipCC-grad}) with $\Lip_{\CC}^{p}(h) \leq 1$. Since $|h(s)-h(s')| = \dist_{\CC}^{p}(s,s')$ we obtain
\begin{equation*} 
\dist_{\CC}^p(s,s') 
\leq \sup \left\{|f(s)-f(s')| : f \in \dom \nabla_\infty, \Lip_{\CC}^{p}(f) \leq 1 \right\}.
\end{equation*}
For the last assertion, we use a regularization argument. We consider a Dirac net $(\varphi_j)$ of functions of $\C^\infty(G)$ satisfying in particular $\int_G \varphi_j\d\mu_G=1$. For any $j$ we let $f_j \ov{\mathrm{def}}{=} \varphi_j * f$.
Using the left-invariance of the distance $\dist_{\CC}^p$ in the third equality, we have
\begin{align*}
\MoveEqLeft
\Lip_{\CC}^p(f_j)            
\ov{\eqref{Def-constant-Lip}}{=} \sup_{s \not= s'} \left|\frac{f_j(s)-f_j(s')}{\dist_{\CC}^p(s,s')} \right| 
\ov{\eqref{Convolution-formulas}}{=} \sup_{s \not= s'}  \left| \int_G \frac{ f(t^{-1}s)\varphi_j(t)-f(t^{-1}s')\varphi_j(t)}{\dist_{\CC}^p(s,s')} \d\mu_G(t)\right| \\
&=\sup_{s \not= s'}  \left| \int_G \frac{ f(t^{-1}s)-f(t^{-1}s')}{\dist_{\CC}^p(t^{-1}s,t^{-1}s')} \varphi_j(t)\d\mu_G(t)\right| 
\leq \int_G \Lip_{\CC}^p(f) \varphi_j(t) \d\mu_G(t) \\
&=\Lip_{\CC}^p(f)\int_G \varphi_j(t) \d\mu_G(t) 
\leq \Lip_{\CC}^p(f).   
\end{align*} 
Since $f$ is left uniformly continuous, the net $(f_j)$ converges uniformly by \cite[Proposition 2.44 p.~58]{Fol16} to $f$. The conclusion is obvious.
\end{proof}

With the terminology of \cite[Definition 1.8]{Lat21}, we can interpret the end of the following result by saying that $(\C(G),\L^2(G) \oplus_2 \L^2(G,\ell^2_m), \slashed{D}_2)$ is a metric spectral triple.

\begin{thm}
\label{Th-recover-metric}
Let $G$ be a connected unimodular Lie group equipped with a family $(X_1, \ldots, X_m)$ of left-invariant H\"ormander vector fields. Suppose $1<p<\infty$. For any $s,s' \in G$, we have
\begin{align}
\dist_\CC^p(s,s')
&=\sup \left\{ |f(s) - f(s')| : f \in \dom \nabla_\infty, \norm{\nabla f}_{\L^\infty(G,\ell^{p^*}_m)} \leq 1 \right\} \nonumber \\
&=\sup \left\{ |f(s) - f(s')| : f \in \dom \nabla_\infty, \norm{\big[\slashed{D}_{p^*},\pi(f)\big]}_{p^* \to p^*} \leq 1 \right\}. \label{Connes-dist-1}           
\end{align}
Moreover, we can replace by $\dom \nabla_\infty$ by the space $\C_c^\infty(G)$.
 
Finally, if $G$ is in addition compact, letting $\norm{f}_{\slashed{D}_p} \ov{\mathrm{def}}{=} \norm{\big[\slashed{D}_p,\pi(f)\big]}_{p \to p}$ for any $f \in \dom \nabla_\infty$, then the pair $(\C(G),\norm{\cdot}_{\slashed{D}_p})$ is a Leibniz quantum compact metric space.
\end{thm}

\begin{proof}
Combining Lemma \ref{lemmaLipCC-grad} and Lemma \ref{Lemma-dCC}, we obtain the first equality.  The second equality is a consequence of \eqref{norm-commutator}.


Now, we prove the last sentence. By \cite[Proposition 5.11 2.]{ArK22}, note that $\norm{\cdot}_{\slashed{D}_p}$ is a seminorm on $\Lip_{\slashed{D}_p}(\L^\infty(G))$, hence on the subspace $\dom \nabla_\infty$. We put $\dom \norm{\cdot}_{\slashed{D}_p} \ov{\mathrm{def}}{=} \dom \nabla_\infty$ and $A \ov{\mathrm{def}}{=} \C(G)$. We check the properties of Proposition \ref{Prop-carac-quantum2}. Note that with a positive answer to the question raised in Remark \ref{Intrigating-rem}, we could use \cite[Proposition 5.11 3. and Remark 5.7]{ArK22} for some assertions.

1. The domain $\dom \norm{\cdot}_{\slashed{D}_p} = \dom \nabla_\infty$ is clearly closed under $f \mapsto \ovl{f}$.

2. Let $f \in \dom \nabla_\infty$ with $\nabla_\infty f=0$. For any $s \in G$, we have
\begin{align*}
\MoveEqLeft
\norm{(\Id-\lambda_s)f}_{\L^p(G)}
\ov{\eqref{Ine-pratique}}{\leq} \dist_\CC^{p^*}(s,e) \bigg(\sum_{k=1}^m \norm{X_kf}_{\L^p(G)}^p \bigg)^{\frac{1}{p}} \\         
&\lesssim_{m,p} \dist_\CC^{p^*}(s,e) \norm{\nabla_\infty f}_{\L^\infty(G,\ell^{p}_m)} 
=0.
\end{align*} 
We deduce that the function $f$ is constant on $G$. The converse is obvious. Hence we have the equality $\big\{ f \in \dom \norm{\cdot}_{\slashed{D}_p} : \: \norm{f}_{\slashed{D}_p} = 0 \big\} = \Cbb 1_{\C(G)}$.

3. By \cite[Proposition 5.11 1.]{ArK22}, for any $f,g \in \Lip_{\slashed{D}_p}(\L^\infty(G))$ we have $fg \in \Lip_{\slashed{D}_p}(\L^\infty(G))$ and
$$
\norm{fg}_{\slashed{D}_p} 
\leq \norm{f}_{\C(G)} \norm{g}_{\slashed{D}_p} + \norm{f}_{\slashed{D}_p} \norm{g}_{\C(G)}.
$$

4. Let $s_0 \in G$ be any point. The Dirac probability measure $\delta_{s_0}$ is supported on the compact $\{s_0\}$. So it is a local state. We consider the subset $\{f \in \dom \norm{\cdot}_{\slashed{D}_p} : \norm{f}_{\slashed{D}_p} \leq 1, f(s_0) = 0\}$. By \eqref{Connes-dist-1} and Lemma \ref{lemmaLipCC-grad}, this subset is equicontinuous. Furthermore, it is pointwise bounded since
$$
|f(s)|
=|f(s)-f(s_0)|
\leq \dist_{\CC}^{p^*}(s_0,s)
\lesssim 1, \quad s \in G
$$ 
by continuity of the function $s \mapsto \dist_{\CC}^{p^*}(s_0,s)$ on the compact $G$. By Arz\`ela-Ascoli theorem, we conclude that it is relatively compact in the space $\C(G)$.

5. Suppose that the net $(f_j)$ converges to $f$ in the space $\L^\infty(G)$ and that $\norm{f_j}_{\slashed{D}_p} \leq 1$, that is $\norm{\nabla f_j}_{\L^\infty(G,\ell^{p^*}_m)} \leq 1$. It converges for the weak* topology. Consequently $(\nabla_\infty f_j)$ is a bounded net of $\L^\infty(G,\ell^{p^*}_m)$. Using Banach-Alaoglu theorem, we can suppose that $\nabla_{\infty}f_j \to g$ for the weak* topology for some function $g \in \L^\infty(G,\ell^{p^*}_m)$. Since the graph of the unbounded operator $\nabla_\infty$ is weak* closed, we conclude that $f$ belongs to the subspace $\dom \nabla_\infty$ and that $g=\nabla_\infty f$. The weak* lower semicontinuity of the norm \cite[Th 2.6.14 p.~227]{Meg98} reveals that 
$$
\norm{\nabla_\infty f}_{\L^\infty(G,\ell^{p^*}_m)} 
\leq \liminf_j \bnorm{\nabla_\infty f_j}_{\L^\infty(G,\ell^{p^*}_m)} 
\leq 1.
$$
\end{proof}

\begin{remark} \normalfont
if $G$ is a unimodular Lie group, it seems apparent that the seminorm $\norm{\cdot}_{\slashed{D}_p}$ can be used to define quantum \textit{locally} compact metric spaces in the spirit of the ones of Section \ref{Sec-locally}. The proof is left to the reader as an exercise.
\end{remark}


We finish the paper by connecting our setting to the vast topic of Dirichlet forms. We refer to the books \cite{BoH91}, \cite{FOT11} and \cite{MaR92} for more information on Dirichlet forms and also to the papers \cite{BeF18}, \cite{Cip16}, \cite{KoY12}, \cite{Stu94} and \cite{Stu95} which are connected to our setting. Let $\Omega$ be a connected second countable Hausdorff locally compact space and $\mu$ be a positive Radon measure with support $\Omega$. We denote by $\cal{M}(\Omega)$ the collection of all signed Radon measures on $\Omega$.

Recall that a Dirichlet form $\cal{E}$ on $\L^2(\Omega)$ is a closed positive definite symmetric bilinear form defined on $\dom \cal{E} \times \dom \cal{E}$ where $\dom \cal{E}$ is a dense linear subspace of the Hilbert space $\L^2(\Omega)$. 

Beurling and Deny showed that if $\cal{E}$ has no killing measure and no jumping measure, it can be written as
$$
\cal{E}(f,g)
=\int_\Omega \d\Gamma(f,g), \quad  f,g \in \dom \cal{E}
$$
for a $\cal{M}(\Omega)$-valued positive definite symmetric bilinear form $\Gamma$ defined by the formula
\begin{equation}
\label{Def-Gamma}
\int_\Omega h \d\Gamma(f,g)
\ov{\mathrm{def}}{=} \frac12\big[\cal{E}(f,h g)+\cal{E}(g,h f)-\cal{E}(fg,h)\big]
\end{equation}
for all $f,g \in \dom \cal{E} \cap \L^\infty(\Omega)$ and $h \in \dom \cal{E} \cap \C_c(\Omega)$. The form $\Gamma$ is called the carr\'e du champ associated to $\cal{E}$. The Radon-Nikodym derivative $\frac{\d\Gamma(f,f)}{\d \mu}(x)$ plays (if it exists) the role of the square of the length of the gradient of $f \in \dom \cal{E}$ at $x \in \Omega$. 
 
An intrinsic pseudo-distance on $X$ associated to $\cal{E}$ is defined in \cite[(4.1)]{Stu94} by
\begin{equation}
\label{distance-intrinsic-0}
\dist_{\cal{E}}(x,y)
\ov{\mathrm{def}}{=} \sup\bigg\{|f(x)-f(y)|:\ f \in \dom \cal{E} \cap \C_c(\Omega), \frac{\d\Gamma(f,f)}{\d \mu} \leq  1\bigg\}.
\end{equation}
Here $\frac{\d\Gamma(f,f)}{\d \mu} \leq 1$ means that $\Gamma(f,f)$ is absolutely continuous with respect to $\mu$ and that $\frac{\d\Gamma(f,f) }{\d \mu} \leq 1$ almost everywhere. We warn the reader that there exist several variants of this distance, see \cite{Stu94} and \cite[p.~236]{Stu95}. 

Returning to the setting of Lie groups, we can consider the symmetric bilinear form
\begin{equation}
\label{Form-ultime}
\cal{E}(f,g)
=\int_G \big\langle \nabla f(s), \nabla g(s) \big\rangle_{\ell^2_m} \d \mu_G(s)
\end{equation}
whose domain is the subspace $\dom \cal{E}=\dom X_{1,2} \cap \cdots \cap \dom X_{m,2}=\dom \nabla_2$ of the Hilbert space $\L^2(G)$. This subspace is considered in the paper \cite{BEM13} and the book \cite{DtER03} and denoted respectively $\W_{1,2}'(G)$ and $\L_{2,1}'(G)$ (and equipped with a suitable norm). A simple computation for any $f,g \in \dom \cal{E} \cap \L^\infty(G)$ and any $h \in \dom \cal{E} \cap \C_c(G)$ gives
\begin{align*}
\MoveEqLeft
\frac12\big[\cal{E}(f,h g)+\cal{E}(g,h f)-\cal{E}(uv,h)\big]  \\          
&\ov{\eqref{Form-ultime}}{=} \frac12\int_G \Big[\big\langle \nabla f(s), \nabla(h g)(s) \big\rangle
+ \big\langle \nabla g(s), \nabla (hf)(s) \big\rangle 
- \big\langle \nabla (fg)(s), \nabla h(s) \big\rangle \Big] \d \mu_G(s) \\
&\ov{\eqref{Leib-gradient}}{=} \frac12\int_G \Big[h(s)\big\langle \nabla f(s), \nabla g(s) \big\rangle+g(s)\langle \nabla f(s), \nabla h(s) \big\rangle
+h(s)\big\langle \nabla g(s), \nabla f(s) \big\rangle \\
&+f(s)\big\langle \nabla g(s), \nabla h(s)\big\rangle-f(s)\big\langle \nabla g(s), \nabla h(s) \big\rangle-g(s)\big\langle \nabla f(s), \nabla h(s) \big\rangle\Big] \d \mu_G(s) \\
&=\int_G h(s)\big\langle \nabla f(s), \nabla g(s) \big\rangle \d \mu_G(s).
\end{align*} 
By \eqref{Def-Gamma}, we conclude (with no surpise) that $\frac{\d\Gamma(f,g)}{\d \mu_G}(s)=\big\langle \nabla f(s), \nabla g(s) \big\rangle$ almost everywhere on $G$. In particular, we have the equality $\frac{\d\Gamma(f,f)}{\d \mu_G}(s)=\norm{\nabla f(s)}_{\ell^2_m}^2$ almost everywhere. In this case, the intrinsic pseudo-distance \eqref{distance-intrinsic-0} is given by 
\begin{align*}
\MoveEqLeft
\dist_{\cal{E}}(s,s')
\ov{\eqref{distance-intrinsic-0}}{=} \sup\Big\{|f(s)-f(s')|:\ f \in \W_{1,2}'(G) \cap \C_c(G), \norm{\nabla f(s)}_{\ell^2_m} \leq  1 \text{ a.e.} \Big\} \\
&=\sup \left\{ |f(s) - f(s')| : f \in \W_{1,2}'(G) \cap \C_c(G), \norm{\nabla f}_{\L^\infty(G,\ell^{2}_m)} \leq 1 \right\}
\end{align*}
where $s,s' \in G$. Using an approximation procedure similar to the one of the proof of Lemma \ref{Lemma-dCC} left to the reader, we could conclude that
\begin{align*}
\dist_{\cal{E}}(s,s')
&=\sup \left\{ |f(s) - f(s')| : f \in \C_c^\infty(G), \norm{\nabla f}_{\L^\infty(G,\ell^{2}_m)} \leq 1 \right\}
\ov{\eqref{dual-carthodory}}{=} \dist_\CC(s,s'),
\end{align*}
i.e. we obtain the Carnot-Carath\'eodory distance. 

\begin{remark} \normalfont
It is possible  that the result \cite[Corollay 2.1]{KoY12} can be used to recover a part of the case $p=2$ of Lemma \ref{lemmaLipCC-grad} with a very different argument.
\end{remark}

\section{Some open problems on functional calculus}
\label{Sec-open-problems}

%

Let $G$ be a connected Lie group of polynomial growth (hence unimodular) equipped with a family $(X_1, \ldots, X_m)$ of left-invariant H\"ormander vector fields and consider a Haar measure $\mu_G$ on $G$. We refer to \cite{Haa06} and \cite{HvNVW18} for more information on functional calculus. Here we use the bisector $\Sigma_\theta^\pm \ov{\mathrm{def}}{=} \Sigma_\theta \cup (-\Sigma_\theta)$ where $\Sigma_\theta^+ \ov{\mathrm{def}}{=} \big\{ z \in \Cbb \backslash \{ 0 \} : \: | \arg z | < \theta \big\}$ for any angle $\theta \in (0,\frac{\pi}{2})$. We will explain why the following conjecture is very natural. 

\begin{conj}
\label{Conj1}
Suppose $1<p<\infty$ with $p\not=2$. The unbounded operator $\slashed{D}_{p}$ is bisectorial and admits a bounded $\H^\infty(\Sigma_\theta^\pm)$ functional calculus on a bisector $\Sigma_\theta^\pm$ for some $0< \theta <\frac{\pi}{2}$ on the Banach space $\L^p(G) \oplus_p \L^p(G,\ell^p_m)$.
\end{conj}

The case $p=2$ is of course obvious since $\slashed{D}_2$ is selfadjoint. The boundedness of the $\H^\infty(\Sigma_\theta^\pm)$ functional calculus of the unbounded operator $\slashed{D}_{p}$ implies the boundedness of the Riesz transforms and this result may be thought of as a strengthening of the equivalence \eqref{Riesz1}. Indeed, consider the function $\sgn \in \H^\infty(\Sigma_\theta^\pm)$ defined by $\sgn(z) \ov{\mathrm{def}}{=} 1_{\Sigma_\theta^+}(z)-1_{\Sigma_\theta^-}(z)$. If the operator $\slashed{D}_p$ has a bounded $\H^\infty(\Sigma^\pm_\theta)$ functional calculus on $\L^p(G) \oplus_p \L^p(G,\ell^p_m)$, the operator $\sgn(\slashed{D}_p)$ is bounded. Moreover, we have
\begin{equation}
\label{lien-sign-abs-Fourier}
|\slashed{D}_p|=\sgn(\slashed{D}_p) \slashed{D}_p
\quad \text{and} \quad
\slashed{D}_p
=\sgn(\slashed{D}_p)|\slashed{D}_p|.
\end{equation}
For any element $\xi$ of the space $\dom \slashed{D}_p = \dom |\slashed{D}_p|$, we deduce that
\begin{align*}
\MoveEqLeft
\bnorm{\slashed{D}_p(\xi)}_{\L^p(G) \oplus_p \L^p(G,\ell^p_m)} 
\ov{\eqref{lien-sign-abs-Fourier}}{=} \bnorm{\sgn(\slashed{D}_p) |\slashed{D}_p|(\xi)}_{\L^p(G) \oplus_p \L^p(G,\ell^p_m)} 
\lesssim_p \bnorm{|\slashed{D}_p|(\xi)}_{\L^p(G) \oplus_p \L^p(G,\ell^p_m)}
\end{align*}
and similarly
\begin{align*}
\MoveEqLeft
\bnorm{|\slashed{D}_{p}|(\xi)}_{\L^p(G) \oplus_p \L^p(G,\ell^p_m)} 
\ov{\eqref{lien-sign-abs-Fourier}}{=} \bnorm{\sgn(\slashed{D}_p) \slashed{D}_p(\xi)}_{\L^p(G) \oplus_p \L^p(G,\ell^p_m)} 
\lesssim_p \bnorm{\slashed{D}_p(\xi)}_{\L^p(G) \oplus_p \L^p(G,\ell^p_m)}.            
\end{align*}
Recall that on $\L^p(G) \oplus_p \L^p(G,\ell^p_m)$, we have
\begin{equation*}
|\slashed{D}_p|
\ov{\eqref{carre-de-D}}{=}\begin{bmatrix}
\Delta_p^{\frac{1}{2}} & 0 \\ 
0 & *
\end{bmatrix}.
\end{equation*}
Using \eqref{Hodge-Dirac-I} and by restricting to elements $\xi$ of the form $(f,0)$ with $f \in \dom \Delta_p^{\frac{1}{2}}$, we obtain the desired equivalence \eqref{Riesz1}.

\begin{remark} \normalfont
With a positive answer to Conjecture \ref{Conj1}, it is not difficult to show in the case where $G$ is compact that the triples $\big(\C(G),\L^p(G) \oplus_p \L^p(G,\ell^p_m), \slashed{D}_p\big)$ gives new examples of compact Banach spectral triples in the sense of \cite[Definition 5.10]{ArK22}. 
\end{remark}

We finish with an other related conjecture.

\begin{conj}
Suppose $1<p<\infty$. If $Y$ is a $\mathrm{UMD}$ Banach space, the unbounded operator $\Delta_p \ot \Id_Y$ is sectorial and admits a bounded $\H^\infty(\Sigma_\theta)$ functional calculus with $0 < \theta <\frac{\pi}{2}$ on the Bochner space $\L^p(G,Y)$.
\end{conj}

This is true for the classical Laplacian on $\R^d$ by \cite[Theorem 10.2.25 p.~391]{HvNVW18}. The scalar case $Y=\Cbb$ seems true by \cite[Theorem 3.4]{DuR96}. The very interesting case where $Y=S^p$ is a Schatten class could have applications in quantum information theory. It is apparent that \cite{Ale94} is related to this problem.

\vspace{0.2cm}

\textbf{Acknowledgement}
The author wishes to thank Tom ter Elst for a short discussion at the very early beginning of this work. I will also thank Fr\'ed\'eric Latr\'emoli\`ere, Jochen Gl\"uck, Peer Kunstmann, Bruno Iochum, Tommaso Bruno and Marco Peloso for short communications. The author is supported by the grant of the French National Research Agency ANR-18-CE40-0021 (project HASCON).


{\footnotesize

\vspace{0.2cm}

\noindent C\'edric Arhancet\\ 
\noindent 6 rue Didier Daurat, 81000 Albi, France\\
URL: \href{http://sites.google.com/site/cedricarhancet}{https://sites.google.com/site/cedricarhancet}\\
cedric.arhancet@protonmail.com\\

}


\small
\begin{thebibliography}{79}

\bibitem{AbA02}
Y. A. Abramovich and C. D. Aliprantis.
\newblock An invitation to operator theory.
\newblock Graduate Studies in Mathematics, 50. American Mathematical Society, Providence, RI, 2002.





\bibitem{AdH96}
D. R. Adams and L. I. Hedberg.
\newblock Function spaces and potential theory. 
\newblock Grundlehren der mathematischen Wissenschaften, 314. Springer-Verlag, Berlin, 1996. 

\bibitem{Ale92}
G. Alexopoulos.
\newblock An application of homogenization theory to harmonic analysis: Harnack inequalities and Riesz transforms on Lie groups of polynomial growth. 
\newblock Canad. J. Math. 44 (1992), no. 4, 691--727. 

\bibitem{Ale94}
G. Alexopoulos.
\newblock Spectral multipliers on Lie groups of polynomial growth. 
\newblock Proc. Amer. Math. Soc. 120 (1994), no. 3, 973--979. 
















\bibitem{AtE17}
W. Arendt, and A. F. M. ter Elst.
\newblock Ultracontractivity and eigenvalues: Weyl's law for the Dirichlet-to-Neumann operator. 
\newblock Integral Equations Operator Theory 88 (2017), no. 1, 65--89. 

\bibitem{Are04}
W. Arendt.
\newblock Semigroups and evolution equations: functional calculus, regularity and kernel estimates. 
\newblock Evolutionary equations. Vol. I, 1--85, Handb. Differ. Equ., North-Holland, Amsterdam, 2004. 








 






\bibitem[Arh22]{Arh22}
C. Arhancet.
\newblock Spectral triples, Coulhon-Varopoulos dimension and heat kernel estimates .
\newblock Preprint, arXiv:2209.12263.

\bibitem{ArK22}
C. Arhancet and C. Kriegler.
\newblock Riesz transforms, Hodge-Dirac operators and functional calculus for multipliers.
\newblock Lecture Notes in Mathematics, 2304. Springer, Cham, 2022.

\bibitem{ArS61}
N. Aronszajn and K. T. Smith.
\newblock Theory of Bessel potentials. I. 
\newblock Ann. Inst. Fourier (Grenoble) 11 (1961), 385--475. 



















 



\bibitem{BEM13}
L. Bandara, A. F. M. ter Elst and A. McIntosh.
\newblock Square roots of perturbed subelliptic operators on Lie groups. 
\newblock Studia Math. 216 (2013), no. 3, 193--217. 















\bibitem{BeF18}
F. Bernicot and D. Frey.
\newblock Sobolev algebras through a `carr\'e du champ' identity. 
\newblock Proc. Edinb. Math. Soc. (2) 61 (2018), no. 4, 1041--1054. 











\bibitem{BLM04}
D. Blecher and C. Le Merdy.
\newblock Operator algebras and their modules-an operator space approach.
\newblock London Mathematical Society Monographs. New Series, 30. Oxford Science Publications. The Clarendon Press, Oxford University Press, Oxford, 2004.
























\bibitem{Bon85}
G. Bohnke.
\newblock Alg\`ebres de Sobolev sur certains groupes nilpotents. (French) [Sobolev algebras on some nilpotent groups]. 
\newblock J. Funct. Anal. 63 (1985), no. 3, 322--343.

\bibitem{BoH91}
N. Bouleau and F. Hirsch.
\newblock Dirichlet forms and analysis on Wiener space. 
\newblock De Gruyter Studies in Mathematics, 14. Walter de Gruyter \& Co., Berlin, 1991. 


\bibitem{Bou04}
N. Bourbaki.
\newblock Integration. II. Chapters 7--9. Translated from the 1963 and 1969 French originals by Sterling K. Berberian. Elements of Mathematics (Berlin).
\newblock Springer-Verlag, Berlin, 2004.



\bibitem{BPTV19}
T. Bruno, M. M. Peloso, A. Tabacco and M. Vallarino.
\newblock Sobolev spaces on Lie groups: embedding theorems and algebra properties. 
\newblock J. Funct. Anal. 276 (2019), no. 10, 3014--3050.








\bibitem{Cal61}
A. P. Calderon.
\newblock Lebesgue spaces of differentiable functions and distributions. 
\newblock Proc. Sympos. Pure Math., Vol. IV pp. 33--49 American Mathematical Society, Providence, R.I (1961). 














\bibitem{CKS87}
E. A Carlen, S. Kusuoka, D. W. Stroock.
\newblock Upper bounds for symmetric Markov transition functions.
\newblock Ann. Inst. H. Poincar\'e Probab. Statist. 23 (1987), no. 2, suppl., 245--287. 






 




\bibitem{Cip16}
F. Cipriani.
\newblock Noncommutative potential theory: a survey.
\newblock J. Geom. Phys. 105 (2016), 25--59.



\bibitem{CGIS14}
F. Cipriani, D. Guido, T. Isola and J.-L. Sauvageot.
\newblock Spectral triples for the Sierpinski gasket.
\newblock J. Funct. Anal. 266 (2014), no. 8, 4809--4869.





 











\bibitem{Con89}
A. Connes.
\newblock Compact metric spaces, Fredholm modules, and hyperfiniteness.
\newblock Ergodic Theory Dynam. Systems 9 (1989), no. 2, 207--220.

\bibitem{Con94}
A. Connes.
\newblock Noncommutative geometry.
\newblock Academic Press, Inc., San Diego, CA, 1994.




\bibitem{CoM08}
A. Connes and M. Marcolli.
\newblock A walk in the noncommutative garden.
\newblock An invitation to noncommutative geometry, 1--128, World Sci. Publ., Hackensack, NJ, 2008.





\bibitem{Cou90}
T. Coulhon.
\newblock Dimension \`a l'infini d'un semi-groupe analytique. (French) [Dimension at infinity of an analytic semigroup]. 
\newblock Bull. Sci. Math. 114 (1990), no. 4, 485--500. 


\bibitem{CoS91}
T. Coulhon and L. Saloff-Coste.
\newblock Semi-groupes d'op\'erateurs et espaces fonctionnels sur les groupes de Lie. (French) [Operator semigroups and function spaces on Lie groups]. 
\newblock J. Approx. Theory 65 (1991), no. 2, 176--199. 

\bibitem{CRT01}
T. Coulhon, E. Russ and V. Tardivel-Nachef.
\newblock Sobolev algebras on Lie groups and Riemannian manifolds. 
\newblock Amer. J. Math. 123 (2001), no. 2, 283--342.

\bibitem{CoM03}
T. Coulhon and S. Meda.
\newblock Subexponential ultracontractivity and $L^p$-$L^q$ functional calculus. 
\newblock Math. Z. 244 (2003), no. 2, 291--308.

\bibitem{CPR11}
A. L. Carey, J. Phillips, A. Rennie.
\newblock Spectral triples: examples and index theory.
\newblock Noncommutative geometry and physics: renormalisation, motives, index theory, 175--265, ESI Lect. Math. Phys., Eur. Math. Soc., Z\"urich, 2011. 

\bibitem{CoM93}
M. Cowling and S. Meda.
\newblock Harmonic analysis and ultracontractivity. 
\newblock Trans. Amer. Math. Soc. 340 (1993), no. 2, 733--752.



\bibitem{CwK95}
M. Cwikel and N. Kalton.
\newblock Interpolation of compact operators by the methods of Calderon and Gustavsson-Peetre. 
\newblock Proc. Edinburgh Math. Soc. 38 (1995).



\bibitem{Dav89}
E. B. Davies.
\newblock Heat kernels and spectral theory. 
\newblock Cambridge Tracts in Mathematics, 92. Cambridge University Press, Cambridge, 1989. 














 

  






\bibitem{Dieu77}
J. Dieudonn\'e.
\newblock \'El\'ements d'analyse. Tome IV. Chapitres XVIII \`a XX. (French) Nouveau tirage.
\newblock Cahiers Scientifiques, Fasc. 34. Gauthier-Villars, Paris, 1977.

\bibitem{Dieu83}
J. Dieudonn\'e.
\newblock \'El\'ements d'analyse. Tome II: Chapitres XII \`a XV, troisi\`eme \'edition. (French)
\newblock Cahiers Scientifiques, Fasc. XXXI Gauthier-Villars, \'Editeur, Paris, 1983.










\bibitem{DtER03}
N. Dungey, A.F.M. ter Elst and D. Robinson.
\newblock Analysis on Lie groups with polynomial growth.
\newblock Progress in Mathematics, 214. Birkh\"auser Boston, Inc., Boston, MA, 2003. viii+312 pp.


\bibitem{DuR96}
X. T. Duong ad D. W. Robinson.
\newblock Semigroup kernels, Poisson bounds, and holomorphic functional calculus. 
\newblock J. Funct. Anal. 142 (1996), no. 1, 89--128.





 


\bibitem{EcI18}
M. Eckstein and B. Iochum.
\newblock Spectral action in noncommutative geometry.
\newblock SpringerBriefs in Mathematical Physics, 27. Springer, Cham, 2018.








\bibitem{Ege15}
M. Egert.
\newblock On Kato's conjecture and mixed boundary conditions.
\newblock PhD, 2015.




\bibitem{EFHN15}
T. Eisner, B. Farkas, M. Haase and R. Nagel.
\newblock Operator theoretic aspects of ergodic theory. 
\newblock Graduate Texts in Mathematics, 272. Springer, Cham, 2015.



\bibitem{ElR01}
A. F. M. ter Elst and D. W. Robinson.
\newblock Subelliptic operators and Lie groups. 
\newblock Geometric analysis and applications (Canberra, 2000), 67--84, Proc. Centre Math. Appl. Austral. Nat. Univ., 39, Austral. Nat. Univ., Canberra, 2001. 

\bibitem{Eme07}
E. Y. Emel'yanov.
\newblock Non-spectral asymptotic analysis of one-parameter operator semigroups. 
\newblock Operator Theory: Advances and Applications, 173. Birkh\"auser Verlag, Basel, 2007. 

\bibitem{EnN00}
K.-J. Engel and R. Nagel.
\newblock One-parameter semigroups for linear evolution equations.
\newblock Graduate Texts in Mathematics, 194. Springer-Verlag, New York, 2000. 












\bibitem{Fol16}
G. B. Folland.
\newblock A course in abstract harmonic analysis. Second edition.
\newblock Textbooks in Mathematics. CRC Press, Boca Raton, FL, 2016. 

\bibitem{FOT11}
M. Fukushima, Y. Oshima and M. Takeda.
\newblock Dirichlet forms and symmetric Markov processes. Second revised and extended edition. 
\newblock De Gruyter Studies in Mathematics, 19. Walter de Gruyter \& Co., Berlin, 2011. 
















 






\bibitem{GVF01}
J. M. Gracia-Bondia, J. C. Varilly and H. Figueroa.
\newblock Elements of noncommutative geometry.
\newblock Birkh\"auser Advanced Texts: Basler Lehrb\"ucher. Birkh\"auser Boston, Inc., Boston, MA, 2001.

\bibitem{Gra14}
L. Grafakos.
\newblock Modern Fourier analysis. Third edition.
\newblock Graduate Texts in Mathematics, 250. Springer, New York, 2014.



\bibitem{GrT12}
A. Grigor'yan and A. Telcs.
\newblock Two-sided estimates of heat kernels on metric measure spaces. 
\newblock Ann. Probab. 40 (2012), no. 3, 1212--1284. 

 

\bibitem{GuK96}
A. Gulisashvili and M. A. Kon.
\newblock Exact smoothing properties of Schr\"odinger semigroups. 
\newblock Amer. J. Math. 118 (1996), no. 6, 1215--1248.

 



\bibitem{GPH15}
B. Guo, X. Pu and F. Huang.
\newblock Fractional partial differential equations and their numerical solutions. 
\newblock Originally published by Science Press in 2011. World Scientific Publishing Co. Pte. Ltd., Hackensack, NJ, 2015. 












\bibitem{Haa06}
M. Haase.
\newblock The functional calculus for sectorial operators.
\newblock Operator Theory: Advances and Applications, 169. Birkh\"auser Verlag (2006).










































\bibitem{HvNVW18}
T. Hyt\"onen, J. van Neerven, M. Veraar and L. Weis.
\newblock Analysis in Banach spaces, Volume~II: Probabilistic Methods and Operator Theory. 
\newblock Springer, 2018. 




















\bibitem{JuM10}
M. Junge and T. Mei.
\newblock Noncommutative Riesz transforms--a probabilistic approach.
\newblock Amer. J. Math. 132 (2010), no. 3, 611--680.



































\bibitem{KaR97}
R. V. Kadison and J. R. Ringrose.
\newblock Fundamentals of the theory of operator algebras. Vol. I. Elementary theory. Reprint of the 1983 original.
\newblock Graduate Studies in Mathematics, 15. American Mathematical Society, Providence, RI, 1997. 


\bibitem{KMS03}
N. J. Kalton and S. Montgomery-Smith.
\newblock Interpolation of Banach spaces.
\newblock Handbook of the geometry of Banach spaces, Vol. 2, 1131--1175, North-Holland, Amsterdam, 2003. 




\bibitem{Kat76}
T. Kato.
\newblock Perturbation theory for linear operators. Second edition.
\newblock Grundlehren der Mathematischen Wissenschaften, Band 132. Springer-Verlag, Berlin-New York, 1976.

\bibitem{KaP88}
T. Kato and G. Ponce.
\newblock Commutator estimates and the Euler and Navier-Stokes equations. 
\newblock Comm. Pure Appl. Math. 41 (1988), no. 7, 891--907.

\bibitem{KoY12}
P. Koskela and Y. Zhou.
\newblock Geometry and analysis of Dirichlet forms. 
\newblock Adv. Math. 231 (2012), no. 5, 2755--2801. 















 










\bibitem{KuW04}
P. C. Kunstmann  and L. Weis.
\newblock Maximal $L_p$-regularity for parabolic equations, Fourier multiplier theorems and $H^\infty$-functional calculus.
\newblock pp. 65-311 in Functional analytic methods for evolution equations, Lect. Notes in Math. 1855, Springer, 2004.









\bibitem{Lat13}
F. Latr\'emoli\`ere.
\newblock Quantum locally compact metric spaces.
\newblock J. Funct. Anal. 264 (2013), no. 1, 362--402.



\bibitem{Lat16a}
F. Latr\'emoli\`ere.
\newblock Quantum metric spaces and the Gromov-Hausdorff propinquity.
\newblock Noncommutative geometry and optimal transport, 47--133, Contemp. Math., 676, Amer. Math. Soc., Providence, RI, 2016.






\bibitem{Lat16b}
F. Latr\'emoli\`ere.
\newblock The quantum Gromov-Hausdorff propinquity.
\newblock Trans. Amer. Math. Soc. 368 (2016), no. 1, 365--411.

\bibitem{Lat21}
F. Latr\'emoli\`ere.
\newblock The Gromov-Hausdorff propinquity for metric Spectral Triples.
\newblock Preprint, arxiv:1811.10843.


\bibitem{Li19}
D. Li.
\newblock On Kato-Ponce and fractional Leibniz. 
\newblock Rev. Mat. Iberoam. 35 (2019), no. 1, 23--100. 











 














\bibitem{MaR92}
Z. M. Ma and M. R\"ockner.
\newblock Introduction to the theory of (nonsymmetric) Dirichlet forms.
\newblock Universitext. Springer-Verlag, Berlin, 1992.

\bibitem{Mag92}
Z. Magyar.
\newblock Continuous linear representations. 
\newblock North-Holland Mathematics Studies, 168. North-Holland Publishing Co., Amsterdam, 1992.



\bibitem{MCSA01}
C. Martinez Carracedo and M. Sanz Alix.
\newblock The theory of fractional powers of operators. 
\newblock North-Holland Mathematics Studies, 187. North-Holland Publishing Co., Amsterdam, 2001.






\bibitem{Meg98}
R. E. Megginson.
\newblock An introduction to Banach space theory.
\newblock Graduate Texts in Mathematics, 183. Springer-Verlag, New York, 1998.

\bibitem{MPS20}
A. Monguzzi, M. M. Peloso and M. Salvatori.
\newblock Fractional Laplacian, homogeneous Sobolev spaces and their realizations. 
\newblock Ann. Mat. Pura Appl. (4) 199 (2020), no. 6, 2243--2261. 

















\bibitem{Nag05}
A. Nagel.
\newblock Analysis and Geometry on Carnot-Carath\'eodory Spaces. 

\newblock \href{https://people.math.wisc.edu/~nagel/2005Book.pdf}{https://people.math.wisc.edu/\~{}nagel/2005Book.pdf}, 2005. 





















\bibitem[Nee22]{Nee22}
J. van Neerven.
\newblock Functional analysis. 
\newblock Cambridge Studies in Advanced Mathematics, 201. Cambridge University Press, Cambridge, 2022. 

\bibitem{Ouh05}
E. M. Ouhabaz.
\newblock Analysis of heat equations on domains.
\newblock London Mathematical Society Monographs Series, 31. Princeton University Press, Princeton, NJ, 2005.


\bibitem{OzR05}
N. Ozawa and M. Rieffel.
\newblock Hyperbolic group $C^*$-algebras and free-product $C^*$-algebras as compact quantum metric spaces.
\newblock Canad. J. Math. 57 (2005), no. 5, 1056--1079.










\bibitem{Pav98}
B. Pavlovic.
\newblock Defining metric spaces via operators from unital $C^*$-algebras. 
\newblock Pacific J. Math. 186 (1998), no. 2, 285--313. 

\bibitem{Ped89}
G. K. Pedersen.
\newblock Analysis now.
\newblock Graduate Texts in Mathematics, 118. Springer-Verlag, New York, 1989.

\bibitem{PeV18}
M. M. Peloso and M. Vallarino.
\newblock Sobolev algebras on nonunimodular Lie groups. 
\newblock Calc. Var. Partial Differential Equations 57 (2018), no. 6, Paper No. 150, 34 pp. 





















































\bibitem{Rie98}
M. A. Rieffel.
\newblock Metrics on states from actions of compact groups.
\newblock Doc. Math. 3 (1998), 215--229. 

\bibitem{Rie02}
M. A. Rieffel.
\newblock Group $C^*$-algebras as compact quantum metric spaces.
\newblock Doc. Math. 7 (2002), 605--651.

\bibitem{Rie04}
M. A. Rieffel.
\newblock Compact quantum metric spaces.
\newblock Operator algebras, quantization, and noncommutative geometry, 315--330, Contemp. Math., 365, Amer. Math. Soc., Providence, RI, 2004.









\bibitem{Rob91}
D. W. Robinson.
\newblock Elliptic operators and Lie groups. 
\newblock Oxford Mathematical Monographs. Oxford Science Publications. The Clarendon Press, Oxford University Press, New York, 1991. 
























\bibitem{Sch20}
K. Schm{\"u}dgen.
\newblock An invitation to unbounded representations of $*$-algebras on Hilbert space.
\newblock Graduate Texts in Mathematics, 285. Springer, Cham, 2020.




















\bibitem{Ste70}
E. M. Stein.
\newblock Singular integrals and differentiability properties of functions.
\newblock Princeton Mathematical Series, No. 30 Princeton University Press, Princeton, N.J. 1970.





 
\bibitem{Str67}
R. S. Strichartz.
\newblock Multipliers on fractional Sobolev spaces. 
\newblock J. Math. Mech. 16 1967 1031--1060. 

\bibitem{Stu94}
K.-T. Sturm.
\newblock Analysis on local Dirichlet spaces. I. Recurrence, conservativeness and $L^p$-Liouville properties. 
\newblock J. Reine Angew. Math. 456 (1994), 173--196. 

\bibitem{Stu95}
K.-T. Sturm.
\newblock On the geometry defined by Dirichlet forms. 
\newblock Seminar on Stochastic Analysis, Random Fields and Applications (Ascona, 1993), 231--242,
Progr. Probab., 36, Birkh\"auser, Basel, 1995. 

\bibitem{SuZ18}
F. Sukochev and D. Zanin.
\newblock The Connes character formula for locally compact spectral triples. 
\newblock Preprint, arXiv:1803.01551.











\bibitem{Tha92}
B. Thaller.
\newblock The Dirac equation.
\newblock Texts and Monographs in Physics. Springer-Verlag, Berlin, 1992.












\bibitem{VSCC92}
N. Varopoulos, L. Saloff-Coste and T. Coulhon.
\newblock Analysis and geometry on groups. 
\newblock Cambridge Tracts in Mathematics, 100. Cambridge University Press, Cambridge, 1992. 

\bibitem{Var1}
J. C. Varilly.
\newblock Dirac operators and
spectral geometry.
\newblock Lecture notes on noncommutative geometry and quantum groups edited by P. M. Hajac. \href{https://www.mimuw.edu.pl/~pwit/toknotes/}{https://www.mimuw.edu.pl/\~{}pwit/toknotes/}






\bibitem{Vil09}
C. Villani.
\newblock Optimal transport. Old and new. 
\newblock Grundlehren der Mathematischen Wissenschaften [Fundamental Principles of Mathematical Sciences], 338. Springer-Verlag, Berlin, 2009. 



















 
\bibitem{Xio17}
X. Xiong.
\newblock Noncommutative harmonic analysis on semigroups and ultracontractivity. 
\newblock Indiana Univ. Math. J. 66 (2017), no. 6, 1921--1947.
















\end{thebibliography}
\end{document}